\documentclass[reqno]{amsart}%
\usepackage{amssymb}%
\usepackage{graphicx}%
\usepackage{tikz}%
\newtheorem{defn}{Definition}%
\newtheorem{prop}{Proposition}%
\newtheorem{lem}{Lemma}%
\newtheorem{thm}{Theorem}%
\newtheorem{cor}{Corollary}%
\def \blank{\phantom{o}}%
\def \con#1{\setbox13\hbox{$#1$}\ifdim%
\wd13<1em\breve{#1}\else{(#1)}\breve{\ }\fi}%
\def \conv#1{\setbox13\hbox{$#1$}\ifdim%
\wd13<1.5em{#1}^{-1}\else{\(#1\)}^{-1}\fi}%
\def \cona{\con\a}%
\def \conb{\con\b}%
\def \halfthinspace{\relax\ifmmode\mskip.5%
\thinmuskip\relax\else\kern.8888em\fi}%
\let \hts=\halfthinspace%
\def \rp{{\hts;\hts}}%
\def \id{{1\kern-.08em\raise1.3ex\hbox{\rm,}\kern.08em}}%
\def \di{{0\kern-.04em\raise1.3ex\hbox{\rm,}\kern.04em}}%
\def \min#1{\overline{#1}}%
\def \Sb#1{\mathfrak{Sb}\hts\(#1\)}%
\def \pow#1{\wp\hts\(#1\)}%
\def \fk#1#2{#1\mathop\nabla#2}%
\def \ot#1#2{#1\otimes#2}%
\def \fkc#1#2{#1\mathop{\con\nabla}#2}%
\def \({\left(}%
\def \){\right)}%
\def \<{\left<}%
\def \>{\right>}%
\def \gc#1{\mathfrak{#1}}%
\def \Re#1{\mathfrak{Re}({#1})}%
\let \bp=\cdot%
\def \a{a}%
\def \b{b}%
\def \c{c}%
\def \d{d}%
\def \e{e}%
\def \f{f}%
\def \g{g}%
\def \h{h}%
\def \i{i}%
\def \j{j}%
\def \k{k}%
\def \m{m}%
\def \n{n}%
\def \p{p}%
\def \q{q}%
\def \r{r}%
\def \s{s}%
\def \t{t}%
\def \u{u}%
\def \v{v}%
\def \w{w}%
\def \x{x}%
\def \y{y}%
\def \z{z}%
\def \A{A}%
\def \AA{\mathfrak{A}}%
\def \B{B}%
\def \BB{\mathfrak{B}}%
\def \C{C}%
\def \D{D}%
\def \E{E}%
\def \EE{\mathcal{S}}%
\def \F{\mathcal{F}}%
\def \II{\mathcal{I}}%
\def \J{J}%
\def \K{K}%
\def \L{L}%
\def \M{\mathcal{M}}%
\def \P{P}%
\def \LL{\mathcal{L}}%
\def \R{R}%
\let \SS=\S%
\def \S{S}%
\def \T{\mathcal{T}}%
\def \U{U}%
\def \UU{\mathcal{U}}%
\def \UUU{\mathfrak{U}}%
\def \VV{\mathcal{V}}
\def \X{X}%
\def \Y{Y}%
\def\0{\mathbf0}%
\def\1{\mathbf1}%
\def\2{\mathbf2}%
\def\3{\mathbf3}%
\def \Fn{\mathsf{Fn}\,}%
\def \Pm{\mathsf{Pm}\,}%
\def \eu{\widehat f}%
\def \ev{\widehat g}%
\def \hx{h^X}%
\def \ie{{\it i.e.}}%
\def \eg{{\it e.g.}}%
\def \etc{{\it etc}}%
\def \RA{\mathsf{RA}}%
\def \JT{\mathsf{JT}}%
\def \JA{\mathsf{JA}}%
\def \RRA{\mathsf{RRA}}%
\def \QRA{\mathsf{QRA}}%
\def \mand{\mathrel{\text{\,and\,}}}%
\def \mor{\mathrel{\text{\,or\,}}}%
\def \KAB{\mathbf{K}_{AB}}%
\def\rrr#1{\setbox13\hbox{$\;#1\;$}{\,\overset{\textstyle#1}%
    {\vrule height2.6pt width\wd13 depth-1.6pt}%
    \kern-3pt\scriptstyle\blacktriangleright\,}}%

\title[Relation algebras containing Thompson groups] {Relation
  algebras containing Thompson groups} \author{Roger D. Maddux}
\address{Department of Mathematics \\ 396 Carver Hall \\ Iowa State
  University \\ Ames, Iowa 50011-2066} \email{maddux@iastate.edu}
\date{\today} \keywords{relation algebras, Q-relation algebras,
  Thompson groups, J\'onsson-Tarski algebras, tabular relation
  algebras, J-algebras, fork algebras, representable relation
  algebras} \subjclass{03G15, 20E08, 20F05, 20F10}

\begin{document}
\allowdisplaybreaks
\begin{abstract}
  The connections between Tarski's relation algebras and Thompson's
  groups $\F$, $\T$, $\VV$, and his monoid $\M$ are reviewed here,
  along with J\'onsson-Tarski algebras, fork algebras, true pairing
  algebras, and tabular relation algebras. All of these algebras are
  related to the finitization problem and to Tarski's formalization of
  set theory without variables.

  Most of the technical details occur in the variety of J-algebras,
  which is obtained from relation algebras by omitting union and
  complementation and adopting a set of axioms created by J\'onsson.
  Every relation algebra or J-algebra that contains a pair of
  conjugated quasiprojections satisfying the Domain and Unicity
  conditions, such as those that arise from J\'onsson-Tarski algebras
  or fork algebras, will also contain homomorphic images of $\F$,
  $\T$, $\VV$, and $\M$.

  The representability of tabular relation algebras is extended here
  to J-algebras, using a notion of tabularity that is equivalent among
  relation algebras to the original definition.
\end{abstract}
\maketitle
\begin{quote}\small
  Dedicated to the memory of Alfred Tarski (1901--1983) and his
  students Bjarni J\'onsson (1920--2016) and George McNulty
  (1945--2023).
\end{quote}

\tableofcontents

\section*{{\bf Part I.}}

\section{Introduction}
\label{s0}
Suppose a relation algebra $\AA$ has elements $\a$ and $\b$ such that
\begin{align}
  \label{Q}
  \cona\rp\a\leq\id, \quad\conb\rp\b\leq\id, \quad1=\cona\rp\b.
\end{align}
Tarski called such elements a pair of conjugated quasiprojections.
Suppose that $\a$ and $\b$ also satisfy the Domain Condition
\begin{align}
  \label{D}
  1=\a\rp1=\b\rp1,
\end{align}
and the Unicity Condition
\begin{align}
  \label{U}
  \a\rp\cona\bp\b\rp\conb\leq\id.
\end{align}
Then $\AA$ contains homomorphic images of Thompson's monoid $\M$ and
Thompson's groups $\F$, $\T$, and $\VV$.  In fact, if $\AA$ is freely
generated by $\a$ and $\b$ subject to the relations \eqref{Q},
\eqref{D}, and \eqref{U} then it actually contains \emph{copies} of
$\M$, $\F$, $\T$, and $\VV$.  This result is easy to prove by applying
a theorem due to Tarski, who defined a Q-relation algebra as a
relation algebra that has a pair of conjugated quasiprojections and
proved that every Q-relation algebra is representable, \ie, every
Q-relation algebra is isomorphic to a relation algebra in which the
operations are union, intersection, complementation, composition of
binary relations, and converse, with the empty relation, a universal
relation, and an identity relation as constants.

The result stated above is easy to prove because of Tarski's theorem,
but it can be extended to a much wider class of algebras lacking the
abstract algebraic operations that correspond to union and
complementation.  These algebras are called J-algebras and their
axioms are due to J\'onsson. The main result of this paper is that
every J-algebra with elements $\a$ and $\b$ satisfying \eqref{Q},
\eqref{D}, and \eqref{U} contains homomorphic images of $\M$, $\F$,
$\T$ and $\VV$.

Tarski's theorem that Q-relation algebras are representable has a
substantial generalization to tabular relation algebras, one in which
the unit element is the join of elements of the form $\cona\rp\b$,
where $\cona\rp\a\leq\id$ and $\conb\rp\b\leq\id$.  The use of join
means the definition of tabularity does not apply directly to
J-algebras, but a suitably modified definition of tabularity (that is
equivalent to the original definition when applied to relation
algebras) does turn out to be sufficient for representability. This
fact is the second major result of the paper.

The paper is divided into Part I (\SS\ref{s0}--\SS\ref{s14}), Part II
(\SS\ref{s16}--\SS\ref{s24}), and Part III
(\SS\ref{s25}--\SS\ref{s28}). Part I contains an exposition of all the
topics discussed in this introduction. The result about the occurrence
of the Thompson monoid and groups in J-algebras is stated in
\SS\ref{s14} at the end of Part I and is proved in Part II. The
representability of tabular J-algebras is proved in Part III.

Here are details about the contents of each section.  \SS\ref{s1}
tells how Thompson's monoid and groups grew out of his work on the
`finitization problem' that was originally posed by J.\ Donald Monk.
The concepts needed from universal algebra and group theory are listed
in \SS\ref{s2}.  Relation algebras are defined by their axioms in
\SS\ref{s3} and J-algebras are defined by their axioms in \SS\ref{s5}.

Part II may be inserted between \SS\ref{s5} and \SS\ref{s4}.
Everything in Part II, specifically Definitions \ref{d13}--\ref{FnPm}
and Props.\ \ref{p1}--\ref{presM}, applies to an arbitrary J-algebra.
All that is needed for Part II are the axioms
\eqref{bassoc}--\eqref{norm} in \SS\ref{s5}.  Consequently, everything
in Part II can be used in \SS\ref{s4}--\SS\ref{s14} of Part I.  There
is no dependence in Part II on anything in \SS\ref{s4}--\SS\ref{s14}.
This arrangement keeps the sometimes long and complicated derivations
out of the way of the discussion in Part I.

Representability for relation algebras and J-algebras is defined in
\SS\ref{s4}.  The way that monoids and groups can occur in relation
algebras and in J-algebras is discussed in \SS\ref{s6}.  Q-relation
algebras are treated in \SS\ref{s7}, including Tarski's theorem, his
formalization of set theory without variables, his `Main Mapping
Theorem', the `Translation Mappings', and their history.

The definition and representability of tabular relation algebras are
stated in \SS\ref{s8}.  Qu-algebras are defined in \SS\ref{s9} as
relation algebras with elements $\a,\b$ that satisfy \eqref{Q},
\eqref{D}, and \eqref{U}.  Fork algebras are presented in
\SS\ref{s10}, including their axioms and their connection with
Qu-algebras.  The pairing identity is introduced in \SS\ref{s10} as an
axiom for fork algebras. Its significance and history is treated in
\SS\ref{s10a}, including its use by Tarski and Givant in their proof
of the Main Mapping Theorem, a key ingredient in Tarski's original
proof that Q-relation algebras are representable. Tarski's proof
employed metamathematical methods and was `more complicated than one
would expect' (Tarski's words).  The historically first purely
relation algebraic proof of Tarski's theorem is outlined in
\SS\ref{origQRAth}.

\SS\ref{s18} presents an abstract algebraic formulation of the concept
of direct product, Gunther Schmidt's conjecture, its resolution by an
example, extensions of the pairing identity, and the relation of the
pairing identity to certain identities that are true in representable
relation algebras but can't be deduced from the axioms of relation
algebras.

In \SS\ref{s11}, J\'onsson-Tarski algebras are introduced and are
shown to be essentially equivalent to certain Qu-algebras and to
bijections between a set and its Cartesian square.  \SS\ref{s12} shows
how to construct functions on the universes of J\'onsson-Tarski
algebras that generate Thompson's monoid and groups. This provides a
precise link between Thompson's original parenthetical notation for
operations on trees and the algebraic notation for elements in a
J-algebra. This link is used in \SS\ref{s12a} to create generators for
$\M$, $\F$, $\T$, and $\VV$.  Thompson's groups and monoid are defined
by their presentations in \SS\ref{s13} and \SS\ref{s13a}.  The main
result concerning the occurrence of $\F$, $\T$, $\VV$, and $\M$ in
J-algebras is stated in \SS\ref{s14}. It is proved by appeals to
various propositions in Part II.

In Part II, \SS\ref{s16} deals with the most elementary consequences
of the J-algebra axioms and specifies notational conventions for
making or omitting references to those consequences.  \SS\ref{s17}
presents standard properties of functional and permutational elements
familiar from the theory of relation algebras. This section also
presents (the fairly complicated) equational derivations of the
pairing identity and its variations from \eqref{Q}.  Two elements are
chosen and fixed as parameters in \SS\ref{s19} in order to define
three binary operations that are important in the theory of fork
algebras.  Several closure properties and useful lemmas are proved
there.

The ten generators of $\F$, $\T$, $\VV$, and $\M$ are listed in
algebraic notation in \SS\ref{s20} and the single proposition there
tells which ones are functional and which ones are permutational.
There are two propositions in \SS\ref{s22}: first a lemma about the
generator $\A$ and then a proposition that provides proofs of the two
relations in the presentation of $\F$.  \SS\ref{s23} contains a single
proposition that proves all the relations in the presentation of $\T$.
The relations defining $\VV$ are not proved since they would be very
similar to the ones already presented and would lengthen the paper
without providing much further insight. Their absence provides
exercises for a reader who may be interested in constructing proofs of
the type well illustrated by the previous two propositions.
\SS\ref{s24} shows that $\M$ is generated by two different sets of
four generators and that one of those two sets of generators satisfies
all the relations in the infinite presentation of $\M$.

Part III shows that tabular J-algebras are representable. Tabular
J-algebras and the notion of partial representation are defined in
\SS\ref{s25}. There are two lemmas in \SS\ref{s26} about the extension
of partial representations.  The key lemma in \SS\ref{s27} shows how
to assemble partial representations into a single function for each
proper two-element chain.  The main result is then stated and proved
by constructing a representation that is essentially an ultraproduct
of these functions.  The concluding \SS\ref{s28} considers some
prospects for further work.

\section{The finitization problem}
\label{s1}
The Q-relation algebras of Alfred Tarski and the groups of Richard
J.\ Thompson are very intimately related. The close connections
between these two types of algebras will be explored in this paper.
As far as I know, this connection has not been previously noted. The
closest approach I have seen is Graham Higman's description of one of
Thompson's groups as the automorphism group of a free J\'onsson-Tarski
algebra on one generator.

It is easy to see from the definitions of Q-relation algebra and
J\'onsson-Tarski algebra that these two types of algebras arose from
the same mathematical considerations that occupied Tarski for much of
his professional life, starting with an unpublished monograph begun by
him at the Institute for Advanced Study in the summer of 1942, and
culminating in the 1987 book that emerged from this monograph, \emph{A
Formalization of Set Theory without Variables} \cite{MR920815}.

Depending on the criteria used for counting, the number of papers that
involve Thompson's groups has three or four digits, while the number
that deal with Q-relation algebras has at most two digits. Thompson's
groups have had a significant impact on group theory that is
independent of their origins, but the most that papers in group theory
say about their origin is that they come from Thompson's ``work in
logic'' or his ``study of the $\lambda$-calculus.''  Even McKenzie and
Thompson \cite{MR0396769} say only that ``Thompson discovered the
groups $\mathfrak{C}'$ and $\mathfrak{P}'$ in connection with his
research in logic about 1965.''

In fact, Thompson's research was directed toward solving what is now
known as ``the finitization problem'', that is, to find a ``finitary
algebraic logic'', a finitely based variety of algebras of finite
similarity type that is equivalent to first-order logic. The variety
of $\omega$-dimensional cylindric algebras does the job, but requires
an infinite set of unary operations called cylindrifications. Each
cylindrification mimics the action of existentially quantifying a
formula with respect to some variable. Furthermore, the representable
algebras in varieties that capture even a small fraction of
first-order logic are not finitely axiomatizable. Monk conjectured
that this might be unavoidable, having shown it for relation algebras
and finite-dimensional cylindric algebras.  To get a negative solution
to Monk's conjecture one must use only finitely many operations and
finitely many equational axioms to show that all such algebras are
representable.

In 1975, when Richard Thompson and I were both graduate students at
U.\ C.\ Berkeley, Thompson gave a report to a seminar on algebraic
logic organized by me and another graduate student of Tarski, Ulf
Wostner. At that seminar, Thompson made a proposal for a finitary
algebraic logic.  This would have been a negative solution to the
conjecture by Monk.  Thompson talked about his finite presentation of
a certain semigroup of operators that act on binary trees, or,
equivalently, on sequences of 0's and 1's. By that time I had worked
for almost a year as Tarski's research assistant on the manuscript for
\cite{MR920815} and was familiar with Q-relation algebras.  The
obvious connection between Thompson's semigroup and Q-relation
algebras led me to propose ``true pairing algebras'' as a finitary
algebraic logic \cite{MR1020505, MR1270402, MR3400602, MR1776151}.
Neither Thompson's semigroup nor true pairing algebras solves the
finitization problem because later, more precise, formulations of the
problem include the requirement that the algebraic operations be
``logical'' in Tarski's sense.  See Tarski-Givant
\cite[3.5(i)(ii)]{MR920815} for Tarski's definition of logical.  One
of the problems Tarski and Givant raise in \cite[\SS3.5]{MR920815} is
a precise version of Monk's problem.  The finitization problem was
also formulated by Henkin and Monk as \cite[Problem 1]{MR0376346}.
For more information and related work on the finitization problem see
\cite{MR1020505, MR1270402, MR3400602, MR1776151, MR1187451,
  MR1636979, MR1488292, MR1011183, MR0172797, MR280345, MR289274,
  MR1465620, MR1153440, MR3433641, MR1446276, Th87, Th87a, MR1442088,
  MR1131030}.

Monk's conjecture comes from \cite[p.\,20]{MR289274}, where he wrote,
\begin{quote}
  The results obtained in this paper are, however, quite analogous to
  those obtained in \cite{MR0256861} for cylindric algebras. They
  contribute to the conjecture that no equational form of first-order
  logic is finitely axiomatizable---more precisely, with respect to
  any conception $\LL$ of a set algebra (corresponding to the notion
  of satisfaction), and any choice of basic operations, the
  corresponding class $\LL'$ is not finitely axiomatizable. It appears
  difficult to give this conjecture a very precise form, since there
  is a wide latitude of choice with regard to the fundamental
  operations as well as the kinds of sequences considered in the
  satisfaction relation. The conjecture has been verified for most
  brands of algebraic logic known to the author.
\end{quote}
Decades later, Thompson gave two talks on the background for the
Thompson groups. The first was at AIM in Palo Alto on January 10,
2004, the second at Luminy on June 2, 2008, and both were attended by
Matthew Brin. Brin took notes and wrote them up in \emph{The Thompson
monoid is finitely presented}, Feb.\ 7, 2021, 13 pp.  According to
Brin's notes, in the first talk Thompson introduced his monoid $\M$
with which he had started his researches and which contains the group
later known as $\VV$ as the group of invertible elements.  Thompson
said he was interested in ``finding an algebraic system that fit the
predicate calculus as well as Boolean algebra fits the propositional
calculus.''  This was a reference to the finitization problem. He gave
background on some systems that were attempts in this direction.  He
said he started with the $\lambda$-calculus of Alonzo Church, but it
was the combinatory logic of Haskell Curry that he ended up using and
is the language that he gave his talks in.  The monoid $\M$ can be
represented as a set of endomorphisms of the Cantor set.  This
representation as endomorphisms of the Cantor set has its own
advantages and is more familiar to those already acquainted with the
Thompson groups. The connection with relation algebras is best made
through manipulations on parenthesized expressions. This connection
will presented in \SS\ref{s12}.

\section{Universal algebra}
\label{s2}
Assumed to be known are the universal algebraic concepts of operation
on a set, algebra, universe of an algebra, subalgebra, direct product,
subdirect product, subdirect irreducibility, function, permutation,
injection, surjection, bijection, homomorphism, isomorphism, equation,
and satisfaction of an equation in an algebra.  An algebra $\AA$ is
{\bf simple} if it has at least two elements and every homomorphic
image of $\AA$ is a one-element algebra or is isomorphic to $\AA$.  A
class $\K$ of algebras is a {\bf variety} if it has an equational
axiomatization. By a theorem of Birkhoff, a class of algebras is a
variety iff every homomorphic image of a subalgebra of a direct
product of algebras in $K$ is again in $K$.  For all of this and much
more see \cite{MR4496007, MR4563240, MR0314620, MR0781929, MR2269199,
  MR0883644, MR3793673}.  We also assume familiarity with basic group
theory, including group presentations.
\section{Relation algebras}
\label{s3}
For this paper, the basic facts about relation algebras presented in
Tarski-Givant \cite[\SS8.2, \SS8.3]{MR920815} are enough. For more,
consult the books by Hirsch and Hodkinson \cite{MR1935083}, Givant
\cite{MR3699802, MR3699801}, and Maddux \cite{MR2269199}. Tarski's ten
axioms \eqref{ra1}--\eqref{ra10} in the list below are designed for a
smaller similarity type that does not include $\bp$, $0$, or
$1$. Tarski preferred introducing $\bp$, $0$, and $1$ by definitions
which, in this paper, are simply the additional axioms
\eqref{ra11}--\eqref{ra13}.
\begin{defn}
  \label{RA}
  Let $\AA=\<A,+,\bp,\min\blank,0,1,\rp,\con\blank,\id\>$, where $\bp$
  and $\rp$ are binary operations on $A$, $\min\blank$ and
  $\con\blank$ are unary operations on $A$, and $0,1,\id\in\A$.  $\AA$
  is a {\bf relation algebra} if it satisfies these axioms:
  \begin{align}
    \label{ra1}
    \x+\y&=\y+\x &&\text{$+$ is commutative} \\
    \label{ra2}
    \x+(\y+\z)&=(\x+\y)+\z &&\text{$+$ is associative} \\
    \label{ra3}
    \min{\min\x+\min\y}+\min{\min\x+\y}&=\x &&\text{the Huntington
      axiom} \\
    \label{ra4} \x\rp(\y\rp\z)&=(\x\rp\y)\rp\z
    &&\text{$\rp$ is associative} \\
    \label{ra5}
    (\x+\y)\rp\z&=\x\rp\z+\y\rp\z &&\text{$\rp$ distributes over $+$}
    \\
    \label{ra6} \x\rp\id&=\x &&\text{the identity law}
    \\
    \label{ra7} \con{\con\x}&=\x &&\text{$\con\blank$ is an
      involution} \\
    \label{ra8} \con{\x+\y}&=\con\x+\con\y
    &&\text{$\con\blank$ distributes over $+$} \\
    \label{ra9}
    \con{\x\rp\y}&=\con\y\rp\con\x &&\text{$\con\blank$ distributes in
      reverse over $\rp$} \\
    \label{ra10}
    \con\x\rp\min{\x\rp\y}+\min\y&=\min\y &&\text{the Tarski/De~Morgan
      axiom} \\
    \label{ra11} \x\bp\y&=\min{\min\x+\min\y}
    &&\text{definition of $\,\bp$} \\
    \label{ra12} 1&=\id+\min\id&&\text{definition of $1$} \\
    \label{ra13}
    0&=\min1&&\text{definition of $0$}
  \end{align}
  $\RA$ is the class of relation algebras.
\end{defn}
Axioms \eqref{ra1}, \eqref{ra2}, and \eqref{ra3} characterize Boolean
algebras as algebras of the form $\<\A,+,\min\blank\,\>$.  They are
due to Huntington \cite{MR1501702, MR1501684, MR1501729}. To develop
the equational theory of Boolean algebras from
\eqref{ra1}--\eqref{ra3} one first states \eqref{ra11}, \eqref{ra12},
and \eqref{ra13} as definitions of $\bp$, $1$, and $0$.  Having the
constant $\id$ with which to define $0$ and $1$ makes it easier to
develop the equational theory. For a complete proof of all the usual
equations true in Boolean algebras for the case with no such constant
see \cite{MR1392462}.

In expressions denoting elements in a relation algebra, unary
operations $\con\blank$ and $\min\blank$ are computed before binary
ones, and among binary operations, the order is first $\rp$, then
$\bp$, and finally, $+$. For example, $\w+\x\bp\y\rp\z =
\w+(\x\bp(\y\rp\z))$.  Association is to the left for repeated uses of
$+$, $\bp$, or $\rp$.  For example, $\x+\y+\z=(\x+\y)+\z$,
$\x\bp\y\bp\z = (\x\bp\y)\bp\z$ and $\x\rp\y\rp\z =
(\x\rp\y)\rp\z$. Repeated relative products are indicated by
exponents, \eg, $\x^2=\x\rp\x$ and $\x^3=\x\rp\x\rp\x$.

If $\AA=\<A,+,\bp,\min\blank,0,1,\rp,\con\blank,\id\>$ is a relation
algebra then $\<A,+,\bp,\min\blank,0,1\>$ is a Boolean algebra called
the {\bf Boolean reduct} of $\AA$. Results and concepts from the
theory of Boolean algebras, when applied to $\AA$, refer to the
Boolean reduct of $\AA$.

\section{J-algebras}
\label{s5}
An algebra whose proper subalgebras have strictly smaller cardinality
than the algebra itself is known in the literature as a ``J\'onsson
algebra''. For that reason, and also for brevity, the algebras
introduced here are simply called ``J-algebras''.  The first ten
axioms occur in the characterization by J\'onsson
\cite[Th.\,1]{MR108459} of algebras isomorphic to sets of binary
relations equipped with the operations of intersection, composition,
and converse. All of the axioms are true in relation algebras.
\begin{defn}
  A {\bf J-algebra} is an algebra
  \begin{align*}
    \AA=\<A,\bp,0,1,\rp,\con\blank,\id\>,
  \end{align*}
  where $\bp$ and $\rp$ are binary operations on $A$, $\con\blank$ is
  a unary operation on $A$, and $0,1,\id\in\A$, satisfying the
  following axioms.
  \begin{align}
    \label{bassoc}
    \x\bp(\y\bp\z)&=(\x\bp\y)\bp\z
    &&\text{\cite[Th.\,1(i)]{MR108459}} \\
    \label{comm}
    \x\bp\y&=\y\bp\x
    &&\text{\cite[Th.\,1(ii)]{MR108459}} \\
    \label{idem}
    \x\bp\x&=\x
    &&\text{\cite[Th.\,1(iii)]{MR108459}} \\
    \label{2}
    \x\rp(\y\rp z)&=(\x\rp\y)\rp z
    &&\text{\cite[Th.\,1(iv)]{MR108459}} \\
    \label{id}
    \x\rp\id&=\x
    &&\text{\cite[Th.\,1(v)]{MR108459}} \\
    \label{mon}
    (\x\bp\y)\rp\z&=(\x\bp\y)\rp\z\bp\y\rp\z
    &&\text{\cite[Th.\,1(vi)]{MR108459}} \\
    \label{6}
    \con{\con\x}&=\x
    &&\text{\cite[Th.\,1(vii)]{MR108459}} \\
    \label{5}
    \con{\x\rp\y}&=\con\y\rp\con\x
    &&\text{\cite[Th.\,1(viii)]{MR108459}}\\
    \label{4}
    \con{\x\bp\y}&=\con\x\bp\con\y
    &&\text{\cite[Th.\,1(ix)]{MR108459}} \\
    \label{rot}
    \x\rp\y\bp\z&=(\z\rp\con\y\bp\x)\rp(\y\bp\con\x\rp\z)\bp\z
    &&\text{\cite[Th.\,1($\Gamma$) for $n=2$]{MR108459}} \\
    \label{zero}
    0\bp\x&=0 &&\text{$0$ is the bottom} \\
    \label{one}
    \x\bp1&=\x &&\text{$1$ is the top} \\
    \label{norm}
    \x\rp0&=0 &&\text{$\rp$ is normal}
  \end{align}
  $\JA$ is the class of J-algebras.
\end{defn}
In Definition \ref{d1} below we define $\x\leq\y$ as $\x\bp\y=\x$ for
J-algebras. This allows some axioms to be rewritten in a more familiar
way. For example, \eqref{one} says $\x\leq1$ and \eqref{zero} is
equivalent by \eqref{comm} to $0\leq\x$.  The following proposition
points out that a J-algebra can be constructed from the algebraic form
of a bounded semilattice, that is, from any idempotent semigroup with
absorbing element 0 and identity element~1.
\begin{prop}
  If $\<A,\bp,0,1\>$ is an algebra satisfying axioms \eqref{bassoc},
  \eqref{comm}, \eqref{idem}, \eqref{zero}, and \eqref{one}, then
  $\<A,\bp,0,1,\bp,\con\blank,1\>$ is a J-algebra where $\con\x=\x$
  for all $\x\in\A$.
\end{prop}
Deleting $+$ and $\min\blank$ from a relation algebra leaves a
J-algebra.
\begin{prop}
  \label{jaisra}
  If $ \<A,+,\bp,\min\blank,0,1,\rp,\con\blank,\id\>$ is a relation
  algebra then $\<A,\bp,0,1,\rp,\con\blank,\id\>$ is a J-algebra.
\end{prop}
\begin{proof}
  It suffices to show {that} the axioms of J-algebras hold in any
  relation algebra. From axioms \eqref{ra1}--\eqref{ra13} we get
  \eqref{bassoc} by \cite[Th.\,3(ix)]{MR1392462}, \eqref{comm} by
  \cite[Th.\,3(viii)]{MR1392462}, \eqref{idem} by
  \cite[Th.\,3(vii)]{MR1392462}, \eqref{2} by \eqref{ra4}, \eqref{id}
  by \eqref{ra6}, \eqref{mon} by \cite[Th.\,253]{MR2269199} applied to
  $\x\bp\y\leq\y$, \eqref{6} by \eqref{ra7}, \eqref{5} by \eqref{ra9},
  \eqref{4} by \cite[Th.\,250]{MR2269199}, \eqref{rot} by
  \cite[Th.\,296, (6.49)]{MR2269199}, \eqref{zero} by
  \cite[Ths.\,3(ii)(iv),\,5(vi)]{MR1392462} plus \eqref{comm},
  \eqref{one} by \cite[Ths.\,3(ii)(v)(vi),\,5(viii)]{MR1392462}, and
  \eqref{norm} by \cite[Th.\,287, (6.31)]{MR2269199}.
\end{proof}

\section{Representability}
\label{s4}
For any set $\E$, $\pow\E=\{\X:\X\subseteq\E\}$ is the {\bf powerset}
of $\E$. The notation $\<\x,\y\>$ is used for ordered pairs. A {\bf
  binary relation} is a set of ordered pairs.
\begin{defn}
  \label{SbE}
  Define operations on binary relations $\R$ and $\S$ as follows.
  \begin{align*}
    \R\cup\S&=\{\<\x,\y\>:\<\x,\y\>\in\R\mor\<\x,\y\>\in\S\}
    &&\text{union}
    \\ \R\cap\S&=\{\<\x,\y\>:\<\x,\y\>\in\R\mand\<\x,\y\>\in\S\}
    &&\text{intersection} \\ \R|\S&=\{\<\x,\y\>:\exists_\z
    (\<\x,\z\>\in\R\mand\<\z,\y\>\in\S)\} &&\text{composition}
    \\ \conv\R&=\{\<\y,\x\>:\<\x,\y\>\in\R\} &&\text{converse}
    \\ Id(\R)&=\{\<\x,\x\>:\<\x,\x\>\in\R\} &&\text{identity part}
  \end{align*}
\end{defn}
A binary relation $\E$ is an {\bf equivalence relation} iff
$\E=\E|\E=\conv\E$.  If $\E$ is an equivalence relation then
$Id(\E)\in\pow\E$ and $\pow\E$ is closed under intersection, union,
composition, and converse.  This observation enables the following
definition.
\begin{defn}
  For any equivalence relation $\E$, the {\bf J-algebra of
    subrelations of $\E$} is
  \begin{align*}
    \<\pow\E,\cap,\emptyset,\E,|,\conv{},Id(\E)\>
  \end{align*}
  and the {\bf relation algebra of subrelations of $\E$} is
  \begin{align*}
    \Sb\E&=\<\pow\E,\cup,\cap,\min\blank,\emptyset,\E,|, \conv{},
    Id(\E)\>,
  \end{align*}
  where, for all $\R\in\pow\E$,
  \begin{align*}
    \min\R&=\{\<\x,\y\>:\<\x,\y\>\in\E\mand\<\x,\y\>\notin\R\}
    &&\text{complement (w.r.t. $\E$)}
  \end{align*}
  For any set $\X$, the {\bf J-algebra of relations on $\X$} is
  \begin{align*}
    \<\pow\E,\cap,\emptyset,\X^2,|,\conv{},Id(\X^2)\>
  \end{align*}
  and the {\bf relation algebra of relations on $\X$} is
  \begin{align*}
    \Re\X&=\Sb{\X^2}.
  \end{align*}
  A J-algebra (or relation algebra) is {\bf representable} if it is
  isomorphic to a subalgebra of the J-algebra (or relation algebra) of
  subrelations of an equivalence relation.  An algebra $\AA$ is a {\bf
    proper relation algebra} if it is a subalgebra of $\Sb\E$ for some
  equivalence relation $\E$.  The {\bf base} of a proper relation
  algebra is the field $\{\x:\<\x,\x\>\in\E\}$ of its equivalence
  relation $\E$.  A {\bf representation} of a relation algebra $\AA$
  is an isomorphism that embeds $\AA$ into a proper relation algebra.
  A {\bf representation over $\X$} is an isomorphism that embeds $\AA$
  into $\Re\X$. $\RRA$ is the class of representable relation algebras
  and ``$\RRA$'' serves as an abbreviation of ``representable relation
  algebra''.
\end{defn}
As one would expect, every representable relation algebras satisfies
the axioms \eqref{ra1}--\eqref{ra10}.  Tarski \cite{MR0066303} proved
that $\RRA$ is a variety but Monk \cite{MR0172797} proved it cannot be
characterized by any finite set of equations.

\section{Monoids and groups in relation algebras}
\label{s6}
Let $\AA$ be a relation algebra or a J-algebra. An element $\x$ in
$\AA$ is {\bf functional} if $\con\x\rp\x\leq\id$ and {\bf
  permutational} if $\con\x\rp\x=\id=\x\rp\con\x$.  $\Fn\AA$ is the
set of functional elements of $\AA$ and $\Pm\AA$ is the set of
permutational elements of $\AA$.  A functional element of $\Re\X$ is a
function that maps a subset of $\X$ to a subset of $\X$, while
permutational elements of $\Re\X$ are permutations of $\X$. It is
proved in Prop.\ \ref{func}(v) that $ \<\Fn\AA,\rp,\id\>$ is a monoid
(a semigroup with an identity element) and
$\<\Pm\AA,\rp,\con\blank,\id\>$ is a group because the required
closure properties and identities hold. We therefore define
$\gc{Fn}(\AA)=\<\Fn\AA,\rp,\id\>$ and
$\gc{Pm}(\AA)=\<\Pm\AA,\rp,\con\blank,\id\>$.  Furthermore, if $\h$ is
a homomorphism from $\AA$ to $\BB$ then the restriction of $\h$ to
$\Fn\AA$ is a homomorphism from $\gc{Fn}(\AA)$ to $\gc{Fn}(\BB)$ and
the restriction of $\h$ to $\Pm\AA$ is a homomorphism from
$\gc{Pm}(\AA)$ to $\gc{Pm}(\BB)$.

\section{Q-relation algebras}
\label{s7}
The primary source for Q-relation algebras is \cite[\SS8.4]{MR920815}.
For a detailed historical survey of the origin of Q-relation algebras
see \cite[Ch.\,1, \SS11]{MR2269199}.  The representability of
Q-relation algebras is proved in \cite[Ch.\,6, \SS53]{MR2269199} and
is mentioned in \cite[p.\,209]{MR1935083} and in
\cite[p.\,301]{MR3699802}.
\begin{defn}[{\cite[8.4(i)(ii)]{MR920815}}]
  \label{d1}
  Two elements $\a,\b$ in a relation algebra or J-algebra are {\bf
    conjugated quasiprojections} if
  \begin{align*}
    \cona\rp\a\leq\id, \quad\conb\rp\b\leq\id, \quad1=\cona\rp\b.
  \end{align*}
  A {\bf Q-relation algebra} is a relation algebra that contains a
  pair of conjugated quasiprojections.  $\QRA$ is the class of all
  Q-relation algebras and ''$\QRA$'' serves as an abbreviation of
  ``Q-relation algebra''.
\end{defn}
It is shown in Prop.\ \ref{prop1a} that, in a relation algebra or
J-algebra, the third equation in \eqref{Q} implies that the first two
equations can be simplified as follows.
\begin{prop}
  \label{prop1}
  A relation algebra is a Q-relation algebra iff it contains elements
  $\a,\b$ such that $\id=\cona\rp\a=\conb\rp\b$ and $1=\cona\rp\b$.
\end{prop}\noindent
Tarski and Givant \cite[p.\,242]{MR920815} wrote, ``The main
contribution of this work to the theory of relation algebras is the
following theorem.''
\begin{thm}[Tarski {\cite[8.4(iii)]{MR920815}}]
  \label{QRA}
  Every $\QRA$ is an $\RRA$.$^{1^*}$
\end{thm}
Footnote 1* was added by Givant (signified by the asterisk) in the
years following Tarski's death in 1983. Notation in the quoted
footnote is explained below.
\begin{quote}
  $^{1^*}$This statement is actually equivalent to the assertion that
  $\LL^+$ and $\LL^\times$ are equipollent in means of proof relative
  to sentences $\mathbf{Q}_{AB}$; \ie, it is equivalent to Theorem
  4.4(xxxvii) (or, alternately, it is equivalent to the semantical
  completeness of $\LL^\times$ relative to sentences
  $\mathbf{Q}_{AB}$, \ie, it is equivalent to Theorem 4.4(xl)). In
  fact, the proof of 8.4(iii) shows that 8.4(iii) is implied by
  4.4(xxxvii). For a simple proof of the converse implication, we use
  the methods of the next section, in particular \dots\ \cite[fn.\,1*,
    pp.\,242--3]{MR920815}
\end{quote}
In footnote 1*, $\LL^+$ is a conservative extension of $\LL$, where
$\LL$ is first-order logic with equality and binary relation symbols,
but no function symbols or constants. $\LL^+$ is obtained by adding
operators that produce new binary relation symbols from old, along
with axioms that define the meanings of these new symbols.  For
example, the axiom for the operator $+$ says that in any model the
relation symbol $A+B$ denotes the union of the relations denoted by
$A$ and $B$.  Tarski also added a new equality symbol that combines
two relation symbols into a sentence which is true in a model iff the
relations denoted by the two relation symbols are the
same. $\LL^\times$ is a sub-system of $\LL^+$ whose sentences are just
the equations between relation symbols. Its axioms are the
appropriately translated equational axioms for relation
algebras. $\mathbf{Q}_{AB}$ is an equation between relation symbols
$A$ and $B$ that says $A$ and $B$ are conjugated quasiprojections,
\ie, an appropriately translated single equation that is equivalent to
\eqref{Q}.

The omitted part of footnote 1* (that follows the quoted part) is an
edited version of a proof I sent to Givant in 1985 at his request in
response to comments by Jan Mycielski, who was the referee of
\cite{MR920815}.  After describing their proof of $\QRA\subseteq\RRA$
which was first announced in 1953 \cite{Tarski1953}, Tarski and Givant
wrote,
\begin{quote}
  The reasoning just outlined uses essentially Theorem 4.4(xxxvii) and
  depends therefore on the heavy proof-theoretical argument by means
  of which that theorem has been established. On the other hand, in
  Maddux \cite{MR460210} a substantial generalization of 8.4(iii) can
  be found which, moreover, is established by purely algebraic
  methods.$^{3^*}$ \cite[p.\,244]{MR920815}
\end{quote}
The ``substantial generalization of 8.4(iii)'' is Theorem \ref{TRA}
below.  Mycielski's concern, addressed by footnote 1*, was that the
proof of the main result in \cite{MR920815} might be out-of-date.

Givant's footnote $3^*$, added to the remarks on page 244 quoted
above, says,
\begin{quote}
  $^{3^*}$In view of the observation made in footnote $1^*$ on p.\,242,
  Maddux's algebraic proof of 8.4(iii) gives a semantical proof of the
  relative equipollence $\LL^\times$ and $\LL^+$ in means of proof,
  \ie, of Theorem 4.4(xxxvii).  Their relative equipollence in means
  of expression, Theorem 4.4(xxxvi), was already established by
  semantical methods in 4.4(xiv). Of course, Maddux's proof also gives
  us a semantical proof of the various properties of the translation
  mappings $\KAB$ (cf.\,2.4(vi)). \cite[fn.\,3*, p.\,244]{MR920815}
\end{quote}
Mycielski was concerned that the lack of a semantical proof of the
main result made the book outdated, but this semantical proof was
provided by Givant's footnotes.

The ``translation mappings $\KAB$'' mentioned in footnote $3^*$ are
discussed by Tarski and Givant in \SS4.3, pp.\ 107--110, entitled
``Historical remarks regarding the translation mapping from $\LL^+$ to
$\LL^\times$''. The section begins,
\begin{quote}
  In establishing the relative equipollence of $\LL^+$ and
  $\LL^\times$, \ie, the equipollence of the systems obtained by
  relativizing the formalisms $\LL^+$ to $\LL^\times$ to any given
  sentence $\mathbf{Q}_{AB}$, we shall apply the same general method
  which was used to establish the equipollence of $\LL^+$ and $\LL$ in
  Chapter 2 and of $\LL^+_3$ and $\LL^\times$ in Chapter 3.$^{3^*}$
\end{quote}
Givant's footnote about their proof of relative equipollence says,
\begin{quote}
  $^{3^*}$Thus, the proof we shall give may be regarded as a
  syntactical proof. In the footnote on p.\ 242 we discuss briefly a
  semantical proof essentially due to Maddux.  \cite[fn.\,3*,
    p.\,107]{MR920815}
\end{quote}
Later in \SS4.3 on p.\ 109, after pointing out that ``\dots\ the proof
(by induction on sentences derivable in $\LL^+$) that $\KAB$ has the
desired property \dots turns out to be more complicated than one would
expect'', Tarski and Givant wrote,
\begin{quote}
  There is another construction of translation mappings which leads to
  some simplification of both the basic definitions and the proofs of
  the fundamental results. This construction was discovered by Monk
  around 1960 (but was never published and was not known to the
  authors) and was rediscovered in 1974 by Maddux in a slightly
  modified form. With their permission we shall use the new
  construction as a base for the subsequent discussion, and in fact we
  shall present it in the form given by Maddux. In particular, the
  specific proof of (viii) that we shall give in the next section is
  essentially due to Maddux.  \cite[p.\,109]{MR920815}
\end{quote}
Theorem (viii), the subject of these remarks, is
\begin{thm}[{\cite[4.3(viii)]{MR920815}}]
  \label{viii}
  For every $\Psi\subseteq\Sigma^+$ and every $X\in\Sigma^+$, if
  $\Psi\vdash^+X$ then
  \begin{align*}
    \mathbf{K}^*_{AB}(\Psi)\vdash^\times_{\mathbf{Q}_{AB}}\KAB(X).
  \end{align*}
\end{thm}\noindent
In Theorem \ref{viii}, $\Sigma^+$ is the set of sentences of $\LL^+$
and $\Sigma^\times$ is the set of equations of $\LL^\times$.  For any
two relation symbols $A$ and $B$, $\KAB$ is a function that takes a
sentence of $\LL^+$ as input and produces an equation in $\LL^\times$.
The theorem says that if a sentence $X$ is provable from a set of
sentences $\Psi$ in first-order logic then the translation of $X$ into
an equation is provable from the equations that are the translations
of the sentences in $\Psi$, using just the equational axioms for
relation algebras along with equations asserting that $A$ and $B$ are
a pair of conjugated quasiprojections.

The proof of Theorem \ref{viii} occupies all of \SS4.4,
pp.\ 110--124, entitled, ``Proof of the main mapping theorem for
$\LL^\times$ and $\LL^+$.''  The {\bf Main Mapping Theorem}, a
restatement of Theorem \ref{viii} as an equivalence with the
hypothesis weakened to provability relative to $\mathbf{Q}_{AB}$, is
\begin{thm}[{\cite[4.4(xxxiv)]{MR920815}}]
  \label{xxxiv}
  For every $\Psi\subseteq\Sigma^+$ and every $X\in\Sigma^+$, we have
  \begin{align*}
    \Psi\vdash^+_{\mathbf{Q}_{AB}}X\;\;\text{ iff }\;\;
    \mathbf{K}^*_{AB}(\Psi)\vdash^\times_{\mathbf{Q}_{AB}}\KAB(X).
  \end{align*}
\end{thm}\noindent
It says that, relative to the assumption that $A$ and $B$ are
conjugated quasiprojections, provability in Tarski's conservative
extension of first-order logic is equivalent to provability in the
equational theory of relation algebras, where $\KAB$ is the function
that translates first-order sentences into equations.

In \cite[\SS4.3]{MR920815}, Tarski's original mapping is given and it
is stated that the Main Mapping Theorem is based on a new mapping that
was discovered by Monk around 1960, rediscovered by me in 1974, not
published by Monk, and unknown to the authors. This is all true, but
there is more to the story. I discovered the new mapping by trying to
write down the old mapping when I didn't have Tarski's manuscript. I
wrote what seemed natural and ended up with something different from
Tarski's original mapping. I wrote up a new proof of Theorem
\ref{viii}, based on the new mapping, and gave it to Steve Givant.  He
used it to rewrite the proof in {\cite[\SS4.4]{MR920815}.}

Prior to having this new version of the proof of the Main Mapping
Theorem, Tarski had had to reconstruct it based on his original
translation mapping. He did not have and had been seeking some old
notes written by Gebhard Fuhrken for a seminar in Berkeley.  Steve
Givant had written to Fuhrken and to Don Monk about those notes, but
they did not know where the notes might be. When Don Pigozzi heard
this story from me in the summer of 1977, he recalled that he had some
notes from one of Tarski's seminars and thought they might be the ones
Tarski and Givant were looking for. He mailed them to me when he got
back to Ames, Iowa. I copied the notes and passed them on to Tarski,
along with the news that the mapping I had discovered was already in
the notes.

Many of the seminar notes were written by Monk. After receiving the
notes, Tarski and Givant probably wrote to Monk to find out when he
discovered the new mapping. Don Pigozzi had the notes because he took
a seminar from Tarski in 1968--69 on equational logic.  Tarski wanted
him to present an undecidable equational theory, namely the equational
theory of relation algebras.  Tarski told Don to write to Monk for
some notes from an earlier seminar. Monk sent notes that came from a
Berkeley seminar, probably around 1960, plus notes from a seminar in
the sixties at Boulder, Colorado. Don thought these notes were written
by Steve Comer and Jim Johnson as well as Monk. The notes also
included the ones by Fuhrken that Tarski wanted so badly. Don used the
notes, made his report, and, years later, sent them all to me.

Thus, what Don sent me were notes from three seminars, one in Berkeley
around 1960, one in Boulder in the 1960's, and Don's own notes for the
1968--69 seminar in Berkeley. The new translation mapping that appears
in those notes may have been presented in Tarski's own seminar just
three years before he began converting his 1942 manuscript into the
book \cite{MR920815}.  Of course, Don may not have presented the
mapping, but it's the central feature of the proof.

\section{Tabular relation algebras}
\label{s8}
\begin{defn}
  \label{tab-orig}
  A relation algebra $\AA$ is {\bf tabular} if every non-zero element
  contains a non-zero element of the form $\con\p\rp\q$ with
  $\p,\q\in\Fn\AA$.
\end{defn}
\begin{thm}[{\cite[Th.\,9(2)]{MR2628352}, \cite[Th.\,7]{MR460210},
      \cite[Th.\,423]{MR2269199}}]
  \label{TRA}
  Every tabular relation algebra is representable.
\end{thm}\noindent
That Theorem \ref{TRA} is ``a substantial generalization'' of
$\QRA\subseteq\RRA$ follows from the observation that if $\x$ is not
zero in a Q-relation algebra with conjugated quasiprojections $\a$ and
$\b$, then $\x$ contains the non-zero element $\con\p\rp\q$ where
$\p=\a\bp\b\rp\con\x$ and $\q=\b\bp\a\rp\x$. This step depends on the
fact that relation algebras have a Boolean part. Theorem \ref{TRA} and
its extension to J-algebras will be proved later in Theorem
\ref{tabrra}, using the following alternate formulation of tabularity.
\begin{prop}
  \label{tab-prop}
  A relation algebra $\AA$ is tabular iff for all $\v,\w\in\A$,
  $\v<\w$ implies there are $\p,\q\in\Fn\AA$ such that
  $0\neq\con\p\rp\q\leq\w$ and $\v\bp\con\p\rp\q=0$.
\end{prop}
\begin{proof}
  Assume $\AA$ is tabular and $\v<\w$. In a relation algebra, $\v<\w$
  implies $\w\bp\min\v\neq0$. By tabularity, there are
  $\p,\q\in\Fn\AA$ such that $0\neq\con\p\rp\q\leq\w\bp\min\v$, hence
  $0\neq\con\p\rp\q\leq\w$ and $\con\p\rp\q\leq\min\v$, but the latter
  equation implies $\con\p\rp\q\bp\v=0$.  For the converse, assume the
  alternate form of tabularity and $0\neq\x$. Then $\v<\w$ where
  $\v=0$ and $\w=\x$, so there are $\p,\q\in\Fn\AA$ such that
  $0\neq\con\p\rp\q\leq\w=\x$ and $0\bp\con\p\rp\q=\v\bp\con\p\rp\q
  =0$.  The first group of equations is the desired conclusion and the
  second set is always true.
\end{proof}

\section{Qu-algebras}
\label{s9}
If $\a$ and $\b$ are functional elements of the simple proper relation
algebra $\Re\X$, then they are, in fact, functions whose domains are
subsets of $\X$. The Universal Domain property,
\begin{align*}
  1=\a\rp1=\b\rp1,
\end{align*}
says that $\a$ and $\b$ are defined on every point in $\X$ and are
therefore unary operations on $\X$. The Universal Domain property is
not needed for the main results of \cite{MR920815}, but it is so
convenient in proofs that it was arranged to hold by the creation of
new quasiprojections from given ones, according to the next
proposition and its corollary. See the remarks surrounding
\cite[4.1(iv), 4.1(xi)]{MR920815} and consult \cite{MR920815} for a
proof of the following proposition, which is rather complicated when
carried out directly from the axioms for relation algebras.
\begin{prop}[{\cite[4.1(xii)]{MR920815}}]
  Assume $\AA\in\RA$, $\a,\b\in\A$, and \eqref{Q} holds. If
  $\p=\a+\id\bp\min{\a\rp1}$, and $\q=\b+\id\bp\min{\b\rp1}$ then
  $\id=\con\p\rp\p=\con\q\rp\q$ and $1=\con\q\rp\p=\p\rp1=\q\rp1$.
\end{prop}
\begin{cor}
  \label{cor2}
  $\AA$ is a Q-relation algebra iff $\AA\in\RA$ and there are
  $\a,\b\in\A$ such that $\id=\cona\rp\a=\conb\rp\b$ and
  $1=\cona\rp\b=\a\rp1=\b\rp1$.
\end{cor}
For any two elements $\a$ and $\b$ of $\Re\X$, the equation
$1=\cona\rp\b$ asserts that if $\x,\y\in\X$ then there is a point
$\z\in\X$ such that $\<\z,\x\>\in\a$ and $\<\z,\y\>\in\b$. If $\a$ and
$\b$ are functional elements of $\Re\X$, then they are functions and
$\x$ and $\y$ can be recovered from $\z$ by applying $\a$ and $\b$,
that is, $\a(\z)=\x$ and $\b(\z)=\y$. The point $\z$ may not be the
only one with this property, but the Unicity Condition
\begin{align*}
  \a\rp\cona\bp\b\rp\conb\leq\id
\end{align*}
says that the point $\z$ is uniquely determined by $\x$ and $\y$.
Indeed, if $\<\z,\x\>\in\a$, $\<\z,\y\>\in\b$, $\<\z',\x\>\in\a$, and
$\<\z',\y\>\in\b$, then $\<\x,\z'\>\in\conv\a$ and
$\<\y,\z'\>\in\conv\b$, so $\<\z,\z'\>\in\a|\conv\a\cap\b|\conv\b
\subseteq Id(\X^2)$ by \eqref{U} applied to $\Re\X$, hence $\z=\z'$.

By Corollary \ref{cor2}, Q-relation algebras are those relation
algebras with elements satisfying both \eqref{Q} and \eqref{D}. The
ones that also satisfy the Unicity Condition \eqref{U} are given a
name with a ``u'' added to Tarski's name as a reminder that
``unicity'' is assumed. In Prop.\ \ref{prop2a} it is shown that
\eqref{Q}, \eqref{D}, and \eqref{U} are jointly equivalent to the
condition \eqref{Qu}.
\begin{defn}
  A {\bf Qu-algebra} is a relation algebra with elements $\a,\b$ such
  that
  \begin{align}
    \label{Qu}
    \id=\cona\rp\a=\conb\rp\b=\a\rp\cona\bp\b\rp\conb,
    \quad1=\cona\rp\b=\a\rp1=\b\rp1.
  \end{align}
\end{defn}

\section{Fork algebras}
\label{s10}
Fork algebras were introduced by Haeberer, Baum, Schmidt, and Veloso
in the early 1990s for applications in computer science; see
\cite{MR1131030, BHSV93, MR1302717}. The primary reference for fork
algebras is \cite{MR1922164}.  Fork algebras are also heavily involved
in the finitization problem; see \cite{ MR1636979, MR1442088,
  MR1131030, BHSV93, MR1302717, MR1478959, MR1996109, MR1611255,
  MR1858518, MR1336885, MR1370064, MR1450025, MR2755277, MR1616405,
  MR1488293, MR2338322}.
\begin{defn}[{\cite[Def.\;3.4]{MR1922164}}]
  A {\bf fork algebra} is a relation algebra with a binary operation
  $\nabla$ satisfying these axioms:
  \begin{align}
    \label{F1}
    \x\nabla \y&=\x\rp(\id\nabla1)\bp\y\rp(1\nabla\id), \\
    \label{F2}
    \u\rp\con\v\bp\x\rp\con\y&=(\u\nabla\x)\rp\con{\v\nabla\y}, \\
    \label{F3}
    \id&\geq\con{\id\nabla1}\nabla\con{1\nabla\id}.
  \end{align}
\end{defn}
In a fork algebra, the elements $\con{\id\nabla1}$ and
$\con{1\nabla\id}$ form a pair of conjugated quasi-projections that
also satisfy the Unicity Condition. Therefore, the algebra obtained by
deleting $\nabla$ is a Q-relation algebra and it is representable by
Tarski's theorem that $\QRA\subseteq\RRA$. This connection has been
used to prove representability for fork algebras; see \cite{MR1465620,
  MR1365215, MR1479630}.

Props.\ \ref{fork} and \ref{QPr}, coming up next, use
Props.\ \ref{p2}, \ref{p4}, \ref{p5}, \ref{p6}, \ref{i1}, \ref{pair},
and \ref{pair2} in their proofs.  Recall that
Props.\ \ref{p1}--\ref{presM} are proved for all elements in an
arbitrary relation algebra or J-algebra. They are derived in Part II
from the axioms for J-algebras and can be used in proofs in Part
I. Props.\ \ref{fork} and \ref{QPr} involve the exceptionally
important identity defined next.
\begin{defn}
  \label{pairing}
  The {\bf pairing identity} for elements $\a$ and $\b$ is
  \begin{align}
    \label{Pr}
    \u\rp\v\bp\x\rp\y&=(\u\rp\cona\bp\x\rp\conb)\rp(\a\rp\v\bp\b\rp\y).
  \end{align}
\end{defn}
\begin{prop}
  \label{fork}
  The axioms for fork algebras hold in a J-algebra or relation algebra
  $\AA$ if $\AA$ has elements $\a$ and $\b$ that satisfy \eqref{Q} and
  \eqref{U} and $\nabla$ is defined by $$\x\nabla\y=\x\rp\cona\bp\y\rp
  \conb.$$
\end{prop}
\begin{proof}
  The conjugated quasiprojections are
  \begin{align*}
    \con{1\nabla\id}
    &=\con{1\rp\cona\bp\id\rp\conb}&&\text{def $\nabla$}
    \\&=\con{1\rp\cona}\bp\con{\id\rp\conb}&&\text{\eqref{4}}
    \\&=\con{\cona}\rp\con1\bp\con{\conb}\rp\con\id&&\text{\eqref{5}}
    \\&=\a\rp\con1\bp\b\rp\con\id&&\text{\eqref{6}}
    \\&=\a\rp1\bp\b\rp\id&&\text{Prop.\ \ref{p5}}
    \\&=\a\rp1\bp\b,&&\text{\eqref{id}}
  \end{align*}
  and, similarly, $\con{\id\nabla1}=\a\bp\b\rp1$.  Therefore, axiom
  \eqref{F3} holds because
  \begin{align*}
    \con{\id\nabla1}\,\nabla\,\con{1\nabla\id}
    &=(\a\bp\b\rp1)\rp\cona\bp(\a\rp1\bp\b)\rp\conb &&\text{}
    \\ &\leq\a\rp\cona\bp\b\rp\conb &&\text{Props.\ \ref{p2},
      \ref{p4}} \\ &\leq\id&&\text{\eqref{U}}
  \end{align*}
  By Prop.\ \ref{p5}, $\id\nabla1=\cona\bp1\rp\conb$ and
  $1\nabla\id=1\rp\cona\bp\conb$, so axiom \eqref{F1} takes the form
  \begin{align*}
    \x\rp\cona\bp\y\rp\conb
    &=\x\rp(\cona\bp1\rp\conb)\bp\y\rp(\conb\bp1\rp\cona),
  \end{align*}
  which can be proved as follows.  We have $\x\leq1$ and $\y\leq1$ by
  \eqref{one}, so $\x\rp\cona\leq1\rp\cona$ and
  $\y\rp\conb\leq1\rp\conb$ by Prop.\ \ref{p6}, \ie,
  $\x\rp\cona\bp1\rp\cona=\x\rp\cona$ and
  $\y\rp\conb\bp1\rp\conb=\y\rp\conb$. This accounts for the first
  step in
  \begin{align*}
    \x\rp\cona\bp\y\rp\conb&=
    (\x\rp\cona\bp1\rp\cona)\bp(\y\rp\conb\bp1\rp\conb)
    \\&=(\x\rp\cona\bp1\rp\conb)\bp(\y\rp\conb\bp1\rp\cona)
    &&\text{\eqref{comm}, \eqref{bassoc}}
    \\&=\x\rp(\cona\bp1\rp\conb)\bp\y\rp(\conb\bp1\rp\cona)
    &&\text{Prop.\ \ref{i1}}
  \end{align*}
  Substituting $\con\v$ for $\v$ and $\con\y$ for $\y$ in axiom
  \eqref{F2} produces an equation that is equivalent to \eqref{F2}
  because of \eqref{6}.  When this equation is rewritten using the
  definition of $\nabla$ it becomes $\u\rp\v\bp\x\rp\y = (\u\rp\cona
  \bp \x\rp\conb)\rp\con{\con\v\rp\cona \bp \con\y\rp\conb}.$ By
  \eqref{6}, \eqref{5}, and \eqref{4}, this equation is equivalent to
  the pairing identity \eqref{Pr} for $\a,\b$.  By
  \cite[4.1(viii)]{MR920815}, \eqref{Pr} follows from \eqref{Q} in any
  relation algebra without help from \eqref{U} or \eqref{D}, but the
  derivation of \eqref{Pr} from \eqref{Q} is generalized from relation
  algebras to J-algebras in Prop.\ \ref{pair}. Axiom \eqref{F2}
  therefore holds by Prop.\ \ref{QPr} below.
\end{proof}
\begin{prop}
  \label{QPr}
  If $\AA\in\JA\cup\RA$ then \eqref{Q} implies \eqref{Pr}.
\end{prop}
\begin{proof}
  Assume $\a,\b\in\Fn\AA$ and $1=\cona\rp\b$.  We can get \eqref{Pr}
  using either of two later propositions about J-algebras.  Let
  $\c=\a$ and $\d=\b$.  Then $\a,\b,\c,\d\in\Fn\AA$, as is required
  for both Prop.\ \ref{pair}(ii) and Prop.\ \ref{pair2}(iii).  The
  remaining hypotheses of Prop.\ \ref{pair}(ii) are satisfied because
  $\u\leq1=\con\c\rp\d=\cona\rp\b$ and $\v\rp\con\y\leq1=\cona\rp\b$.
  For Prop.\ \ref{pair2}(iii) we need only note that
  $\u\rp\v\cdot\x\rp\y\leq1=\con\c\rp\d=\cona\rp\b$.  In both cases
  the conclusion is \eqref{Pr}.
\end{proof}

\section{The pairing identity}
\label{s10a}
Referring to Theorem 4.1(viii), that \eqref{Q} implies \eqref{Pr},
Tarski and Givant wrote,
\begin{quote}
  ``The proof of the next theorem is the first long and rather
  involved derivation in this chapter within the formalism
  $\LL^\times$. In connection with such derivations, the reader may
  recall the closing remarks of \SS3.2. The particular proof of this
  theorem presented below is due to Maddux.'' \cite[p.\,97]{MR920815}
\end{quote}
In ``the closing remarks of \SS3.2'' they point out that derivations
in $\LL^\times$ have the same form as equational derivations from the
axioms for relation algebras.  An equational derivation of \eqref{Pr}
from \eqref{Q} is encountered at this early stage in \cite{MR920815}
because \eqref{Pr} plays a vital r\^ole in the proof{s} of Theorem
\ref{viii} and the Main Mapping Theorem. It needs to be established
first before the main work begins.

As was mentioned earlier in the quotations from \cite{MR920815},
Tarski's original proof of $\QRA\subseteq\RRA$ depends on Theorem
4.4(xxxvii), which requires a ``heavy proof-theoretical argument''.
Tarski realized that one could go in the other direction and prove
Theorem 4.4(xxxvii) using $\QRA\subseteq\RRA$ (as was eventually done
in Givant's footnote 1*, quoted earlier).  However, this would require
a proof of $\QRA\subseteq\RRA$ that did not use Theorem 4.4(xxxvii).
Since $\QRA\subseteq\RRA$ is a purely algebraic statement about
relation algebras, such a proof should also be purely algebraic and
remain within the theory of relation algebras.

Tarski recommended this problem to his student George McNulty during
George's final year at Berkeley. George passed the problem along to
me.  Having already proved that point-dense relation algebras are
representable (see \cite{MR1049616}), I thought this problem might be
easily solved using similar methods.  My initial attempts in the fall
of 1973 to discover an abstract algebraic proof of $\QRA\subseteq\RRA$
included the consideration of \eqref{Pr} as a test case.

The advantage of the pairing identity \eqref{Pr} is that it is shorter
and simpler than some other equations that fail in some relation
algebra and yet hold in every representable relation algebra, such as
the equations \eqref{L} and \eqref{M} that are mentioned in
\SS\ref{s18}. Also, \eqref{Pr} directly involves the functional
elements $\a$ and $\b$, whose presence in the algebra is required for
representability. I thought that an equational derivation of
\eqref{Pr} from \eqref{Q} might suggest an algebraic method for
proving $\QRA\subseteq\RRA$. Instead, in December of 1973 I found an
algebraic proof of $\QRA\subseteq\RRA$, described below in
\SS\ref{origQRAth}, that relies on the representability of point-dense
relation algebras. My proof suggested a method for creating an
equational derivation of \eqref{Pr} from \eqref{Q}.

The situation was reversed a year later, when an equational derivation
of \eqref{Pr} was required for a proof rather than being suggested by
a proof.  At that time in 1974 the manuscript for \cite{MR920815}
contained Tarski's original construction of $\KAB$.  Tarski and Givant
were working on a proof of Theorem \ref{viii} that was based on this
construction and was too long to be entirely included in the book.
They called the proof ``more complicated than one would expect'', as
was mentioned earlier.  While home for the holidays in December, 1974,
I attempted to write out Tarski's construction of $\KAB$ without
having Tarski's manuscript with me. I started fresh, using $\exists$
as primitive and $\forall$ as defined, unlike Tarski's original
construction, which takes $\forall$ as primitive and $\exists$ as
defined.  The resulting construction of $\KAB$ allowed me to write out
a complete proof of Theorem \ref{viii} in a reasonable number of
pages. This proof, dated May 23, 1975, was given to Givant, edited,
and included in \SS4.4, as was noted above in \SS\ref{s7}.

By Tarski's theorem that $\QRA\subseteq\RRA$, \eqref{Pr} must hold in
any relation algebra satisfying \eqref{Q}.  Indeed, a relation algebra
satisfying \eqref{Q} has a representation, so a proof of \eqref{Pr} in
a relation algebra satisfying \eqref{Q} can proceed by looking at the
points in the base of a representation.  However, Tarski's original
method of proof through metamathematical means offered no obvious way
of constructing a direct equational derivation of \eqref{Pr} from
\eqref{Q} using the axioms for relation algebras.

By contrast, the first algebraic proof of $\QRA\subseteq\RRA$,
described in the next section, shows that one can assume the existence
of ``points''. These are elements of the algebra that mimic the
behavior of singleton relations of the form $\{\<\x,\x\>\}$ where $\x$
is an element in the base of a representation. This allows the
construction of equational derivations that directly mimic proofs of
equations that refer to points in the base set of a representation.

For the proof of Theorem \ref{viii} in \cite{MR920815}, this situation
is reversed.  The first step is to derive \eqref{Pr} directly from
\eqref{Q} using the axioms for relation algebras.  The pairing
identity \eqref{Pr} is then repeatedly applied in proofs of properties
of finite sequences.

The same situation occurs here.  The pairing identity \eqref{Pr} must
be derived first because it is used frequently in proofs of properties
of elements representing actions on trees; see \SS\ref{s20}.
Furthemore, the equational derivation of \eqref{Pr} from \eqref{Q}
must be based on the more limited set of axioms for J-algebras; see
the proofs of Props.\ \ref{pair} and \ref{pair2} in \SS\ref{s17}.

\section{The first algebraic proof of Tarski's theorem}
\label{origQRAth}
A purely algebraic proof of Theorem \ref{QRA}, found in December 1973,
begins with Lemmas \ref{firstlemma} and \ref{secondlemma}.  Lemma
\ref{firstlemma} says that if a $\QRA$ has a nonzero element $\z$ then
it can be embedded in a larger $\QRA$ that has a ``point'' $\u$ in the
domain of $\z$.  Lemma \ref{secondlemma} follows immediately from
Lemma \ref{firstlemma} by the general theory of algebras.  Using Lemma
\ref{secondlemma}, one can prove $\QRA\subseteq\RRA$ by imitating the
construction of an algebraically closed extension of an arbitrary
field to show that every Q-relation algebra can be embedded in a
relation algebra that is point-dense and therefore representable by
\cite[Th.\,54]{MR1049616}.
\begin{lem}
  \label{firstlemma}
  If $\AA=\<A,+,\bp,\min\blank,0,1,\rp,\con\blank,\id\>\in\QRA$ and
  $0\neq\z\in\A$ then there is some $\gc\B\in\QRA$ and a function
  $\f\colon\A\to\B$ such that
  \begin{enumerate}
  \item
    $\f$ is an isomorphic embedding of $\AA$ into $\BB$,
  \item
    for some $\u\in\B$, $\u\rp\di\rp\u=0$ and
    $1\rp\u\rp\f(\z)\rp1=1\rp\f(\z)\rp1$.
  \end{enumerate}
\end{lem}
\begin{proof}
  Since $\AA\in\QRA$, by Corollary \ref{cor2} there are $\a,\b\in\A$
  such that $\id=\cona\rp\a=\conb\rp\b$ and
  $1=\cona\rp\b=\a\rp1=\b\rp1$. Let $\w=\b\rp\z\rp1+\min{1\rp\z\rp1}$,
  $\B=\{\x:\x\in\A,\,(\a\rp\cona\bp\b\rp\conb)\rp\x=\x\leq\w\}$,
  $\i=\a\bp\w$, and, for all $\x,\y\in\A$, let
  $\widetilde\x=\min\x\bp\w$,
  $\x^\circ=\a\rp\con\x\bp\b\rp\conb\rp\a$,
  $\x\bullet\y=\x\rp\cona\bp\b\rp\conb\rp\y$, and
  $\f(\x)=\a\rp\x\bp\w$.  Set
  $\BB=\<\B,+,\bp,\min\blank,0,1,\bullet,{}^\circ,\i\>$.  Some
  calculations show that $\BB\in\QRA$ and $\f$ embeds $\AA$ into
  $\BB$.  To prove part (ii), let $\u=\b\bp\f(\id)$ and verify by some
  more calculation that $\u\in\B$, $\u\rp\di\rp\u=0$, and
  $1\rp\u\rp\f(\x)\rp1=1\rp\f(\x)\rp1$.
\end{proof}
\begin{lem}
  \label{secondlemma}
  If $\AA\in\QRA$ and $\z\in\A$ then there is some $\BB\in\QRA$ such
  that
  \begin{enumerate}
  \item $\AA\subseteq\BB$,
  \item there is some $\u\in\B$ such that $\u\rp\di\rp\u=0$ and
    $1\rp\u\rp\x\rp1=1\rp\x\rp1$.
  \end{enumerate}
\end{lem}
In the seminar notes discussed in remarks following Theorem
\ref{xxxiv} in \SS\ref{s7} there is algebraic proof of a
cylindric-algebraic version of $\QRA\subseteq\RRA$ that bears a
striking resemblance to the proof just outlined. It is in a section of
Monk's notes entitled ``\SS6. \emph{Pairing elements in cylindric
algebras}''.  Because of the connections between cylindric algebras
and relation algebras it is possible to see that Monk's algebraic
proof and the algebraic proof of $\QRA\subseteq\RRA$ described above
are, in general methodological terms, ``the same''.

Both proofs show that a Q-relation algebra, or a ``Q-cylindric
algebra'' in Monk's case, can be embedded in another one with one more
dimension.  This embedding can be repeated for Q-cylindric algebras
until one gets an embedding into a locally finite-dimensional
$\omega$-dimensional cylindric algebra, which had already been proved
to be representable by Tarski. This step is similar to the extension
of a relation algebra to one that is point-dense.

\section{Direct products}
\label{s18}
Abstract algebraic formulations of direct products have been
introduced, either as data types or as operations, for applications in
computer science by de~Roever \cite{MR0660461}, Schmidt
\cite{MR0619683}, Schmidt and Str\"ohlein \cite{MR1254438}, Zierer
\cite{MR1130150}, and Berghammer and Zierer \cite{MR855968}.  The
relevance of direct products to this paper is that the first three
conditions in the following definition are jointly equivalent to
\eqref{Qu}, the conclusion of formula \eqref{I} is the pairing
identity, and the conclusion of formula \eqref{J} is stronger half of
the pairing identity. Schmidt's conjecture and the results obtained
for its solution show the limits and extent of the validity of the
pairing identity.
\begin{defn}
  \label{ij1}
  Two elements $\a$ and $\b$ of a relation algebra $\AA$ form a {\bf
    direct product} if the following conditions hold.
  \begin{enumerate}
  \item
    $\a,\b\in\Fn\AA$ \quad($\a$ and $\b$ are functional),
  \item
    $\a\rp1=\b\rp1$ \quad($\a$ and $\b$ have the same domain),
  \item
    $\a\rp\cona\bp\b\rp\conb\leq\id$ \quad(ordered pairs are unique),
  \item
    $\cona\rp1\rp\b=\cona\rp\b$ \quad(\;$\cona\rp\b$ is
    ``rectangular'').
  \end{enumerate}
  Let $\Pi(\a,\b)$ be the conjunction of conditions \rm(i)--(iv).
  Then \eqref{I} and \eqref{J} are the formulas
  \begin{align}
    \label{I}
    \Pi(\a,\b)\land\con\u\rp\x\bp\v\rp\con\y\leq\cona\rp\b\implies
    \u\rp\v\bp\x\rp\y
    =(\u\rp\cona\bp\x\rp\conb)\rp(\a\rp\v\bp\b\rp\y),
    \\
    \label{J}
    \con\u\rp\x\bp\v\rp\con\y\leq\cona\rp\b\implies\u\rp\v\bp\x\rp\y
    \leq(\u\rp\cona\bp\x\rp\conb)\rp(\a\rp\v\bp\b\rp\y).
  \end{align}
\end{defn}
Because of the apparent difficulty of deriving \eqref{I} in an
arbitrary relation algebra, Gunther Schmidt conjectured that it may
not be possible. His conjecture was confirmed by the following result.
\begin{thm}[\cite{MR1375942}]
  \label{ij5}
  There is a finite simple nonintegral relation algebra with 58 atoms
  in which \eqref{I} and \eqref{J} fail for appropriately chosen atoms
  $\a,\b,\u,\v,\x,\y$.
\end{thm}
Besides \cite{MR1375942}, this construction can be found in
\cite[\SS3.2]{KahlSchmidt2000}.  The solution to Schmidt's conjecture
is discussed by Kahl and Schmidt \cite[\SS3.2]{KahlSchmidt2000},
Schmidt and Winter \cite[\SS3.2]{MR3793161}, and by Berghammer,
Haeberer, Schmidt, and Veloso \cite[\SS9]{BHSV93}. Although \eqref{I}
and \eqref{J} can fail, \eqref{J} holds under five similar hypotheses.
\begin{thm}[\cite{MR1375942}]
  \label{ij6}
  Assume $\AA\in\JA\cup\RA$, $\a,\b,\c,\d,\t,\u,\v,\x,\y\in A$, and
  $\t$ is one of these five elements: $\u\rp\v\bp\x\rp\y$,
  $\u\bp\x\rp\y\rp\con\v$, $\v\bp\con\u\rp\x\rp\y$,
  $\x\bp\u\rp\v\rp\con\y$, $\y\bp\con\x\rp\u\rp\v$.  If
  $\t\leq\con\c\rp\d$ and $\c,\d\in\Fn\AA$ then \eqref{J} holds.
\end{thm}
Theorem \ref{ij6} is proved in Prop.\ \ref{pair2} for the case
$\t=\u\rp\v\bp\x\rp\y$.  Consult \cite{MR1375942} for the remaining
cases, which reduce to a single proof because of the symmetries
involved.

We can summarize the contents of Theorems \ref{ij5} and \ref{ij6} by
using the following implication.
\begin{gather}
  \label{K}
        {}^{\text{(1)}}\cona\rp\a\leq\id
        \land{}^{\text{(2)}}\conb\rp\b\leq\id
        \land{}^{\text{(3)}}\a\rp1=\b\rp1
        \land{}^{\text{(4)}}\a\rp\cona\bp\b\rp\conb\leq\id \\\notag
        \land{}^{\text{(5)}}\cona\rp1\rp\b=\cona\rp\b
        \land{}^{\text{(6)}}\con\c\rp\c\leq\id
        \land{}^{\text{(7)}}\con\d\rp\d\leq\id
        \land{}^{\text{(8)}}\t\leq\con\c\rp\d \\\notag
        \land{}^{\text{(9)}}\con\u\rp\x\bp\v\rp\con\y\leq\cona\rp\b
        \implies{}^{\text{(10)}}
        \u\rp\v\bp\x\rp\y=(\u\rp\cona\bp\x\rp\conb)\rp(\a\rp\v\bp\b\rp\y).
\end{gather}
Formulas numbered (1)--(5) in \eqref{K} say that $\a,\b$ form a direct
product. Formulas (6) and (7) assert that $\c,\d\in\Fn\AA$.
Theorems \ref{ij5} and \ref{ij6} imply that \eqref{K} fails in some
relation algebra when $\t$ is the left-hand-side of (9), even if the
hypotheses $\c=\a$ and $\d=\b$ are added, but \eqref{K} holds in every
relation algebra whenever $\t$ is replaced by $\u\rp\v\bp\x\rp\y$,
$\u\bp\x\rp\y\rp\con\v$, $\x\bp\u\rp\v\rp\con\y$,
$\v\bp\con\u\rp\x\rp\y$, or $\y\bp\con\x\rp\u\rp\v$, even if
hypotheses (3)--(5) are deleted.

Along with \eqref{J}, the formulas \eqref{L} and \eqref{M} below are
instances of condition $(\Gamma)$ when $n=3$ in J\'onsson's
\cite[Th.\,1]{MR108459}. Note that \eqref{J} closely resembles
\eqref{L} while \eqref{M} is rather different. This similarity and
difference is indicated by Lyndon's \cite{MR37278} way of writing
them, where $\x_{ji}=\con{\x_{ij}}$ when $i,j=0,\dots,6$. The terms in
the hypothesis of \eqref{J} appear in the left side of \eqref{L} and
the terms in the conclusion of \eqref{J} appear in the right side of
\eqref{L}.
\begin{gather}
  \x_{20}\rp\x_{03}\bp\x_{21}\rp\x_{13}\leq\x_{24}\rp\x_{43}\implies
  \\ \notag
  \x_{02}\rp\x_{21}\bp\x_{03}\rp\x_{31}\leq(\x_{02}\rp\x_{24}\bp
  \x_{03}\rp\x_{34})\rp(\x_{42}\rp\x_{21}\bp\x_{43}\rp\x_{31}), \\
  \label{L}
  \x_{20}\rp\x_{03}\bp\x_{21}\rp\x_{13}\bp\x_{24}\rp\x_{43}\leq
  \\ \notag \x_{20}\rp\big(\x_{02}\rp\x_{21}\bp\x_{03}\rp\x_{31}\bp
  (\x_{02}\rp\x_{24}\bp\x_{03}\rp\x_{34})\rp(\x_{42}\rp\x_{21}\bp
  \x_{43}\rp\x_{31})\big)\rp\x_{13}, \\
  \label{M}
  \x_{01}\bp(\x_{02}\bp\x_{05}\rp\x_{52})\rp(\x_{21}\bp\x_{26}\rp\x_{61})\leq
  \\ \notag
  \x_{05}\rp\big((\x_{50}\rp\x_{01}\bp\x_{52}\rp\x_{21})\rp\x_{16}\bp
  \x_{52}\rp\x_{26}\bp\x_{50}\rp(\x_{01}\rp\x_{16}\bp\x_{02}\rp\x_{26})\big)\rp\x_{61}.
\end{gather}
The formulas \eqref{J}, \eqref{L}, and \eqref{M} are independent.
Each of them fails in some relation algebra that satisfies the other
two \cite[Ch.\,6, \SS64--5]{MR2269199}. The independence of \eqref{L}
is ``hardest'' to prove.  Table~26 in \cite{MR2269199} (which is
missing one line) and Table \ref{tabl} (which contains the missing
line) show that among 4527 integral relation algebras with at most
five atoms, there are \emph{exactly two} that fail to satisfy
\eqref{J} and \eqref{M} but not \eqref{L}. By contrast, the
independence of \eqref{J} is shown by 170 of these 4527 algebras and
that of \eqref{M} by 602. All these algebras are not representable,
because representability requires that a finite relation algebra have
an $n$-dimensional relational basis for every $\n\in\omega$ and an
algebra that has a 5-dimensional relational basis satisfies \eqref{J},
\eqref{L}, and \eqref{M} \cite[Th.\,341]{MR2269199}.  Among the 1729
relation algebras that satisfy \eqref{J}, \eqref{L}, and \eqref{M},
many are not representable because they fail to have a 5-dimensional
relational basis, or have a 5-dimensional relational basis but no
6-dimensional relational basis, \etc.
\begin{table}
  $\renewcommand{\arraystretch}{1.02}
  \begin{array}{|c|c|cccccccc|}\hline
    & \text{Total} & \text{fail:} & \text{fail:} & \text{fail:} &
    \text{fail:} & \text{fail:} & \text{fail:} & \text{fail:} &
    \text{fail:} \\ \text{Atoms} & \#\RA & \text{(J)(L)(M)} &
    \text{(J)(L)} & \text{(J)(M)} & \text{(L)(M)} & \text{(J)} &
    \text{(L)} & \text{(M)} & \emptyset \\\hline \id & 1 & 0 & 0 & 0
    & 0 & 0 & 0 & 0 & 1 \\ \id a & 2 & 0 & 0 & 0 & 0 & 0 & 0 & 0 & 2
    \\ \id a \con a & 3 & 0 & 0 & 0 & 0 & 0 & 0 & 0 & 3 \\ \id a b &
    7 & 0 & 0 & 0 & 0 & 0 & 0 & 0 & 7 \\ \id a b \con b & 37 & 5 & 0
    & 2 & 0 & 0 & 0 & 2 & 28 \\ \id a b c & 65 & 5 & 2 & 3 & 0 & 0 &
    0 & 6 & 49 \\ \id a \con a b \con b & 83 & 9 & 0 & 4 & 1 & 1 & 0
    & 8 & 60 \\ \id a b c \con c & 1316 & 369 & 76 & 127 & 16 & 37 &
    0 & 132 & 559 \\ \id a b c d & 3013 & 741 & 168 & 495 & 1 & 132
    & 2 & 454 & 1020 \\\hline \text{Totals} & 4527 &1129 & 246 & 631
    & 18 & 170 & 2 & 602 & 1729\\\hline
  \end{array}$
  \bigskip
  \caption{\Large Failures of (J), (L), and (M)}
  \label{tabl}
\end{table}

\section{J\'onsson-Tarski algebras}
\label{s11}
These algebras were introduced by J\'onsson and Tarski in 1955.
Props.\ \ref{JT=bij} and \ref{JT=Qu} below detail the intimate
connections between J\'onsson-Tarski algebras, Qu-algebras, and
bijections between a set and its Cartesian square.
\begin{defn}[{\cite[Th.\,5]{MR126399}}]
  The algebra $\UUU=\<\U,*,\a,\b\>$ is a {\bf J\'onsson-Tarski
    algebra} if $*$ is a binary operation on $\U$, $\a$ and $\b$ are
  unary operations on $\U$, and, for all $\x,\y\in\U$,
  \begin{align}
    \label{JT1} \a(\x*\y)&=\x,
    \\
    \label{JT2} \b(\x*\y)&=\y,
    \\
    \label{JT3} \a(\x)*\b(\x)&=\x.
  \end{align}
  $\JT$ is the class of J\'onsson-Tarski algebras.
\end{defn}
Using juxtaposition instead of $*$, the equations characterizing
J\'onsson-Tarski algebras become
\begin{align*}
  \a(\x\y)=\x,\quad\b(\x\y)=\y,\quad\a(\x)\b(\x)=\x.
\end{align*}
These algebras originated as an example in \cite[Th.\,5]{MR126399}. As
part of the proof J\'onsson and Tarski showed that finitely generated
J\'onsson-Tarski algebras are 1-generated, and, in fact, they have the
stronger property that if a J\'onsson-Tarski algebra $\UUU$ is freely
generated by $\X\cup\{\y,\z\}$ and $\y,\z\notin\X$, then $\UUU$ is
freely generated by $\X\cup\{\y*\z\}$.
\begin{prop}
  \label{JT=bij}
  $\UUU=\<\U,*,\a,\b\>$ is a J\'onsson-Tarski algebra if and only if
  $*\colon\U^2\to\U$ is a bijection that determines $\a$ and $\b$ by
  $\a=\{\<\x,\y\>:\exists_\z(\y*\z=\x)\}$ and
  $\b=\{\<\x,\y\>:\exists_\z(\z*\y=\x)\}$.
\end{prop}
\begin{proof} If $\UUU=\<\U,*,\a,\b\>$ is a J\'onsson-Tarski
  algebra, then $*$ is surjective because $\UUU$ satisfies
  $\a(\x)*\b(\x)=\x$.  The other two equations satisfied by $\UUU$
  show that $*$ is injective, for if $\x*\y=\x'*\y'$ then
  $\x=\a(\x*\y)=\a(\x'*\y')=\x'$ and
  $\y=\b(\x*\y)=\b(\x'*\y')=\y'$. Thus, $*$ is a bijection mapping
  $\U^2$ onto $\U$.  Conversely, to obtain a J\'onsson-Tarski
  algebra, choose any bijection $*\colon\U^2\to\U$ between a set
  $\U$ and its Cartesian square $\U^2$. Such functions exist iff
  $\U$ is infinite or has a single element.  Then $\a$ and $\b$ can
  be defined as relations by $\a=\{\<\x,\y\>:\exists_\z(\y*\z=\x)\}$
  and $\b=\{\<\x,\y\>:\exists_\z(\z*\y=\x)\}$.

  To see that $\a$ is actually a function, assume $\<\x,\y\>,
  \<\x,\y'\>\in\a$. Then by the definition of $\a$ there are
  $\z,\z'\in\U$ such that $\y*\z=\x$ and $\y'*\z'=\x$, hence
  $\y*\z=\y'*\z'$, but $*$ is a bijection so $\y=\y'$ (and
  $\z=\z'$). Similarly, $\b$ is a function.  Both $\a$ and $\b$ are
  defined on every $\x\in\U$ because $*$ is assumed to be surjective
  so that $\conv*(\x)$ always exists.  By the definition of $\a$ as
  a binary relation, $\a(\x*\y)=\x$ iff $\<\x*\y,\x\>\in\a$ iff
  $\exists_\z(\x*\z=\x*\y)$. The third statement is obviously true
  so the first is true as well. Similarly, $\b(\x*\y)=\y$.

  Finally, we will show $\a(\x)*\b(\x)=\x$.  By notational
  conventions only, $\<\x,\a(\x)\>\in\a$, so by the definition of
  $\a$, $\a(\x)*\z=\x$ for some $\z\in\U$ and similarly
  $\z'*\b(\x)=\x$ for some $\z'\in\U$. Consequently,
  $\x=\a(\x)*\z=\z'*\b(\x)$, but $*$ is bijective, hence
  $\a(\x)=\z'$ and $\z=\b(\x)$. Thus, $\<\U,*,\a,\b\>$ is a
  J\'onsson-Tarski algebra.
\end{proof}
Because of this connection between J\'onsson-Tarski algebras and
bijections between a set and its Cartesian squre, the reduct
$\<\U,*\>$ obtained from a J\'onsson-Tarski algebra $\<\U,*,\a,\b\>$
by deleting the two unary operations is sometimes called a ``Cantor
algebra'' (or even a ``J\'onsson-Tarski algebra''), axiomatized
by
\begin{align*}
  \x*\y=\u*\v\implies\x=\y\land\u=\v.
\end{align*}
\begin{prop}
  \label{JT=Qu}
  $\UUU=\<\U,*,\a,\b\>$ is a J\'onsson-Tarski algebra if and only if
  $\Re\U$ is a Qu-algebra such that $\a$ and $\b$ satisfy \eqref{Qu}
  and $\a,\b$ determine $*$ as follows:
  \begin{align*}
    \x*\y=\z\iff\x=\a(\z)\land\b(\z)=\y.
  \end{align*}
\end{prop}
\begin{proof} Let $\UUU=\<\U,*,\a,\b\>$ be a J\'onsson-Tarski
  algebra.  Consider the proper relation algebra
  $\mathfrak{B}=\Re\U$.  The unary operations $\a$ and $\b$ are also
  elements of $\BB$.  That $\a$ and $\b$ are unary operations on
  $\U$ (functions defined on every point in $\U$) is expressed in
  $\Re\U$ by the equations $1=\a\rp1=\b\rp1$, $\cona\rp\a\leq\id$,
  and $\conb\rp\b\leq\id$.  Note that $\a$ and $\b$ are surjective
  because $\a(\x*\x)=\x =\b(\x*\x)$ by \eqref{JT1} and \eqref{JT2}.
  That $\a$ and $\b$ are surjective functions is expressed by
  $\cona\rp\a=\id=\conb\rp\b$.  The equation $\cona\rp\b=1$ is
  equivalent to $\forall_{\x\y}\exists_\z\(\a(\z)=\x
  \land\b(\z)=\y\)$, which is certainly true because one can just
  take $\z=\x*\y$.  To prove $\a\rp\cona\bp\b\rp\conb\leq\id$, the
  Unicity Condition \eqref{U}, note first that it is equivalent to
  \begin{align*}
    \exists_\z(\a(\x)=\z=\a(\y))\land\exists_{\z'}(\b(\x)=\z'=\b(\y))
    \implies\x=\y.
  \end{align*}
  To prove this, assume the hypotheses, that there are $\z,\z'\in\U$
  such that $\a(\x)=\z=\a(\y)$ and $\b(\x)=\z'=\b(\y)$. Then by
  \eqref{JT3}, we have $\x=\a(\x)*\b(\x) = \a(\y)*\b(\y)=\y$.  This
  proves that $\Re\U$ is a Qu-algebra with $\a$ and $\b$ as the
  conjugated quasiprojections that satisfy \eqref{Qu}.

  Conversely, if $\Re\U$ is a Qu-algebra with $\a$ and $\b$ as the
  conjugated quasiprojections satisfying \eqref{Qu}, then,
  according to the interpretation in $\Re\U$ of the equations in
  \eqref{Qu}, $\a$ and $\b$ are functions defined on all of $\U$
  (unary operations on $\U$) and the required binary operation $*$
  on $\U$ can be defined as a ternary relation by
  \begin{align*}
    *=\{\<\x,\y,\z\>:\x=\a(\z)\land\b(\z)=\y\}
  \end{align*}
  To show that $*$ as a function on two inputs, assume $\<\x,\y,\z\>$
  and $\<\x,\y,\z'\>$ are both in $*$.  We must show $\z=\z'$. By the
  definition of $*$, we have $\x=\a(\z)$, $\b(\z)=\y$, $\x=\a(\z')$,
  and $\b(\z')=\y$, \ie, $\<\z,\x\>\in\a$, $\<\z,\y\>\in\b$,
  $\<\x,\z'\>\in\cona$, and $\<\y,\z'\>\in\conb$, from which it
  follows that $\<\z,\z'\>\in\a\rp\cona$ and
  $\<\z,\z'\>\in\b\rp\conb$, hence
  $\<\z,\z'\>\in\a\rp\cona\bp\b\rp\conb\leq\id$ by \eqref{Qu}, so
  $\z=\z'$ since $\id=\{\<\w,\w\>:\w\in\U\}$ in $\Re\U$.  From
  $1=\cona\rp\b$ it follows that for all $\x,\y\in\U$ there is some
  $\z\in\U$ such that $\x=\a(\z)\land\b(\z)=\y$, \ie, $\<\x,\y,\z\>$
  is in $*$ according to the definition of $*$, but, written in
  functional notation, this says $\x*\y=\z$. Thus, the domain of $*$
  as a binary operation is $\U^2$. It is easy to check that the
  algebra $\<\U,*,\a,\b\>$ satisfies the equations
  \eqref{JT1}--\eqref{JT3}, and is therefore a J\'onsson-Tarski
  algebra.
\end{proof}

\section{Relations on J\'onsson-Tarski algebras}
\label{s12}
In his presentation to the Berkeley seminar, Thompson used the
parenthetical notation of combinatory logic to describe groups of
functions acting on rooted infinite binary trees. In terms of
parenthetical notation, Thompson's group $\F$ is the group of
associative laws, $\T$ allows cyclic rearrangements, and $\VV$ allows
arbitrary rearrangements.  Later, Thompson discovered a topological
interpretation.  Topological representations of his groups are now the
most common way of briefly defining his groups.

The parenthetical notation for the generator called $\A$ in
\cite{MR1426438} and $\R$ in Brin's notes is $0(12)\mapsto(01)2$. This
function carries a tree called $0(12)$ to the tree called $(01)2$,
where $0$ is both the left branch of $0(12)$ and the left branch of
the left branch of $(01)2$, $1$ is the left branch of the right branch
of $0(12)$ and also the right branch of the left branch of $(01)2$,
and $2$ is the right branch of the right branch of $0(12)$ and the
right branch of $(01)2$. See Figure \ref{fog}, known as a ``paired
tree diagram''. It shows how three branches of an input tree on the
left should be cut off and grafted back onto the remaining finite tree
in a way described by the output tree on the right.

Every element of a J\'onsson-Tarski algebra $\UUU=\<\U,*,\a,\b\>$ can
be seen as the root of an infinite binary tree. The ``tree'' $\x\in\U$
has the ``tree'' $\a(\x)$ as its left branch and the ``tree'' $\b(\x)$
as its right branch.  Viewed this way, elements of Thompson's groups
are functions acting on $\U$. The unary operations $\a$ and $\b$ are
themselves members of the Thompson monoid $\M$.  In Figure \ref{fog1},
$0$, $1$, and $2$ are elements of the J\'onsson-Tarski algebra $\UUU$.
Arrows, colors, and labels have been added to the edges to signify the
action of $\a$ (red arrow, left branch) and $\b$ (blue arrow, right
branch) on the elements called $(01)2$ and $0(12)$. Figure \ref{fog1}
is simplified in Figure \ref{fog2} so that $0$, $1$, and $2$ occur
only once with no decorations suggesting they are trees.  All the
information in Figure \ref{fog2} is represented in the series-parallel
diagram shown in Figure \ref{fog3}, which includes the corresponding
J-algebraic term.
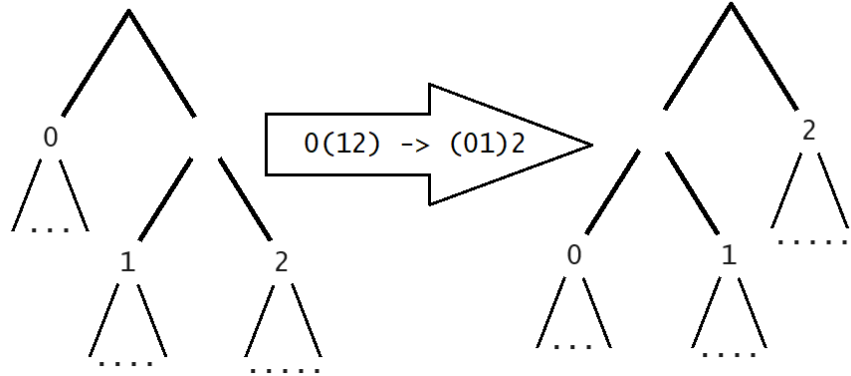
\begin{figure}
  \begin{tikzpicture}[scale=.5]
    \draw ( 8, 7) node[anchor=center] {$0(12)\mapsto(01)2$};
    \draw ( 0, 5) node[anchor=center] {0};
    \draw (16, 5) node[anchor=center] {2};
    \draw ( 2, 2) node[anchor=center] {1};
    \draw ( 6, 2) node[anchor=center] {2};
    \draw (10, 2) node[anchor=center] {0};
    \draw (14, 2) node[anchor=center] {1};
    \draw[line width=.6 mm] (   2,    8) -- ( 2/5, 28/5);
    \draw[line width=.6 mm] (  14,    8) -- (  12,    5);
    \draw[line width=.6 mm] (   2,    8) -- (   4,    5);
    \draw[line width=.6 mm] (  14,    8) -- (78/5, 28/5);
    \draw[line width=.6 mm] (18/5, 22/5) -- (12/5, 13/5);
    \draw[line width=.6 mm] (22/5, 22/5) -- (28/5, 13/5);
    \draw[line width=.6 mm] (58/5, 22/5) -- (52/5, 13/5);
    \draw[line width=.6 mm] (62/5, 22/5) -- (68/5, 13/5);
    \draw[line width=.4 mm] (-1/5, 23/5) -- ( -4/5, 17/5);
    \draw[line width=.4 mm] ( 1/5, 23/5) -- (  4/5, 17/5);
    \draw[line width=.4 mm] ( 9/5,  8/5) -- (  6/5,  2/5);
    \draw[line width=.4 mm] (11/5,  8/5) -- ( 14/5,  2/5);
    \draw[line width=.4 mm] (29/5,  8/5) -- ( 26/5,  2/5);
    \draw[line width=.4 mm] (31/5,  8/5) -- ( 34/5,  2/5);
    \draw[line width=.4 mm] (49/5,  8/5) -- ( 46/5,  2/5);
    \draw[line width=.4 mm] (51/5,  8/5) -- ( 54/5,  2/5);
    \draw[line width=.4 mm] (69/5,  8/5) -- ( 66/5,  2/5);
    \draw[line width=.4 mm] (71/5,  8/5) -- ( 74/5,  2/5);
    \draw[line width=.4 mm] (79/5, 23/5) -- ( 76/5, 17/5);
    \draw[line width=.4 mm] (81/5, 23/5) -- ( 84/5, 17/5);
    \draw ( 0,3) node[anchor=center] {.....};
    \draw (16,3) node[anchor=center] {.........};
    \draw ( 2,0) node[anchor=center] {.......};
    \draw ( 6,0) node[anchor=center] {.........};
    \draw (10,0) node[anchor=center] {.....};
    \draw (14,0) node[anchor=center] {.......};
  \end{tikzpicture}
  \caption{The action of $\R$ on trees}
  \label{fog}
\end{figure}
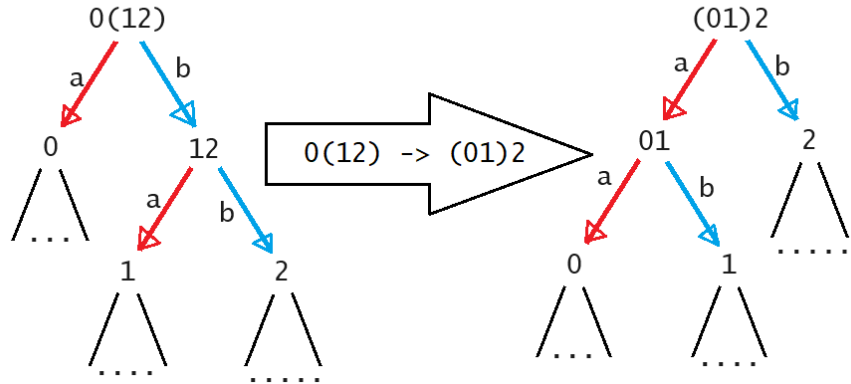
\begin{figure}
  \begin{tikzpicture}[scale=.5]
    \draw (  .8,7) node[anchor=center] {\a};
    \draw ( 3.2,7) node[anchor=center] {\b};
    \draw (12.8,7) node[anchor=center] {\a};
    \draw (15.2,7) node[anchor=center] {\b};
    \draw ( 2.8,4) node[anchor=center] {\a};
    \draw ( 5.2,4) node[anchor=center] {\b};
    \draw (10.8,4) node[anchor=center] {\a};
    \draw (13.2,4) node[anchor=center] {\b};
    \draw ( 8,7) node[anchor=center] {$0(12)\mapsto(01)2$};
    \draw ( 2,8) node[anchor=center] {0(12)};
    \draw (14,8) node[anchor=center] {(01)2};
    \draw ( 0,5) node[anchor=center] {0};
    \draw ( 4,5) node[anchor=center] {12};
    \draw (12,5) node[anchor=center] {01};
    \draw (16,5) node[anchor=center] {2};
    \draw ( 2,2) node[anchor=center] {1};
    \draw ( 6,2) node[anchor=center] {2};
    \draw (10,2) node[anchor=center] {0};
    \draw (14,2) node[anchor=center] {1};
    \draw[-stealth, line width=.6 mm, color=red ] ( 8/5, 37/5) -- ( 2/5, 28/5);
    \draw[-stealth, line width=.6 mm, color=red ] (68/5, 37/5) -- (62/5, 28/5);
    \draw[-stealth, line width=.6 mm, color=blue] (12/5, 37/5) -- (18/5, 28/5);
    \draw[-stealth, line width=.6 mm, color=blue] (72/5, 37/5) -- (78/5, 28/5);
    \draw[-stealth, line width=.6 mm, color=red ] (18/5, 22/5) -- (12/5, 13/5);
    \draw[-stealth, line width=.6 mm, color=blue] (22/5, 22/5) -- (28/5, 13/5);
    \draw[-stealth, line width=.6 mm, color=red ] (58/5, 22/5) -- (52/5, 13/5);
    \draw[-stealth, line width=.6 mm, color=blue] (62/5, 22/5) -- (68/5, 13/5);
    \draw[line width=.4 mm] (-1/5, 23/5) -- ( -4/5, 17/5);
    \draw[line width=.4 mm] ( 1/5, 23/5) -- (  4/5, 17/5);
    \draw[line width=.4 mm] ( 9/5,  8/5) -- (  6/5,  2/5);
    \draw[line width=.4 mm] (11/5,  8/5) -- ( 14/5,  2/5);
    \draw[line width=.4 mm] (29/5,  8/5) -- ( 26/5,  2/5);
    \draw[line width=.4 mm] (31/5,  8/5) -- ( 34/5,  2/5);
    \draw[line width=.4 mm] (49/5,  8/5) -- ( 46/5,  2/5);
    \draw[line width=.4 mm] (51/5,  8/5) -- ( 54/5,  2/5);
    \draw[line width=.4 mm] (69/5,  8/5) -- ( 66/5,  2/5);
    \draw[line width=.4 mm] (71/5,  8/5) -- ( 74/5,  2/5);
    \draw[line width=.4 mm] (79/5, 23/5) -- ( 76/5, 17/5);
    \draw[line width=.4 mm] (81/5, 23/5) -- ( 84/5, 17/5);
    \draw ( 0,3) node[anchor=center] {.....};
    \draw (16,3) node[anchor=center] {.........};
    \draw ( 2,0) node[anchor=center] {.......};
    \draw ( 6,0) node[anchor=center] {.........};
    \draw (10,0) node[anchor=center] {.....};
    \draw (14,0) node[anchor=center] {.......};
  \end{tikzpicture}
  \caption{Tree diagram of $\R$ in a J\'onsson-Tarski algebra}
  \label{fog1}
\end{figure}
\begin{figure}
  \begin{tikzpicture}[scale=.5]
    \draw ( 2  ,3.5) node[anchor=center] {\b};
    \draw ( 6  ,3.5) node[anchor=center] {\a};
    \draw ( 4  ,6.5) node[anchor=center] {\a};
    \draw (14  ,5.5) node[anchor=center] {\a};
    \draw (10  ,5.5) node[anchor=center] {\b};
    \draw (10.3,7.2) node[anchor=center] {\a};
    \draw (12  ,2.5) node[anchor=center] {\b};
    \draw ( 6  ,1.5) node[anchor=center] {\b};
    
    \draw ( -.4,3.9) node[anchor=center] {0(12)};
    \draw (4   ,2  ) node[anchor=center] {12};
    \draw (8   , .2) node[anchor=center] {2};
    \draw (8   ,4  ) node[anchor=center] {1};
    \draw (8   ,7.8) node[anchor=center] {0};
    \draw (12  ,6  ) node[anchor=center] {01};
    \draw (16.4,3.9) node[anchor=center] {(01)2};
    \draw[-stealth, line width=.6 mm, color=blue] ( 4/5,18/5) -- (16/5,12/5);
    \draw[-stealth, line width=.6 mm, color=red ] ( 4/5,22/5) -- (36/5,38/5);
    \draw[-stealth, line width=.6 mm, color=blue] (24/5, 8/5) -- (36/5, 2/5);
    \draw[-stealth, line width=.6 mm, color=red ] (76/5,22/5) -- (64/5,28/5);
    \draw[-stealth, line width=.6 mm, color=blue] (76/5,18/5) -- (44/5, 2/5);
    \draw[-stealth, line width=.6 mm, color=red ] (56/5,32/5) -- (44/5,38/5);
    \draw[-stealth, line width=.6 mm, color=red ] (24/5,12/5) -- (36/5,18/5);
    \draw[-stealth, line width=.6 mm, color=blue] (56/5,28/5) -- (44/5,22/5);
  \end{tikzpicture}
  \caption{Simplified tree diagram of $\R$ in a J\'onsson-Tarski
    algebra}
  \label{fog2}
\end{figure}
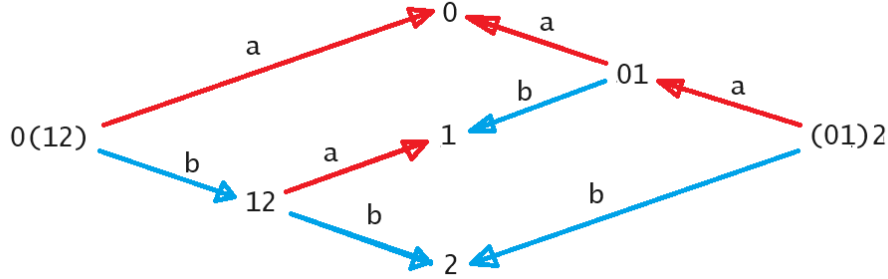
\begin{figure}
  \begin{tikzpicture}[scale=.5]
    \draw ( 1.8,2.5) node[anchor=center] {\b};
    \draw ( 2.4,4.5) node[anchor=center] {\b};
    \draw ( 6  ,4.5) node[anchor=center] {\a};
    \draw ( 4  ,6.5) node[anchor=center] {\a};
    \draw (14  ,5.5) node[anchor=center] {\a};
    \draw (13.8,3.6) node[anchor=center] {\a};
    \draw (10  ,4.5) node[anchor=center] {\b};
    \draw (10.3,7.2) node[anchor=center] {\a};
    \draw (12  ,2.5) node[anchor=center] {\b};
    \draw ( 6  ,1.5) node[anchor=center] {\b};
    
    \draw ( -.4,3.9) node[anchor=center] {0(12)};
    \draw ( 4  ,2  ) node[anchor=center] {12};
    \draw ( 4  ,4  ) node[anchor=center] {12};
    \draw ( 8  ,0.2) node[anchor=center] {2};
    \draw ( 8  ,4  ) node[anchor=center] {1};
    \draw ( 8  ,7.8) node[anchor=center] {0};
    \draw (12  ,4  ) node[anchor=center] {01};
    \draw (12  ,6  ) node[anchor=center] {01};
    \draw (16.4,3.9) node[anchor=center] {(01)2};
    \draw[-stealth, line width=.6 mm, color=blue] ( 4/5, 18/5) -- (16/5, 12/5);
    \draw[-stealth, line width=.6 mm, color=blue] ( 4/5,  4  ) -- (16/5,  4  );
    \draw[-stealth, line width=.6 mm, color=red ] ( 4/5, 22/5) -- (36/5, 38/5);
    \draw[-stealth, line width=.6 mm, color=blue] (24/5,  8/5) -- (36/5,  2/5);
    \draw[-stealth, line width=.6 mm, color=red ] (24/5,  4  ) -- (36/5,  4  );
    \draw[-stealth, line width=.6 mm, color=red ] (76/5, 22/5) -- (64/5, 28/5);
    \draw[-stealth, line width=.6 mm, color=red ] (76/5,  4  ) -- (64/5,  4  );
    \draw[-stealth, line width=.6 mm, color=blue] (76/5, 18/5) -- (44/5,  2/5);
    \draw[-stealth, line width=.6 mm, color=red ] (56/5, 32/5) -- (44/5, 38/5);
    \draw[-stealth, line width=.6 mm, color=blue] (56/5,  4  ) -- (44/5,  4  );
    \draw (8,-1) node[anchor=center]
    {$\R=\a\rp\con\a\rp\con\a\bp\b\rp\a\rp\con\b\rp\con\a\bp\b\rp\b\rp\con\b$};
  \end{tikzpicture}
  \caption{Series-parallel diagram of $\R$ in a J-algebra}
  \label{fog3}
\end{figure}
\begin{defn}
  \label{paths}
  Assume $\AA\in\JA\cup\RA$ and $\a,\b$ are elements of $\AA$.  Let
  $\P_{\a,\b}$ be the closure of $\a$ and $\b$ in $\AA$ under relative
  product $\rp$ and Boolean product $\bp$, \ie, $\P_{\a,\b}=
  \bigcup_{\n\in\omega}\Y_\n$ where $\Y_0=\{0,\id,\a,\b\}$, and
  \begin{align*}
    \Y_\n = \Y_{\n-1}\cup\{\x\rp\y:\x,\y\in\Y_{\n-1}\}
    \cup\{\x\bp\y:\x,\y\in\Y_{\n-1}\}
  \end{align*}
  for $\n>0$.  Elements of $\P_{\a,\b}$ are called {\bf paths} or {\bf
    branches}.
\end{defn}
Definition \ref{parnot} below is intended for application when $\a$
and $\b$ are functional because in that case we have the following
consequence of Prop.\ \ref{func}.
\begin{prop}
  \label{closed}
  If $\AA\in\JA\cup\RA$ and $\a,\b$ are functional elements of $\AA$
  then every path is functional.
\end{prop}
The passage from parenthetical notation to elements of an arbitrary
J-algebra or relation algebra $\AA$ with elements $\a,\b$ satisfying
\eqref{Qu} is made precise in the definition of the two operations
$\wedge$ and $\mapsto$.  The operation $\wedge$ combines two functions
that map finite sets into $\AA$ into a single function that maps the
union of the domains of the two input functions into $\AA$.  Any two
functions $\sigma,\tau$ that map finite sets into $\AA$ determine an
element of $\AA$ called $\sigma\mapsto\tau$, defined as $1$ if the
domains of $\sigma$ and $\tau$ are disjoint, and otherwise defined as
the Boolean product of $\sigma(\x)\rp\con{\tau(\x)}$ for every $\x$ in
the domain of both $\sigma$ and $\tau$.

For intuition regarding these operations, consider the case
$\AA=\Re\U$. If the domain of $\sigma$ were a subset of $\U$ then
$\sigma$ would say, ``to get from where you are to the point $\x$ in
my domain, follow the path (or climb the branch) $\sigma(\x)$.'' What
$\sigma$ says is true of $\y$ iff $\sigma(\x)(\y)=\x$ for every $\x$
in the domain of $\sigma$. The relation $\sigma\mapsto\tau$ is the
intersection of all the two-part paths that pass through points in the
intersection of the domains of $\sigma$ and $\tau$. Such a two-part
path consists of travel first along a $\sigma$-path (to a point in the
common domain) and then backward along a $\tau$-path.
\begin{defn}
  \label{parnot}
  Assume $\AA\in\JA\cup\RA$ and $\a,\b$ are elements of $\AA$.
  Suppose there are two finite sets $\X,\Y$ and two functions
  \begin{align*}
    \sigma\colon\X\to\AA,&&\tau\colon\Y\to\AA.
  \end{align*}
  Define a function
  \begin{align*}
    \sigma\wedge\tau\colon\X\cup\Y\to\AA
  \end{align*}
  for every $\u\in\X\cup\Y$ by
  \begin{align}
    \label{wedg}
    (\sigma\wedge\tau)(\u)=
    \begin{cases}
      \a\rp\sigma(\u) &\text{if}\quad \u\in\X\setminus\Y
      \\ \a\rp\sigma(\u)\bp\b\rp\tau(\u)&\text{if}\quad\u\in\X\cap\Y
      \\ \b\rp\tau(\u) &\text{if}\quad \u\in\Y\setminus\X
    \end{cases}
  \end{align}
  Define an element $\sigma\mapsto\tau$ of $\AA$ by
  \begin{align}
    \label{maps}
    \sigma\mapsto\tau=\prod_{\u\in\X\cap\Y}\sigma(\u)\rp\con{\tau(\u)},
  \end{align}
  and $\sigma\mapsto\tau=1$ if $\X\cap\Y=\emptyset$.
\end{defn}

\section{Generators of $\F,\T,\VV,\M$}
\label{s12a}
Assume $\AA\in\JA\cup\RA$ and \eqref{Qu} holds for elements $\a$ and
$\b$. Using Definition \ref{parnot} we will create elements of $\AA$
that generate homomorphic images of the Thompson groups and monoid.
Suppose $\x_0\neq\x_1$. Start with two initial functions
$\iota_0=\{\<\x_0,\id\>\}$ and $\iota_1=\{\<\x_1,\id\>\}$.  First,
combining these two initial functions in three ways using $\wedge$,
let
\begin{align*}
  \sigma_0&=\iota_0\wedge\iota_0, & \sigma_1&=\iota_0\wedge\iota_1, &
  \sigma_2&=\iota_1\wedge\iota_0.
\end{align*}
By Definition \ref{parnot},
\begin{align*}
  \sigma_0&=\{\<\x_0,\a\bp\b\>\}, & \sigma_1&=\{\<\x_0,\a\>,
  \<\x_1,\b\>\}, & \sigma_2&=\{\<\x_0,\b\>, \<\x_1,\a\>\}.
\end{align*}
The following four elements, given the names they have in Brin's
notes, are obtained by combining the $\iota$'s and $\sigma$'s with
$\mapsto$ and applying Definition \ref{parnot}.
\begin{align*}
  \K=\sigma_1\mapsto\iota_0
  &=\sigma_1(\x_0)\rp\con{\iota_0(\x_0)}=\a\rp\con\id=\a,
  \\ \L=\sigma_1\mapsto\iota_1
  &=\sigma_1(\x_1)\rp\con{\iota_1(\x_1)}=\b\rp\con\id=\b,
  \\ \U=\iota_0\mapsto\sigma_0
  &=\iota_0(\x_0)\rp\con{\sigma_0(\x_0)}=\id\rp\con{\a\bp\b}=\cona\bp\conb,
  \\ \P=\sigma_1\mapsto\sigma_2
  &=\prod_{\u\in\{\x_0,\x_1\}}\sigma_1(\u)\rp\con{\sigma_2(\u)}
  \\ &=\sigma_1(\x_0)\rp\con{\sigma_2(\x_0)}\bp\sigma_1(\x_1)
  \rp\con{\sigma_2(\x_1)} \\ &= \a\rp\conb \bp \b\rp\cona.
\end{align*}
Next we rewrite the results of these computations entirely in terms of
$\iota$'s and then delete the $\iota$'s, leaving only the subscripts,
which are then written in regular, non-subscript, size.  Also, we
delete $\wedge$ in favor of juxtaposition. For example,
$\K=\sigma_1\mapsto\iota_0= (\iota_0\wedge\iota_1)\mapsto
\iota_0=01\mapsto0$. The resulting abbreviations are
\begin{align*}
  \K&=01\mapsto0=\a, \\ \L&=01\mapsto1=\b,
  \\ \U&=0\mapsto00=\cona\bp\conb,
  \\ \P&=01\mapsto10=\a\rp\conb\bp\b\rp\cona.
\end{align*}
The actions of $K$, $L$, $U$,and $P$ can be read directly from this
abbreviated notation. $\K$ maps an infinite rooted binary tree to its
left branch and $\L$ maps it to its right branch.  $\U$ maps a tree to
a new tree whose left and right branches coincide with the input
tree. $\P$ produces a new tree from an input tree by interchanging the
two branches.

Assume $\{\x_0,\x_1,\x_2,\x_4\}$ is a set of four distinct points. As
initial functions use $\iota_0=\{\<\x_0,\id\>\}$,
$\iota_1=\{\<\x_1,\id\>\}$, $\iota_2=\{\<\x_2,\id\>\}$, and
$\iota_3=\{\<\x_3,\id\>\}$.  In general, for a given finite set $\X$,
its initial functions are those of the form $\{\<\x,\id\>\}$, each of
which says, ``To get to the tree $\x$, the sole tree in my domain, do
nothing. You are already there.''  Consider the relations
$\sigma\mapsto\tau$ for $\sigma,\tau$ in the closure of
$\{\{\<\x,\id\>\}:\x\in\X\}$ under $\wedge$. It can easily be seen
that if the domain of $\tau$ is a subset of the domain of $\sigma$,
then $\sigma\mapsto\tau$ is functional.  The following results are
obtained by applying Definition \ref{parnot} to various
$\wedge$-combinations of the four initial functions, deleting
$\wedge$'s and $\iota$'s, and restoring subscripts to regular size,
\eg, $\iota_3\wedge((\iota_0\wedge\iota_1)\wedge\iota_2)=3((01)2)$.
Recall that $\a^2=\a\rp\a$, $\b^3=\b\rp\b\rp\b$, \etc.
\begin{align*}
  0(12)&=\big\{\<\x_0,\a\>,\<\x_1,\b\rp\a\>,\<\x_2,\b^2\>\big\},
  \\ (01)2&=\big\{\<\x_0,\a^2\>,\<\x_1,\a\rp\b\>,\<\x_2,\b\>\big\},
  \\ 1(20)&=\big\{\<\x_0,\b^2\>,\<\x_1,\a\>,\<\x_2,\b\rp\a\>\big\},
  \\ 1(02)&=\big\{\<\x_0,\b\rp\a\>,\<\x_1,\a\>,\<\x_2,\b^2\>\big\},
  \\ 3(0(12))&=\big\{\<\x_0,\b\rp\a\>,\<\x_1,\b^2\rp\a\>,
  \<\x_2,\b^3\>,\<\x_3,\a\>\big\},
  \\ 3((01)2)&=\big\{\<\x_0,\b\rp\a^2\>,\<\x_1,\b\rp\a\rp\b\>,
  \<\x_2,\b^2\>,\<\x_3,\a\>\big\}.
\end{align*}
With these functions we define six more elements of $\AA$ using the
same names that were given to them in Brin's notes and
\cite{MR1426438}. Both sources copied Thompson's notation.  The
element called $\A$ in \cite{MR1426438} is called $\R$ in Brin's
notes.
\begin{align*}
  \R=\A&=0(12)\mapsto(01)2
  =\a\rp\con\a^2\bp\b\rp\a\rp\con\b\rp\con\a\bp\b^2\rp\con\b,
  \\ \B&=3(0(12))\mapsto3((01)2)=\a\rp\cona
  \bp\b\rp\a\rp\cona^2\rp\conb
  \bp\b^2\rp\a\rp\conb\rp\cona\rp\conb\bp\b^3\rp\conb^2,
  \\ \C&=0(12)\mapsto1(20)
  =\a\rp\con\b^2\bp\b\rp\a\rp\con\a\bp\b^2\rp\con\a\rp\con\b,
  \\ \pi_0&=0(12)\mapsto1(02)
  =\a\rp\cona\rp\conb\bp\b\rp\a\rp\cona\bp\b^2\rp\conb^2,
  \\ \P_0&=(01)2\mapsto(10)2
  =\a^2\rp\conb\rp\cona\bp\a\rp\b\rp\cona^2\bp\b\rp\conb,
  \\ \R_0&=(0(12))3\mapsto((01)2)3=\a^2\rp\cona^3
  \bp\a\rp\b\rp\a\rp\conb\rp\cona^2
  \bp\a\rp\b^2\rp\conb\rp\cona\bp\b\rp\conb.
\end{align*}
Assume $\AA\in\JA\cup\RA$ and \eqref{Qu} holds for $\a$ and $\b$.  By
Prop.\ \ref{perms}, the elements $\A=\R$, $\B$, $\C$, $\pi_0$, $\P$,
$\P_0$, and $\R_0$ are permutational, while $\K$, $\L$, $\U$, and
$\con\U$ are only functional.  Combinations of these elements generate
homomorphic images of Thompson's groups $\F$, $\T$, and $\VV$ and
monoid $\M$.

Assuming $\AA$ is sufficiently free, specifically, that no equations
hold in $\AA$ other than consequences of either the relation algebra
or J-algebra axioms together with the equations in \eqref{Qu}, we can
conclude by Theorem \ref{Tgroups} below that $\AA$ will contain actual
\emph{copies} of $\F$, $\T$, $\VV$, and $\M$ instead of possibly
proper homomorphic images.  In the group $\gc{Pm}(\AA)$, $\{\A,\B\}$
generates a copy of $\F$, $\{\A,\B,\C\}$ generates a copy of $\T$, and
a copy of $\VV$ is generated by both $\{\A,\B,\C,\pi_0\}$ and
$\{\P,\R,\P_0,\R_0\}$. The latter generating set of $\VV$ is mentioned
in Brin's notes.  In the monoid $\gc{Fn}(\AA)$, a copy of $\M$ is
generated by both $\{\P,\R,\U,\K\}$ and $\{\P_0,\R_0,\U,\K\}$, as was
also pointed out in Brin's notes.

\section{Finite presentations of $\F$, $\T$, and $\VV$}
\label{s13}
Thompson's finite presentations for $\F$, $\T$, and $\VV$ are covered
in detail by Cannon, Floyd, and Parry \cite{MR1426438}. Those
presentations, repeated below, are in the notation used by Thompson in
his hand-written notes, reproduced in \cite{MR1426438}.  The
commutator in a group may defined as $[\X,\Y]=\X\Y\X^{-1}\Y^{-1}$ or
$[\X,\Y]=\X^{-1}\Y^{-1}\X\Y$. Either way, the assertion that the
commutator of two elements is the identity element is simply a way of
saying that the two elements commute.

When relations to be satisfied by the generators in a presentation are
stated in group-theoretic notation, juxtaposition denotes $\circ$ (the
composition of functions), a superscript $-1$ indicates the inverse of
a function, and $1$ denotes the identity function.  The corresponding
algebraic notation uses $\rp$, $\con\blank$, and $\id$ in their place.
Functions are composed right-to-left according to the standard usage
of $\circ$.  Relation-algebraic notation involves an order-reversal to
account for this. For example, if the element $\X_2$ in Thompson's
group $\F$ is represented as a function acting on a set containing an
element $\x$, then $\x$ is mapped by $\X_2$ to $\X_2(x)$. By the
definition $\X_2=\A^{-1}\B\A$ we then have $\X_2=\A|\B|\A^{-1}$
because
\begin{align*}
  \X_2(\x)&=(\A^{-1}\B\A)(\x) =(\A^{-1}\circ\B\circ\A)(\x)
  \\ &=\A^{-1}(\B(\A(\x))) =(\A|\B|\A^{-1})(\x).
\end{align*}
In a J-algebra or relation algebra, the element $\X_2$ is defined by
$\X_2=\A\rp\B\rp\con\A$.  The relations and definitions below are
written in both group-theoretic notation (juxtaposition means $\circ$)
and in relation-algebraic notation.
\begin{itemize}
\item[($\F$)] The presentation of $\F$ has two generators $\A$ and
  $\B$, and two relations \eqref{ta1} and \eqref{ta2}.
\item[($\T$)] The presentation of $\T$ has three generators $\A$,
  $\B$, and $\C$, and six relations \eqref{ta1}--\eqref{ta6}.
\item[($\VV$)] The presentation of $\VV$ has four generators $A$, $B$,
  $C$, and $\pi_0$, and 14 relations \eqref{ta1}--\eqref{ta14}.
\end{itemize}
The following elements are defined in \cite[pp.\ 218, 236,
  241]{MR1426438}.
\begin{align*}
  \X_1&=\B & \C_1&=\C
  \\ \X_2&=\A^{-1}\B\A=\A\rp\B\rp\con\A&\X_3&=\A^{-2}\B\A^2=\A^2\rp\B\rp\con\A^2
  \\ \C_2&=\A^{-1}\C\B=\B\rp\C\rp\con\A&\C_3&=\A^{-2}\C\B^2=\B^2\rp\C\rp\con\A^2
  \\ \pi_1&=\C_2^{-1}\pi_0\C_2=\C_2\rp\pi_0\rp\con\C_2
  \\ \pi_2&=\A^{-1}\pi_1\A=\A\rp\pi_1\rp\con\A&\pi_3
  &=\A^{-2}\pi_1\A^2=\A^2\rp\pi_1\rp\con\A^2
\end{align*}
The fourteen relations {\eqref{ta1}--\eqref{ta14}},
\begin{align}
  \text{in a group:}&&\text{in a J-algebra:}\notag \\
  \label{ta1}
  [\A\conv\B,\X_2]&=1&[\con\B\rp\A,\X_2]&=\id \\
  \label{ta2}
  [\A\conv\B,X_3]&=1&[\con\B\rp\A,\X_3]&=\id \\
  \label{ta3}
  \C&=\B\C_2&\C&=\C_2\rp\B \\
  \label{ta4}
  \C_2\X_2&=\B\C_3&\X_2\rp\C_2&=\C_3\rp\B \\
  \label{ta5}
  \C\A&=\C_2^2&\A\rp\C&=\C_2^2 \\
  \label{ta6}
  \C^3&=1&\C^3&=\id \\
  \label{ta7} \pi_1^2&=1&\pi_1^2&=\id \\
  \label{ta8}
  \pi_1\pi_3&=\pi_3\pi_1&\pi_3\rp\pi_1&=\pi_1\rp\pi_3 \\
  \label{ta9}
  (\pi_2\pi_1)^3&=1&(\pi_1\rp\pi_2)^3&=\id \\
  \label{ta10}
  \X_3\pi_1&=\pi_1X_3&\pi_1\rp\X_3&=\X_3\rp\pi_1 \\
  \label{ta11}
  \pi_1\X_2&=\B\pi_2\pi_1&\X_2\rp\pi_1 &=\pi_1\rp\pi_2\rp\B \\
  \label{ta12}
  \pi_2\B&=\B\pi_3&\B\rp\pi_2&=\pi_3\rp\B \\
  \label{ta13}
  \pi_1\C_3&=C_3\pi_2&\C_3\rp\pi_1&=\pi_2\rp\C_3 \\
  \label{ta14}
  (\pi_1\C_2)^3&=1&(\C_2\rp\pi_1)^3&=\id
\end{align}
Bleak and Quick \cite{MR3737287} found two smaller presentations of
$\VV$, one with three generators and eight relations and another with
only two generators and seven relations. The latter presentation has
generators $\u$ and $\v$ from \cite[Th.\,1.3]{MR3737287}, translated
here into parenthetical notation and algebraic notation from their
tree diagrams on \cite[p.\ 1407]{MR3737287}.
\begin{align*}
  \u&=(01)(2(34))\mapsto(10)(4(23)) \\ &=\a^2\rp\conb\rp\cona \bp
  \a\rp\b\rp\cona^2 \bp \b\rp\a\rp\cona\rp\conb^2 \bp
  \b^2\rp\a\rp\conb^3 \bp \b^3\cona\rp\conb
  \\ \v&=(01)(23)\mapsto(03)(12) \\ &=\a^2\rp\cona\rp\cona \bp
  \a\rp\b\rp\cona\rp\conb \bp \b\rp\a\rp\conb^2 \bp  \b^2\rp\conb\rp\cona
\end{align*}
In \cite[\SS2.1]{MR3737287} the generators are $\t_{00,01}$ and
$\t_{1,00}$, while $\t_{01,10,11}$ is the product of two of them,
namely, $\t_{01,10}$ and $\t_{01,11}$.  When interpreted as an
operator on trees, $\t_{00,01}$ takes the left branch of the left
branch and interchanges it with the right branch of the left
branch. The $0$ and $1$ match up with $\a$ and $\b$, respectively.
The three generators of $\VV$ from \cite[Th.\,1.2]{MR3737287} are
\begin{align}
  \t_{00,01}&=(01)2\mapsto(10)2 =\a\rp\a\rp\conb\rp\cona \bp
  \a\rp\b\rp\cona\rp\cona \bp \b\rp\conb\notag
  \\ \t_{01,10,11}&=(01)(23)\mapsto(03)(12) =\a^2\rp\cona^2 \bp
  \a\rp\b\rp\cona\rp\conb \bp \b\rp\a\rp\conb^2 \bp
  \b^2\rp\conb\rp\cona \notag \\ \t_{1,00}&=(01)2\mapsto(21)0
  =\a^2\rp\conb \bp \a\rp\b\rp\conb\rp\cona \bp \b\rp\cona^2 \notag
\end{align}

\section{Infinite presentation of $\M$}
\label{s13a}
The infinite presentation of $\M$ presented here comes from Brin's
notes, where it is proved that there is a finite subset of the
relations from which all the other relations can be deduced. In these
relations, $\x$ and $\y$ are arbitrary elements of $\M$, but the
relations are proved here for arbitrary functional elements of $\AA$.
Only two special cases of the concept of `deferment' from Brin's notes
are needed: for any $\x$, let $\x_0= \a\rp\x\rp\cona \bp \b\rp\conb$
and $\x_1= \a\rp\cona \bp \b\rp\x\rp\conb$.  The generators are $\K$,
$\L$, $\U$, $\P_0$, and $\R_0$. Composition proceeds in Brin's notes
from left to right, so the order-reversal required in writing
algebraic versions of the relations in the presentations of $\F$,
$\T$, and $\VV$ is not required. The relations on the generators come
in five groups.
\begin{enumerate}
\item
  Invertiblity relations: $\P\rp\P=\id$, and $(\P\rp\R)^3=(\R\rp\P)^3=\id$.
\item
  Commutativity relations: $\x_0\rp\y_1=\y_1\rp\x_0$ for all $\x,\y$.
\item
  Splitting relations: $\x\rp\U=\U\rp\x_0\rp\x_1$ for all $\x$.
\item
  Reconstruction relations: $\x=\U\rp\x_0\rp\x_1\rp\K_0\rp\L_1$ for
  all $\x,\y$.
\item
  Rewriting relations:
  \begin{align*}
    \U\rp\K&=\id,&\P_0\rp\K\rp\K&=\K\rp\L,&\R_0\rp\K\rp\K\rp\K&=\K\rp\K,
    \\ \U\rp\L&=\id,&\P_0\rp\K\rp\L&=\K\rp\K,
    &\R_0\rp\K\rp\K\rp\L&=\K\rp\L\rp\K,
    \\ &&\P_0\rp\L&=\L,&\R_0\rp\K\rp\L&=\K\rp\L\rp\L,
    \\ &&&&\R_0\rp\L&=\L.
  \end{align*}
\end{enumerate}
Thompson proved that this infinite presentation can be reduced to a
finite one. Details are worked out in Brin's notes, based on two talks
by Thompson.

\section{$\F$, $\T$, $\VV$, and $\M$ in relation algebras
  and J-algebras}
\label{s14}
Suppose $\AA$ is an algebra that is free over the variety of algebras
obtained by supplementing the operations and axioms of $\JA$ with two
new constants, $\a$ and $\b$, and the equations in \eqref{Qu} are
regarded as additional axioms. Then all the homomorphisms in the
following theorem are isomorphisms. The monoid of functional elements
of $\AA$ contains a copy of $\M$ and the group of permutational
elements of $\AA$ contains copies of $\F$, $\T$, and $\VV$. For an
arbitrary algebra with quasiprojections, one can define elements that
behave like generators of Thompson's groups and monoid, \ie, they
satisfy all the relations required by the presentations, but that is
only enough to conclude that they generate homomorphic images.
\begin{thm}
  \label{Tgroups}
  Assume $\AA$ is J-algebra or a relation algebra and that \eqref{Qu}
  holds for elements $\a$ and $\b$ of $\AA$, \ie,
  \begin{align*}
    \id=\cona\rp\a=\conb\rp\b=\a\rp\cona\bp\b\rp\conb\;\mand\;
    1=\cona\rp\b=\a\rp1=\b\rp1.
  \end{align*}
  Define elements of $\AA$ as follows.
  \begin{align*}
    \K&=01\mapsto0=\a
    \\ \L&=01\mapsto1=\b
    \\ \U&=0\mapsto00=\cona\bp\conb
    \\ \P&=01\mapsto10=\a\rp\conb\bp\b\rp\cona
    \\ \P_0&=(01)2\mapsto(10)2=\a\rp\P\rp\cona\bp\b\rp\conb
    \\ \A=\R&=0(12)\mapsto(01)2
    =\a\rp\cona^2 \bp \b\rp\a\rp\conb\rp\cona \bp \b^2\rp\conb
    \\ \R_0&=(0(12))3\mapsto((01)2)3=\a\rp\R\rp\cona\bp\b\rp\conb
    \\ \B&=0(1(23))\mapsto0((12)3)=\a\rp\cona \bp \b\rp\A\rp\conb
    \\ \C&=0(12)\mapsto1(20)=\a\rp\conb^2 \bp \b\rp\a\rp\cona \bp
    \b^2\rp\cona\rp\conb
    \\ \pi_0&=0(12)\mapsto1(02)=\a\rp\cona\rp\conb \bp \b\rp\a\rp\cona
    \bp \b^2\rp\conb^2
  \end{align*}
  Then
  \begin{enumerate}
  \item
    $\K$, $\L$, $\U$, and $\con\U$ are functional elements of $\AA$.
  \item
    $\P$, $\P_0$, $\R$, $\R_0$, $\A$, $\B$, $\C$, and $\pi_0$ are
    permutational elements of $\AA$.
  \item
    The relations \eqref{ta1}--\eqref{ta14} hold in the group
    $\gc{Pm}(\AA)$.
  \item
    The subgroup of $\gc{Pm}(\AA)$ generated by $\{\A,\B\}$ is a
    homomorphic image of the Thompson group $\F$.
  \item
    The subgroup of $\gc{Pm}(\AA)$ generated by $\{\A,\B,\C\}$ is a
    homomorphic image of the Thompson group $\T$.
  \item
    The subgroup of $\gc{Pm}(\AA)$ generated by $\{\A,\B,\C,\pi_0\}$
    is a homomorphic image of the Thompson group $\VV$.
  \item
    $\{\P_0,\R_0,\K,\U\}$ and $\{\P,\R,\K,\U\}$ generate the same
    submonoid of $\gc{Fn}(\AA)$.
  \item
    The relations for $\M$ hold in the monoid $\gc{Fn}(\AA)$.
  \item
    The submonoid of $\gc{Fn}(\AA)$ generated by $\{\P_0,\R_0,\K,\U\}$
    is a homomorphic image of the Thompson monoid $\M$.
  \end{enumerate}
\end{thm}
\begin{proof}
  Parts (i) and (ii) hold by Prop.\ \ref{perms}.  For part (iii),
  relations \eqref{ta1} and \eqref{ta2} hold by Prop.\ \ref{presF} and
  \eqref{ta3}--\eqref{ta6} hold by Prop.\ \ref{presT}. Relations
  \eqref{ta7}--\eqref{ta14} can be proved similarly: this is left as an
  exercise for the interested reader.  Parts (iv), (v), and (vi) follow
  from part (iii).  Part (vii) is proved in Prop.\ \ref{same}.  Part
  (viii) is proved in Prop.\ \ref{presM}.  Part (ix) follows from part
  (viii).
\end{proof}

\section*{{\bf Part II.}}
Throughout Part II, consisting of \SS\ref{s16}--\SS\ref{s24}, the
blanket assumptions for Definitions \ref{d13}--\ref{FnPm} and
Props.\ \ref{p1}--\ref{presM} are that
$\AA=\<A,\bp,0,1,\rp,\con\blank,\id\>$ is a J-algebra containing
elements $\a$, $\b$, $\c$, $\d$, $\e$, $\p$, $\q$, $\r$, $\s$, $\u$,
$\v$, $\w$, $\x$, $\y$, and $\z$.  Starting in \SS\ref{s19}, two
elements $\a$ and $\b$ are chosen to remain fixed so that new binary
operations $\nabla$, $\otimes$, and $\fkc{}{}$ can be defined using
$\a,\b$ as parameters.
\section{Consequences of the J-algebra axioms}
\label{s16}
\begin{defn}
  \label{d13}
  Relations $\leq$, $\geq$, $<$, and $>$ are defined by
  \begin{align*}
    \x\leq\y&\iff\y\geq\x\iff\x\bp\y=\x,
    \\ \x<\y&\iff\y<\x\iff\x\bp\y=\x\neq\y.
  \end{align*}
\end{defn}
\begin{prop}
  \label{p1}
  The relations $\leq$ and $\geq$ are partial orderings.
\end{prop}
\begin{proof}
  A partial ordering is a reflexive, transitive, and antisymmetric
  relation. By Definition \ref{d13}, $\x\leq\x$ is equivalent to
  $\x\bp\x=\x$, which holds by \eqref{idem}. Thus, $\leq$ is
  reflexive. For transitivity, assume $\x\leq\y$ and $\y\leq\z$.  Then
  $\x\bp\y=\x$ and $\y\bp\z=\y$ by Definition \ref{d13}.  These two
  equations and \eqref{bassoc} imply $\x\bp\z = (\x\bp\y)\bp\z =
  \x\bp(\y\bp\z) = \x\bp\y = \x$, \ie, $\x\leq\z$.  For antisymmetry,
  assume $\x\leq\y$ and $\y\leq\x$, \ie, $\x\bp\y=\x$ and
  $\y\bp\x=\y$. These two equations and \eqref{comm} imply
  $\x=\x\bp\y=\y\bp\x=y$. Thus, $\leq$ is a partial ordering.  The
  converse of any partial ordering is a partial ordering, so $\geq$ is
  also a partial ordering.
\end{proof}
\begin{prop}
  \label{p2}
  $\x\bp\y\leq\y$ and $\x\bp\y\leq\x$.
\end{prop}
\begin{proof}
  The inclusion $\x\bp\y\leq\y$ is equivalent to
  $(\x\bp\y)\bp\y=\x\bp\y$. By \eqref{bassoc} and \eqref{idem},
  $(\x\bp\y)\bp\y=\x\bp(\y\bp\y)=\x\bp\y$, so $\x\bp\y\leq\y$. The
  other inclusion follows from this by \eqref{comm}.
\end{proof}
\begin{prop}
  \label{p3}
  If $\x\leq\y$ and $\x\leq\z$ then $\x\leq\y\bp\z$.
\end{prop}
\begin{proof}
  Assume $\x\leq\y$ and $\x\leq\z$, \ie, $\x\bp\y=\x$ and
  $\x\bp\z=\x$. Then $\x\leq\y\bp\z$ because
  \begin{align*}
    \x\bp(\y\bp\z)&=(\x\bp\x)\bp(\y\bp\z)&&\text{\eqref{idem}}
    \\ &=(\x\bp\y)\bp(\x\bp\z)&&\text{\eqref{bassoc} four times and
      \eqref{comm} once} \\ &=\x\bp\x&&\text{$\x\bp\y=\x$ and
      $\x\bp\z=\x$} \\ &=\x&&\text{\eqref{idem}}
  \end{align*}
\end{proof}
\begin{prop}
  \label{p4}
  If $\x\leq\y$ and $\u\leq\v$ then $\x\bp\u\leq\y\bp\v$.
\end{prop}
\begin{proof}
  Assume $\x\leq\y$ and $\u\leq\v$. Then $\x\bp\y=\x$ and $\u\bp\v=\u$,
  hence $\x\bp\u\leq\y\bp\v$ because
  \begin{align*}
    (\x\bp\u)\bp(\y\bp\v) &=(\x\bp\y)\bp(\u\bp\v)
    &&\text{\eqref{bassoc} and \eqref{comm}} \\ &=\x\bp\u.
  \end{align*}
\end{proof}
\begin{prop}
  \label{p5}
  $\con0=0$, $\con1=1$, $\con\id=\id$, $0\rp\x=0$, and $\id\rp\x=\x$.
\end{prop}
\begin{proof}
  We use only axioms \eqref{comm}, \eqref{id}, \eqref{6}, \eqref{5},
  \eqref{4}, \eqref{zero}, and \eqref{one}.  Each equality is labeled
  with the axiom that justifies it, with three exceptions, where we
  use the previously proved facts that $\con0=0$ (twice in the fourth
  line) and $\con\id=\id$ (once in the fifth line).
  \begin{align*}
    \con0&\overset{\eqref{zero}}=\con{0\bp\con0}
    \overset{\eqref{5}}=\con0\bp\con{\con0}
    \overset{\eqref{6}}=\con0\bp0 \overset{\eqref{comm}}=0\bp\con0
    \overset{\eqref{zero}}=0 \\ 1&\overset{\eqref{6}}=\con{\con1}
    \overset{\eqref{one}}=\con{\con1\bp1}
    \overset{\eqref{4}}=\con{\con1}\bp\con1
    \overset{\eqref{6}}=1\bp\con1 \overset{\eqref{comm}}=\con1\bp1
    \overset{\eqref{one}}=\con1
    \\ \id&\overset{\eqref{6}}=\con{\con{\id}}
    \overset{\eqref{id}}=\con{\con{\id}\rp\id}
    \overset{\eqref{5}}=\con\id\rp\con{\con\id}
    \overset{\eqref{6}}=\con\id\rp\id \overset{\eqref{id}}=\con\id
    \\ 0\rp\x&\overset{\eqref{6}}=\con{\con{0\rp\x}}
    \overset{\eqref{5}}=\con{\con\x\rp\con0} =\con{\con\x\rp0}
    \overset{\eqref{norm}}=\con0=0
    \\ \id\rp\x&\overset{\eqref{6}}=\con{\con{\id\rp\x}}
    \overset{\eqref{5}}=\con{\con\x\rp\con{\id}} =\con{\con\x\rp\id}
    \overset{\eqref{id}}=\con{\con\x} \overset{\eqref{6}}=\x
  \end{align*}
\end{proof}
\begin{prop}
  \label{p6}
  If $\x\leq\y$ then $\con\x\leq\con\y$, $\x\rp\z\leq\y\rp\z$, and
  $\z\rp\x\leq\z\rp\y$.
\end{prop}
\begin{proof}
  Assume $\x\leq\y$, \ie, $\x\bp\y=\x$. Then, by \eqref{4},
  $\con\x\bp\con\y=\con{\x\bp\y}=\con\x$, so $\con\x\leq\con\y$.  From
  $\x\bp\y=\x$ we get $(\x\bp\y)\rp\z=\x\rp\z$. By \eqref{mon},
  $(\x\bp\y)\rp\z\leq\y\rp\z$, so $\x\rp\z\leq\y\rp\z$.  By applying
  these two principles, starting with $\x\leq\y$, we first get
  $\con\x\leq\con\y$, then $\con\x\rp\con\z\leq\con\y\rp\con\z$, then
  $\con{\con\x\rp\con\z}\leq\con{\con\y\rp\con\z}$, and finally
  $\z\rp\x\leq\z\rp\y$ by \eqref{5} and \eqref{6}.
\end{proof}
\begin{prop}
  \label{p7}
  $(\u\bp\v)\rp(\x\bp\y)\leq\u\rp\x\bp\v\rp\y$
\end{prop}
\begin{proof}
  We have $\u\bp\v\leq\u$ and $\x\bp\y\leq\x$ by Prop.\ \ref{p2}, so
  $(\u\bp\v)\rp(\x\bp\y)\leq\u\rp(\x\bp\y)\leq\u\rp\x$ by
  Prop.\ \ref{p6}, and, similarly, $(\u\bp\v)\rp(\x\bp\y)\leq\v\rp\y$,
  so the conclusion follows by Prop.\ \ref{p4}.
\end{proof}
\begin{prop}
  \label{p8}
  $\x\rp\y\bp\z\leq(\z\rp\con\y\bp\x)\rp\y$ and
  $\x\rp\y\bp\z\leq\x\rp(\y\bp\con\x\rp\z)$.
\end{prop}
\begin{proof}
  By \eqref{rot} and Prop.\ \ref{p2},
  $\x\rp\y\bp\z\leq(\z\rp\con\y\bp\x)\rp(\y\bp\con\x\rp\z)$. Both
  inclusions follow from this by Prop.\ \ref{p2} and Prop.\ \ref{p6}.
\end{proof}
\begin{prop}
  \label{p9}
  $\x\leq\x\rp1$ and $\x\leq1\rp\x$.
\end{prop}
\begin{proof}
  We have $\id\leq1$ by \eqref{one}, hence $\x=\x\rp\id\leq\x\rp1$ and
  $\x=\id\rp\x\leq1\rp\x$ by \eqref{id}, Prop.\ \ref{p5}, and
  Prop.\ \ref{p6}.
\end{proof}
References to axioms \eqref{bassoc}--\eqref{rot} and
Props.\ \ref{p1}--\ref{p9} and \ref{f-dist} will often be indirect or
omitted, according to the conventions that
\begin{itemize}
\item
  \eqref{bassoc}--\eqref{2}, \eqref{zero}--\eqref{norm}, and
  Props.\ \ref{p1}--\ref{p4} may be used with no explicit reference,
\item
  ``id'' (identity) refers to \eqref{id} or the last part of
  Prop.\ \ref{p5},
\item
  ``mon'' (monotonicity) refers to \eqref{mon}, Prop.\ \ref{p2},
  Prop.\ \ref{p3}, Prop.\ \ref{p4}, Prop.\ \ref{p6}, or
  Prop.\ \ref{p7},
\item
  ``rot'' (rotation) refers to \eqref{rot} or Prop.\ \ref{p8},
  combined with ``mon'',
\item
  ``assoc'' (associativity) refers to \eqref{2},
\item
  ``con'' (converse) refers to \eqref{6}, \eqref{5}, \eqref{4}, or one
  of the first three parts of Prop.\ \ref{p5}, combined with
  \eqref{2},
\item
  ``func~dist'' (functional elements distribute) refers to
  Prop.\ \ref{f-dist} below.
\end{itemize}
Proofs of equations are often a sequence of inclusions that start with
one side of the equation, pass through the other side, and return to
the first side, followed by an implicit reference to the antisymmetry
of $\leq$.
\begin{prop}
  \label{p10}
  $\x\leq\x\rp\con\x\rp\x$ and $1\rp\x\rp1=1\rp\con\x\rp1$.
\end{prop}
\begin{proof}
  For the first part,
  \begin{align*}
    \x&=\id\rp\x\bp\x&&\text{\eqref{idem}, id}
    \\ &\leq(\id\bp\x\rp\con\x)\rp\x&&\text{rot}
    \\ &\leq\x\rp\con\x\rp\x&&\text{mon} \intertext{For the second
      part, use the first part in the first step.}
    1\rp\x\rp1&\leq1\rp(\x\rp\con\x\rp\x)\rp1 &&\text{mon}
    \\ &\leq(1\rp\x)\rp\con\x\rp(\x\rp1)&&\text{assoc}
    \\ &\leq1\rp\con\x\rp1&&\text{mon}
  \end{align*}
  The opposite inclusion follows from this by substituting $\con\x$
  for $\x$ and invoking \eqref{6}.
\end{proof}
\begin{prop}
  \label{cyc1}
  $(\y\rp\z\bp\x)\rp1=(\x\rp\con\z\bp\y)\rp1$ and
  $1\rp(\y\rp\z\bp\x)=1\rp(\x\rp\con\z\bp\y)$.
\end{prop}
\begin{proof}
  The proof of the first part in one direction is
  \begin{align*}
    (\y\rp\z\bp\x)\rp1&\leq(\x\rp\con\z\bp\y)\rp\z\rp1&&\text{rot}
    \\ &\leq(\x\rp\con\z\bp\y)\rp1\rp1&&\text{mon}
    \\ &\leq(\x\rp\con\z\bp\y)\rp(1\rp1)&&\text{assoc}
    \\ &\leq(\x\rp\con\z\bp\y)\rp1&&\text{mon}
  \end{align*}
  The opposite direction follows from this by \eqref{6}. The second
  part has a similar proof.
\end{proof}
\begin{prop}
  \label{i1}
  $\x\rp1\bp\y\rp\z=(\x\rp1\bp\y)\rp\z$ and
  $\y\rp\z\bp1\rp\x=\y\rp(\z\bp1\rp\x)$.
\end{prop}
\begin{proof}
  We prove only the first equation. The second can be proved
  similarly.
  \begin{align*}
    \x\rp1\bp\y\rp\z&\leq(\x\rp1\rp\con\z\bp\y)\rp\z&&\text{rot}
    \\ &\leq(\x\rp1\rp1\bp\y)\rp\z&&\text{mon}
    \\ &\leq(\x\rp(1\rp1)\bp\y)\rp\z&&\text{assoc}
    \\ &\leq(\x\rp1\bp\y)\rp\z&&\text{mon}
    \\ &\leq(\x\rp1)\rp\z\bp\y\rp\z&&\text{mon}
    \\ &\leq\x\rp(1\rp\z)\bp\y\rp\z&&\text{assoc}
    \\ &\leq\x\rp1\bp\y\rp\z&&\text{mon}
  \end{align*}
\end{proof}
\begin{prop}
  \label{icyc}
  $\id\bp\u\rp\v\bp\x\rp\y\leq
  \id\bp(\u\bp\con\v)\rp(\con\u\rp\x\bp\v\rp\con\y)\rp(\y\bp\con\x)$
\end{prop}
\begin{proof}
  \begin{align*}
    \id\bp\u\rp\v\bp\x\rp\y&\leq
    \id\bp(\u\bp\id\rp\con\v)\rp\v\bp\x\rp(\y\bp\con\x\rp\id)&&\text{rot}
    \\ &=\id\bp(\u\bp\con\v)\rp\v\bp\x\rp(\y\bp\con\x) &&\text{id}
    \\ &\leq\id\bp(\u\bp\con\v)\rp(\v\bp\con{\u\bp\con\v}
    \rp(\x\rp(\y\bp\con\x))) &&\text{rot}
    \\ &\leq\id\bp(\u\bp\con\v)\rp(\v\bp\con\u\rp(\x\rp(\y\bp\con\x)))
    &&\text{mon}
    \\ &=\id\bp(\u\bp\con\v)\rp(\v\bp\con\u\rp\x\rp(\y\bp\con\x))
    &&\text{assoc}
    \\ &\leq\id\bp(\u\bp\con\v)\rp((\con\u\rp\x\bp\v\rp\con{\y\bp\con\x})
    \rp(\y\bp\con\x)) &&\text{rot}
    \\ &\leq\id\bp(\u\bp\con\v)\rp(\con\u\rp\x\bp\v\rp\con\y)\rp(\y\bp\con\x)
    &&\text{mon, assoc}
  \end{align*}
\end{proof}
\begin{prop}
  \label{exch}
  $\id\bp(\u\bp\x)\rp(\v\bp\y)=\id\bp(\u\bp\con\v)\rp(\con\x\bp\y)$
  and $\id\bp\x\rp\y=\id\bp(\x\bp\con\y)\rp(\con\x\bp\y))$.
\end{prop}
\begin{proof}
  The inclusion from left to right in the first equation holds because
  \begin{align*}
    \id\bp(\u\bp\x)\rp(\v\bp\y)
    &=\id\bp(\id\rp\con{\v\bp\y}\bp\u\bp\x)\rp(\v\bp\y\bp\con{\u\bp\x}\rp\id)
    &&\text{rot}
    \\&=\id\bp(\con{\v\bp\y}\bp\u\bp\x)\rp(\v\bp\y\bp\con{\u\bp\x})
    &&\text{id}
    \\&=\id\bp(\con\v\bp\con\y\bp\u\bp\x)\rp(\v\bp\y\bp\con\u\bp\con\x)
    &&\text{con}
    \\&\leq\id\bp(\con\v\bp\x)\rp(\con\u\bp\y)&&\text{mon}
  \end{align*}
  The opposite inclusion follows from this by \eqref{6}.  The second
  equation follows from the first when $\u=\x$ and $\v=\y$.
\end{proof}

\section{Functional and permutational elements}
\label{s17}
\begin{defn}
  \label{FnPm}
  An element $\x$ is {\bf functional} if $\con\x\rp\x\leq\id$ and {\bf
    permutational} if $\x\rp\con\x=\id=\con\x\rp\x$.  $\Fn\AA$ is the
  set of {\bf functional elements} of $\AA$ and $\Pm\AA$ is the set of
  {\bf permutational elements} of $\AA$.
\end{defn}
\begin{prop}\;
  \label{func}
  \begin{enumerate}
  \item
    $0,\id\in\Fn\AA$ and $\id\in\Pm\AA\subseteq\Fn\AA$,
  \item
    if $\x\leq\y\in\Fn\AA$ then $\x\in\Fn\AA$,
  \item
    if $\x,\y\in\Fn\AA$ then $\x\rp\y\in\Fn\AA$,
  \item
    if $\x,\y\in\Pm\AA$ then $\x\rp\y\in\Pm\AA$ and $\con\x\in\Pm\AA$,
  \item
    $\<\Fn\AA,\rp,\id\>$ is a monoid and $\<\Pm\AA,\rp,\con\;,\id\>$
    is a group.
  \end{enumerate}
\end{prop}
Recall from \SS\ref{s6} that on the basis of Prop.\ \ref{func}(v) we
define $\gc{Fn}(\AA)=\<\Fn\AA,\rp,\id\>$ and
$\gc{Pm}(\AA)=\<\Pm\AA,\rp,\con\;,\id\>$.
\begin{proof}
  The inclusion $\Pm\AA\subseteq\Fn\AA$ is an immediate consequence of
  Def.\ \ref{FnPm}. For the rest of part (i), it suffices to note that
  $\con0\rp0=0\leq\id$ by \eqref{norm} and \eqref{zero} and
  $\con\id\rp\id=\id$ by \eqref{id}.  For part (ii), if
  $\x\leq\y\in\Fn\AA$ then $\con\x\rp\x\leq\con\y\rp\y$ by
  Prop.\ \ref{p6}, but $\con\y\rp\y\leq\id$ by hypothesis, so
  $\con\x\rp\x\leq\id$ by Prop.\ \ref{p1}.  For part (iii), if
  $\x,\y\in\Fn\AA$ then $\x\rp\y\in\Fn\AA$ because
  \begin{align*}
    \con{\x\rp\y}\rp(\x\rp\y)
    &=(\con\y\rp\con\x)\rp(\x\rp\y)&&\text{con}
    \\ &=\con\y\rp(\con\x\rp\x)\rp\y&&\text{assoc}
    \\ &\leq\con\y\rp\id\rp\y&&\text{mon, $\x\in\Fn\AA$}
    \\ &=\con\y\rp\y&&\text{id} \\ &\leq\id&&\text{$\y\in\Fn\AA$}
  \end{align*}
  For part (iv), assume $\x,\y\in\Pm\AA$. Then
  $\id=\con\x\rp\x=\x\rp\con\x$, hence
  $\id=\con\x\rp\con{\con\x}=\con{\con\x}\rp\con\x$ by \eqref{6}, so
  $\con\x\in\Pm\AA$. We also have two calculations similar to the one
  for part (iii), namely,
  \begin{align*}
    \con{\x\rp\y}\rp(\x\rp\y)&=(\con\y\rp\con\x)\rp(\x\rp\y)
    =\con\y\rp(\con\x\rp\x)\rp\y=\con\y\rp\id\rp\y=\con\y\rp\y=\id,
    \\ (\x\rp\y)\rp\con{\x\rp\y}&=(\x\rp\y)\rp(\con\y\rp\con\x)
    =\x\rp(\y\rp\con\y)\rp\con\x=\x\rp\id\rp\con\x=\con\x\rp\x=\id.
  \end{align*}
  The closure properties needed for part (v) come from parts (i),
  (iii), and (iv), while the requisite properties for a monoid and
  group are insured by \eqref{2}, \eqref{id}, \eqref{6}, \eqref{5},
  and Prop.\ \ref{p5}.
\end{proof}
By Prop.\ \ref{func}(v), $\id$ is the identity element of the group
$\gc{Pm}(\AA)$. More generally, every element $\e$ contained in $\id$
is the identity element of a group of elements that are
``permutational with respect to $\e$.''
\begin{prop}
  \label{grp}
  Assume $\e\leq\id$ and let $\X=\{\x:\x\rp\con\x=\e=\con\x\rp\x\}$.
  Then $\<\X,\rp,\con\;,\e\>$ is a group.
\end{prop}
\begin{proof}
  To show $\e\in\X$, we note that $\e=\e\rp\con\e$ because
  \begin{align*}
    \e&\leq\e\rp\con\e\rp\e&&\text{Prop.\ \ref{p10}} \\
    &\leq\e\rp\con\e\rp\id&&\text{mon, $\e\leq\id$} \\
    &=\e\rp\con\e&&\text{id} \\
    &\leq\e\rp\con\id&&\text{mon, $\e\leq\id$} \\
    &=\e\rp\id&&\text{con} \\
    &=\e&&\text{id}
  \end{align*}
  and, similarly, $\con\e\rp\e=\e$. Therefore, $\e\in\X$.  Also,
  $\e\leq\con\e$ because
  \begin{align*}
    \e&\leq\e\rp\con\e\rp\e&&\text{Prop.\ \ref{p10}} \\
    &\leq\id\rp\con\e\rp\id&&\text{mon, $\e\leq\id$}\\
    &=\con\e&&\text{id}
  \end{align*}
  hence $\con\e\leq\con{\con\e}=\e$ by Prop.\ \ref{p6} and \eqref{6}, so
  $\e=\con\e$.  If $\x\in\X$ then $\con\x\rp\x=\e=\con\x\rp\x$, so
  $\con{\con\x\rp\x}=\con\e=\con{\con\x\rp\x}$, hence
  $\x\rp\con\x=\e=\con\x\rp\x$ by \eqref{6}, \eqref{5}, and $\e=\con\e$,
  \ie, $\con\x\in\X$.  Thus, $\X$ is closed under converse. Next we see
  that $\e$ is an identity element for $\<\X,\rp,\con\;,\e\>$, that is,
  $\x=\x\rp\e=\e\rp\x$ for every $\x\in\X$. First, $\x=\e\rp\x$ because
  \begin{align*}
    \x&\leq\x\rp\con\x\rp\x&&\text{Prop.\ \ref{p10}}
    \\ &=\e\rp\x&&\x\in\X
    \\ &\leq\id\rp\x&&\text{mon, $\e\leq\id$}
    \\ &=\x&&\text{id}
  \end{align*}
  A similar proof shows $\x\rp\e=\x$. If $\x,\y\in\X$ then
  \begin{align*}
    \con{\x\rp\y}\rp(\x\rp\y)&=(\con\y\rp\con\x)\rp(\x\rp\y)&&\text{con}
    \\ &=\con\y\rp(\con\x\rp\x\rp\y)&&\text{assoc}
    \\ &=\con\y\rp(\e\rp\y)&&\x\in\X
    \\ &=\con\y\rp\y&&\text{proved above}
    \\ &=\e&&\y\in\X,
  \end{align*}
  and a similar proof shows $\x\rp\y\rp\con{\x\rp\y}=\e$. Therefore,
  $\x\rp\y\in\X$. We have shown that $\X$ is closed under converse and
  composition, and contains an identity element $\e$. Associativity
  holds by \eqref{2} and $\con\x$ is the inverse of $\x$, simply by
  the definition of $\X$. All the requirements for being a group are
  satisfied.
\end{proof}
\begin{prop}
  \label{f-dist}
  If $\cona\rp\a\leq\id$ then $\a\rp(\x\bp\y)=\a\rp\x\bp\a\rp\y$ and
  $(\x\bp\y)\rp\cona=\x\rp\cona\bp\y\rp\cona$.
\end{prop}
\begin{proof}
  Assume $\cona\rp\a\leq\id$. Then
  $\a\rp(\x\bp\y)=a\rp\x\bp\a\rp\y$ because
  \begin{align*}
    \a\rp(\x\bp\y) &\leq\a\rp\x\bp\a\rp\y&&\text{mon}
    \\ &\leq\a\rp(\x\bp\cona\rp(\a\rp\y))&&\text{rot}
    \\ &=\a\rp(\x\bp(\cona\rp\a)\rp\y)&&\text{assoc}
    \\ &\leq\a\rp(\x\bp\id\rp\y)&&\text{mon, $\cona\rp\a\leq\id$}
    \\ &\leq\a\rp(\x\bp\y)&&\text{id}
  \end{align*}
  and $(\x\bp\y)\rp\cona=\x\rp\cona\bp\y\rp\cona$ because
  \begin{align*}
    (\x\bp\y)\rp\cona &\leq\x\rp\cona\bp\y\rp\cona&&\text{mon}
    \\ &\leq(\y\bp\x\rp\cona\rp\con{\cona})\rp\cona&&\text{rot}
    \\ &=(\y\bp\x\rp(\cona\rp\a))\rp\cona&&\text{con, assoc}
    \\ &\leq(\y\bp\x\rp\id)\rp\cona&&\text{mon, $\cona\rp\a\leq\id$}
    \\ &\leq(\x\bp\y)\rp\cona&&\text{id}
  \end{align*}
\end{proof}
\begin{prop}
  \label{prop1a}
  Suppose $1=\cona\rp\b$ and $\cona\rp\a\leq\id$. Then
  $\cona\rp\a=\id$.
\end{prop}
\begin{proof}
  \begin{align*}
    \id&=\id\bp1 &&\eqref{one}
    \\ &=\id\bp\cona\rp\b &&\text{hypothesis}
    \\ &\leq\cona\rp(\b\bp\con{\cona}\rp\id)&&\text{rot}
    \\ &=\cona\rp(\b\bp\a)&&\text{id, con}
    \\ &\leq\cona\rp\a &&\text{mon}
    \\ &\leq\id &&\text{hypothesis}
  \end{align*}
\end{proof}
The following proposition was used in \SS\ref{s9} to note that if
\eqref{D} and \eqref{U} hold then $\a\rp\con\a\bp\b\rp\con\b=\id$.
\begin{prop}
  \label{prop2a}
  Assume $1=\x\rp1=\y\rp1$ and
  $\x\rp\con\x\bp\y\rp\con\y\leq\id$. Then
  $\x\rp\con\x\bp\y\rp\con\y=\id$.
\end{prop}
\begin{proof}
  First we obtain $\id\leq\x\rp\con\x$ as follows.
  \begin{align*}
    \id&=\id\bp1 &&\eqref{one}
    \\ &=\id\bp\x\rp1 &&\text{hypothesis}
    \\ &\leq\id\bp\x\rp(1\bp\con\x\rp\id) &&\text{rot}
    \\ &=\id\bp\x\rp\con\x&&\text{mon, id}
    \\ &\leq\x\rp\con\x&&\text{mon}
  \end{align*}
  We get $\id\leq\y\rp\con\y$ similarly, hence
  $\id\leq\x\rp\con\x\bp\y\rp\con\y$ by Prop.\ \ref{p3}.  Combining
  this with the other hypothesis $\x\rp\con\x\bp\y\rp\con\y\leq\id$
  yields the desired equality by Prop.\ \ref{p1}.
\end{proof}

\begin{prop}
  \label{f1}
  If $\cona\rp\a\leq\id$ then $\a\bp(\a\bp\x)\rp1\leq\x$.
\end{prop}
\begin{proof}
  Assume $\cona\rp\a\leq\id$.
  \begin{align*}
    \a\bp(\a\bp\x)\rp1&\leq(\a\bp\x)\rp(1\bp\con{\a\bp\x}\rp\a)&&\text{rot}
    \\ &\leq\x\rp(\cona\rp\a)&&\text{mon}
    \\ &\leq\x\rp\id&&\text{mon, $\cona\rp\a\leq\id$}
    \\ &=\x&&\text{id}
  \end{align*}
\end{proof}

\begin{prop}[{\cite[4.1(ix)($\beta$)]{MR920815}}]
  \label{q1}
  Assume $\x\leq\cona\rp\b$ and $\a,\b\in\Fn\AA$. Then
  $\x=\cona\rp(\id\bp\a\rp\x\rp\conb)\rp\b$.
\end{prop}
\begin{proof}
  \begin{align*}
    \x &=\x\bp\cona\rp\b &&\text{$\x\leq\cona\rp\b$}
    \\ &\leq(\x\rp\conb\bp\cona)\rp\b &&\text{rot}
    \\ &=(\x\rp\conb\bp\cona\rp\id)\rp\b &&\text{id}
    \\ &\leq\cona\rp(\id\bp\con{\cona}\rp\x\rp\conb)\rp\b&&\text{rot}
    \\ &=\cona\rp(\id\bp\a\rp\x\rp\conb)\rp\b&&\text{con }
    \\ &\leq\cona\rp(\a\rp\x\rp\conb)\rp\b&&\text{mon}
    \\ &=(\cona\rp\a)\rp\u\rp(\conb\rp\b)&&\text{assoc}
    \\ &\leq\id\rp\u\rp\id&&\text{$\a,\b\in\Fn\AA$, mon}
    \\ &=\x&&\text{id}
  \end{align*}
\end{proof}
The proof of \cite[4.1(viii)]{MR920815} is an equational derivation of
\eqref{Pr} from \eqref{Q}. It was adapted from an earlier proof dated
May 23, 1975, and given to Givant; see \SS\ref{s10a}.  Proposition
\ref{pair} below shows the two applications of the equation
$\cona\rp\b=1$ that occur in the proof of \cite[4.1(viii)]{MR920815}
are independent of each other. The functionality of $\a,\b$ is used
only for one direction of \eqref{Pr} in part (ii) and is not needed
for the other direction proved in part(i), which requires only that
one element in the left side of \eqref{Pr} is covered by a product of
two (possibly different) functional elements.  When \eqref{Q} holds,
Prop.\ \ref{pair}(ii) provides a proof of \eqref{Pr} because \eqref{Q}
implies that the hypotheses of part (iii) hold when $\c=\a$ and
$\d=\b$; see Prop.\ \ref{QPr}. First we note one direction of the
pairing identity for $\a,\b$ follows from just the functionality of
$\a,\b$.
\begin{prop}
  \label{1/2pr}
  If $\a,\b\in\Fn\AA$ then
  $\u\rp\v\bp\x\rp\y\geq(\u\rp\cona\bp\x\rp\conb)\rp(\a\rp\v\bp\b\rp\y)$.
\end{prop}
\begin{proof}
  \begin{align*}
    (\u\rp\cona\bp\x\rp\conb)\rp(\a\rp\v\bp\b\rp\y)
    &\leq\u\rp\cona\rp(\a\rp\v)\bp\x\rp\conb\rp(\b\rp\y)&&\text{mon}
    \\ &=\u\rp(\cona\rp\a)\rp\v\bp\x\rp(\conb\rp\b)\rp\y&&\text{assoc}
    \\ &\leq\u\rp\id\rp\v\bp\x\rp\id\rp\y&&\text{$\a,\b\in\Fn\AA$,
      mon} \\ &=\u\rp\v\bp\x\rp\y &&\text{id}
  \end{align*}
\end{proof}
\begin{prop}
  \label{pair}
  Assume $\u\leq\con\c\rp\d$ and $\v\rp\con\y\leq\cona\rp\b$.
  \begin{enumerate}
  \item If $\c,\d\in\Fn\AA$ then
    $\u\rp\v\bp\x\rp\y\leq(\u\rp\cona\bp\x\rp\conb)\rp(\a\rp\v\bp\b\rp\y)$.
  \item If $\a,\b,\c,\d\in\Fn\AA$, then
    \begin{align*}
      \u\rp\v\bp\x\rp\y&=(\u\rp\cona\bp\x\rp\conb)\rp(\a\rp\v\bp\b\rp\y).
    \end{align*}
  \end{enumerate}
\end{prop}
\begin{proof}
  The following calculation proves part (i) and one direction of the
  pairing identity \eqref{Pr} in part (iii).
  \begin{align*}
    &\u\rp\v\bp\x\rp\y \\ &=(\u\bp\con\c\rp\d)\rp\v\bp\x\rp\y
    &&\u\leq\con\c\rp\d
    \\ &=((\u\rp\con\d\bp\con\c)\rp\d)\rp\v\bp\x\rp\y &&\text{rot}
    \\ &=(\u\rp\con\d\bp\con\c)\rp(\d\rp\v)\bp\x\rp\y &&\text{assoc}
    \\ &\leq(\u\rp\con\d\bp\con\c)\rp\big(\d\rp\v\bp
    \con{\u\rp\con\d\bp\con\c}\rp(\x\rp\y)\big) &&\text{rot}
    \\ &\leq(\u\rp\con\d\bp\con\c)\rp\big(\d\rp\v\bp
    (\d\rp\con\u\bp\c)\rp(\x\rp\y)\big) &&\text{con}
    \\ &\leq(\u\rp\con\d\bp\con\c)\rp\big(\d\rp\v\bp
    \c\rp(\x\rp\y)\big) &&\text{mon}
    \\ &=(\u\rp\con\d\bp\con\c)\rp\big(\d\rp\v\bp\c\rp\x\rp\y\big)
    &&\text{assoc}
    \\ &=(\u\rp\con\d\bp\con\c)\rp\big(\d\rp\v\bp(\c\rp\x\bp\d\rp\v\rp\con\y)
    \rp\y\big) &&\text{rot}
    \\ &=(\u\rp\con\d\bp\con\c)\rp\big(\d\rp\v\bp(\c\rp\x\bp\d\rp(\v\rp\con\y))
    \rp\y\big) &&\text{assoc}
    \\ &\leq(\u\rp\con\d\bp\con\c)\rp\big(\d\rp\v\bp(\c\rp\x\bp\d\rp(\cona\rp\b))
    \rp\y\big) &&\text{$\v\rp\con\y\leq\cona\rp\b$, mon}
    \\ &=(\u\rp\con\d\bp\con\c)\rp\big(\d\rp\v\bp(\c\rp\x\bp\d\rp\cona\rp\b)\rp\y\big)
    &&\text{assoc}
    \\ &\leq(\u\rp\con\d\bp\con\c)\rp\big(\d\rp\v\bp((\c\rp\x\rp\conb\bp\d\rp\cona)
    \rp\b)\rp\y\big) &&\text{rot}
    \\ &=(\u\rp\con\d\bp\con\c)\rp\big(\d\rp\v\bp(\c\rp\x\rp\conb\bp\d\rp\cona)
    \rp(\b\rp\y)\big) &&\text{assoc}
    \\ &\leq(\u\rp\con\d\bp\con\c)\rp\big((\c\rp\x\rp\conb\bp\d\rp\cona)\rp
    (\con{\c\rp\x\rp\conb\bp\d\rp\cona}\rp(\d\rp\v)\bp\b\rp\y)\big)&&\text{rot}
    \\ &\leq(\u\rp\con\d\bp\con\c)\rp\big((\c\rp\x\rp\conb\bp\d\rp\cona)\rp
    (\con{\d\rp\cona}\rp(\d\rp\v)\bp\b\rp\y)\big) &&\text{mon}
    \\ &=(\u\rp\con\d\bp\con\c)\rp\big((\c\rp\x\rp\conb\bp\d\rp\cona)\rp
    (\a\rp\con\d\rp(\d\rp\v)\bp\b\rp\y)\big) &&\text{con}
    \\ &=(\u\rp\con\d\bp\con\c)\rp(\c\rp\x\rp\conb\bp\d\rp\cona)\rp
    (\a\rp(\con\d\rp\d)\rp\v\bp\b\rp\y) &&\text{assoc}
    \\ &\leq(\u\rp\con\d\bp\con\c)\rp(\c\rp\x\rp\conb\bp\d\rp\cona)\rp
    (\a\rp\id\rp\v\bp\b\rp\y) &&\text{$\d\in\Fn\AA$, mon}
    \\ &=(\u\rp\con\d\bp\con\c)\rp(\c\rp\x\rp\conb\bp\d\rp\cona)
    \rp(\a\rp\v\bp\b\rp\y) &&\text{id}
    \\ &\leq\big(\u\rp\con\d\rp(\d\rp\cona)\bp\con\c\rp(\c\rp\x\rp\conb)\big)
    \rp(\a\rp\v\bp\b\rp\y) &&\text{mon}
    \\ &=\big(\u\rp(\con\d\rp\d)\rp\cona\bp\con\c\rp\c\rp\x\rp\conb\big)
    \rp(\a\rp\v\bp\b\rp\y) &&\text{assoc}
    \\ &\leq\big(\u\rp\id\rp\cona\bp\id\rp\x\rp\conb\big)\rp(\a\rp\v\bp\b\rp\y)
    &&\text{$\c,\d\in\Fn\AA$, mon}
    \\ &=(\u\rp\cona\bp\x\rp\conb)\rp(\a\rp\v\bp\b\rp\y)&&\text{id}
  \end{align*}
  The other direction of \eqref{Pr} in part (ii) follows from the
  assumption that $\a,\b\in\Fn\AA$ by Prop.\ \ref{1/2pr}.
\end{proof}
The next proposition achieves conclusions similar to those of
Prop.\ \ref{pair} but under rather different hypotheses. Both
propositions yield \eqref{Pr} when \eqref{Q} holds in a J-algebra; see
Prop.\ \ref{QPr}. The conclusion in part (i) is the same formula
\eqref{J} introduced in \SS\ref{s18}; see Def.\ \ref{ij1}.  Thus,
Prop.\ \ref{pair2}(i) provides a proof for one case of Th.\ \ref{ij6}.
\begin{prop}
  \label{pair2}
  Assume $\c,\d\in\Fn\AA$ and $\u\rp\v\bp\x\rp\y\leq\con\c\rp\d$.
  \begin{enumerate}
  \item For all $\a,\b$,
    \begin{align*}
      \con\u\rp\x\bp\v\rp\con\y\leq\cona\rp\b\implies\u\rp\v\bp\x\rp\y
      \leq(\u\rp\cona\bp\x\rp\conb)\rp(\a\rp\v\bp\b\rp\y).
    \end{align*}
  \item
    If $\a,\b\in\Fn\AA$ then
    \begin{align*}
      \con\u\rp\x\bp\v\rp\con\y\leq\cona\rp\b\implies\u\rp\v\bp\x\rp\y
      =(\u\rp\cona\bp\x\rp\conb)\rp(\a\rp\v\bp\b\rp\y).
    \end{align*}
  \item
    If $\a,\b\in\Fn\AA$ and $1=\cona\rp\b$ then
    \begin{align*}
      \u\rp\v\bp\x\rp\y
      =(\u\rp\cona\bp\x\rp\conb)\rp(\a\rp\v\bp\b\rp\y).
    \end{align*}
  \end{enumerate}
\end{prop}
\begin{proof}
  For part (i), assume $\con\u\rp\x\bp\v\rp\con\y\leq\cona\rp\b$.  Let
  $\w=\u\rp\v$ and $\z=\x\rp\y$, so one of the hypotheses now says
  $\w\bp\z\leq\con\c\rp\d$.  First we show
  \begin{align}
    \label{what}
    \w\bp\z&=\con\c\rp(\id\bp\c\rp\w\rp\con\d\bp\c\rp\z\rp\con\d)\rp\d
  \end{align}
  as follows.
  \begin{align*}
    \w\bp\z&=\con\c\rp(\id\bp\c\rp(\w\bp\z)\rp\con\d)\rp\d
    &&\text{Prop.\ \ref{q1}, $\c,\d\in\Fn\AA$,
      $\w\bp\z\leq\con\c\rp\d$}
    \\ &=\con\c\rp(\id\bp\c\rp\w\rp\con\d\bp\c\rp\z\rp\con\d)\rp\d
    &&\text{func dist, $\c,\d\in\Fn\AA$}
    \\ &\leq\con\c\rp(\c\rp\w\rp\con\d)\rp\d
    \bp\con\c\rp(\c\rp\z\rp\con\d)\rp\d &&\text{mon}
    \\ &=(\con\c\rp\c)\rp\w\rp(\con\d\rp\d)
    \bp(\con\c\rp\c)\rp\z\rp(\con\d\rp\d) &&\text{assoc}
    \\ &\leq\id\rp\w\rp\id\bp\id\rp\z\rp\id&&\text{mon,
      $\c,\d\in\Fn\AA$} \\ &=\w\bp\z &&\text{id}
  \end{align*}
  Let $\p=\c\rp\u$, $\q=\v\rp\con\d$, $\r=\c\rp\x$, and
  $\s=\y\rp\con\d$. Then, by \eqref{what}, \eqref{2}, and the
  definitions of $\w,\z,\p,\q,\r,\s$,
  \begin{align}
    \label{1}
    \u\rp\v\bp\x\rp\y=\con\c\rp(\id\bp\p\rp\q\bp\r\rp\s)\rp\d
  \end{align}
  Consider the subterm $\id\bp\p\rp\q\bp\r\rp\s$ appearing on the
  right side of \eqref{1}. By Prop.\ \ref{icyc} we have
  \begin{align}
    \label{2a}
    \id\bp\p\rp\q\bp\r\rp\s&\leq\id\bp(\p\bp\con\q)
    \rp(\con\p\rp\r\bp\q\rp\con\s)\rp(\s\bp\con\r)
  \end{align}
  The subterm $\con\p\rp\r\bp\q\rp\con\s$ on the right side of
  \eqref{2a} has the property that
  \begin{align}
    \label{3a}
    \con\p\rp\r\bp\q\rp\con\s \leq\cona\rp\b
  \end{align}
  because
  \begin{align*}
    \con\p\rp\r\bp\q\rp\con\s
    &=\con{\c\rp\u}\rp(\c\rp\x)\bp\v\rp\con\d\rp\con{\y\rp\con\d}
    &&\text{defs of $\p,\q,\r,\s$}
    \\ &=\con\u\rp\con\c\rp(\c\rp\x)\bp\v\rp\con\d\rp(\d\rp\con\y)
    &&\text{con}
    \\ &=\con\u\rp(\con\c\rp\c)\rp\x\bp\v\rp(\con\d\rp\d)\rp\con\y
    &&\text{assoc} \\ &\leq\con\u\rp\id\rp\x\bp\v\rp\id\rp\con\y
    &&\text{$\c,\d\in\Fn\AA$, mon} \\ &\leq\con\u\rp\x\bp\v\rp\con\y
    &&\text{id} \\ &\leq\cona\rp\b&&\text{assumption}
  \end{align*}
  From \eqref{2a} and \eqref{3a} we conclude that
  \begin{align}
    \label{4a}
    \id\bp\p\rp\q\bp\r\rp\s
    \leq\id\bp(\c\rp\u\rp\cona\bp\d\rp\con\v\rp\cona)
    \rp(\b\rp\con\x\rp\con\c\bp\b\rp\y\rp\con\d)
  \end{align}
  because
  \begin{align*}
    &\id\bp\p\rp\q\bp\r\rp\s
    \\ &\leq\id\bp(\p\bp\con\q)\rp(\cona\rp\b)\rp(\s\bp\con\r)
    &&\text{\eqref{2a}, \eqref{3a}, mon}
    \\ &=\id\bp(\c\rp\u\bp\con{\v\rp\con\d})\rp(\cona\rp\b)
    \rp(\y\rp\con\d\bp\con{\c\rp\x}) &&\text{defs of $\p,\q,\r,\s$}
    \\ &=\id\bp(\c\rp\u\bp\d\rp\con\v)\rp(\cona\rp\b)
    \rp(\y\rp\con\d\bp\con\x\rp\con\c)
    &&\text{con}
    \\ &=\id\bp(\c\rp\u\bp\d\rp\con\v)\rp\cona
    \rp(\b\rp(\y\rp\con\d\bp\con\x\rp\con\c))
    &&\text{assoc}
    \\ &\leq\id\bp(\c\rp\u\rp\cona\bp\d\rp\con\v\rp\cona)
    \rp(\b\rp(\con\x\rp\con\c)\bp
    \b\rp(\y\rp\con\d)) &&\text{mon}
    \\ &=\id\bp(\c\rp\u\rp\cona\bp\d\rp\con\v\rp\cona)
    \rp(\b\rp\con\x\rp\con\c\bp\b\rp\y\rp\con\d) &&\text{assoc}
  \end{align*}
  We now finish the proof of part (i) by showing
  $\u\rp\v\bp\x\rp\y\leq(\u\rp\cona\bp\x\rp\conb)\rp(\a\rp\v\bp\b\rp\y)$,
  the conclusion of \eqref{J}.
  \begin{align*}
    &\u\rp\v\bp\x\rp\y
    \\ &\leq\con\c\rp(\id\bp(\c\rp\u\rp\cona\bp\d\rp
    \con\v\rp\cona)\rp(\b\rp\con\x\rp\con\c\bp\b\rp\y\rp\con\d))\rp\d
    &&\text{\eqref{1}, \eqref{4a}, mon}
    \\ &=\con\c\rp(\id\bp(\c\rp\u\rp\cona\bp\con{\b\rp\con\x\rp\con\c})
    \rp(\con{\d\rp\con\v\rp\cona}\bp\b\rp\y\rp\con\d))\rp\d
    &&\text{Prop.\ \ref{exch}}
    \\ &=\con\c\rp(\id\bp(\c\rp\u\rp\cona\bp\c\rp\x\rp\conb)
    \rp(\a\rp\v\rp\con\d\bp\b\rp\y\rp\con\d))\rp\d &&\text{con, assoc}
    \\ &=\con\c\rp(\id\bp\c\rp(\u\rp\cona\bp\x\rp\conb)
    \rp((\a\rp\v\bp\b\rp\y)\rp\con\d))\rp\d
    &&\text{func dist, $\c,\d\in\Fn\AA$}
    \\ &\leq(\con\c\rp\c)\rp(\u\rp\cona\bp\x\rp\conb)
    \rp(\a\rp\v\bp\b\rp\y)\rp(\con\d\rp\d) &&\text{mon, assoc}
    \\ &\leq\id\rp(\u\rp\cona\bp\x\rp\conb)\rp(\a\rp\v\bp\b\rp\y)\rp\id
    &&\text{$\c,\d\in\Fn\AA$, mon}
    \\ &=(\u\rp\cona\bp\x\rp\conb)\rp(\a\rp\v\bp\b\rp\y) &&\text{id}
  \end{align*}

  Part (ii) holds by part (i) and Prop.\ \ref{1/2pr}, which provides
  the opposite direction of the pairing identity for $\a,\b$ with no
  need for the hypotheses of part (i), namely, $\c,\d\in\Fn\AA$,
  $\u\rp\v\bp\x\rp\y\leq\con\c\rp\d$, and
  $\con\u\rp\x\bp\v\rp\con\y\leq\cona\rp\b$.

  For part (iii) it suffices to note that the hypothesis of the
  implication proved in part (ii) holds by \eqref{one} under the
  assumption that $\cona\rp\b=1$.
\end{proof}

\section{Quasiprojections}
\label{s19}
Consider two fixed elements $\a$ and $\b$.  Borrowing notation and
concepts from fork algebras (see \SS\ref{s10} and
\cite[p.\,23]{MR1922164}), we use $\a$ and $\b$ as parameters to
define binary operations $\nabla$, $\otimes$, and $\fkc{}{}$ by
setting, for all $\x,\y\in\A$,
\begin{align}
  \fk\x\y&=\x\rp\cona\bp\y\rp\conb, \\
  \label{ot}
  \ot\x\y&=\a\rp\x\rp\cona\bp\b\rp\y\rp\conb,
  \\ \fkc\x\y&=\a\rp\x\bp\b\rp\y.
\end{align}
These operations acquire interesting and useful properties when the
parameters $\a$ and $\b$ satisfy various combinations of the equations
\eqref{Q}, \eqref{D}, and \eqref{U}.  By Prop.\ \ref{prop2a},
\eqref{D} and \eqref{U} imply
\begin{align}
  \label{T}
  \id&=\a\rp\cona\bp\b\rp\conb
\end{align}
By Prop.\ \ref{QPr}, which relies on either Prop.\ \ref{pair} or
Prop.\ \ref{pair2}, the pairing identity \eqref{Pr} holds for the
elements $\a$ and $\b$ whenever \eqref{Q} holds. By Prop.\ \ref{fork},
if \eqref{Q} and \eqref{U} then $\nabla$ is the fork operation
satisfying axioms \eqref{F1}--\eqref{F3}. The operation $\otimes$
could be called ``parallel computation'' to reflect its use in
applications in computer science, or ``parallel composition'' because
of Prop.\ \ref{gg-rule} below.  We show next the operations $\nabla$
and $\otimes$ produce functional outputs from functional inputs
whenever $\a$ and $\b$ are functional and the Unicity Condition
\eqref{U} holds.
\begin{prop}
  \label{f-closed}
  Assume $\a,\b\in\Fn\AA$ and \eqref{U}. If $\x,\y\in\Fn\AA$ then
  $\fk\x\y,\ot\x\y\in\Fn\AA$.
\end{prop}
\begin{proof}
  We have $\fk\x\y\in\Fn\AA$ because
  \begin{align*}
    \con{\fk\x\y}\rp\fk\x\y
    &=\con{\x\rp\cona\bp\y\rp\conb}\rp(\x\rp\cona\bp\y\rp\conb)
    &&\text{def $\fk{}{}$}
    \\ &=(\a\rp\con\x\bp\b\rp\con\y)\rp(\x\rp\cona\bp\y\rp\conb)
    &&\text{con}
    \\ &\leq(\a\rp\con\x)\rp(\x\rp\cona)\bp(\b\rp\con\y)\rp(\y\rp\conb)
    &&\text{mon}
    \\ &=\a\rp(\con\x\rp\x)\rp\cona\bp\b\rp(\con\y\rp\y)\rp\conb&&\text{assoc}
    \\ &\leq\a\rp\id\rp\cona\bp\b\rp\id\rp\conb
    &&\text{$\x,\y\in\Fn\AA$, mon}
    \\ &=\a\rp\cona\bp\b\rp\conb&&\text{id}
    \\ &\leq\id&&\eqref{U}
  \end{align*}
  From $\a,\b,\x,\y\in\Fn\AA$ we get $\a\rp\x,\b\rp\y\in\Fn\AA$ by
  Prop.\ \ref{func}(iii), hence $\ot\x\y=\fk{(\a\rp\x)}{(\b\rp\y)}
  \in\Fn\AA$ by what was just proved.
\end{proof}
\begin{prop}
  \label{fgh-closed}
  Assume \eqref{Q} and \eqref{U}. Then $0,\id,\a,\b\in\Fn\AA$ and
  $\Fn\AA$ is closed under inclusion $\leq$, Boolean product $\bp$,
  relative product $\rp$, $\fk{}{}$, and $\ot{}{}$.
\end{prop}
\begin{proof}
  First observe that $0,\id\in\Fn\AA$ by Prop.\ \ref{func}(i), that
  $\a,\b\in\Fn\AA$ by \eqref{Q}, and that $\Fn\AA$ is closed under
  inclusion and relative product by Prop.\ \ref{func}(ii)(iii). It
  follows from closure under inclusion that $\Fn\AA$ is also closed
  under Boolean product. $\Fn\AA$ is closed under $\fk{}{}$ and
  $\ot{}{}$ by Prop.\ \ref{f-closed}.
\end{proof}
\begin{prop}
  \label{g-id}
  If $\a\rp\cona\bp\b\rp\conb=\id$ then $\ot\id\id=\id$.
\end{prop}
\begin{proof}
  By the definition of $\otimes$ and \eqref{id},
  \begin{align*}
    \ot\id\id&=\a\rp\id\rp\cona\bp\b\rp\id\rp\conb
    =\a\rp\cona\bp\b\rp\conb=\id.
  \end{align*}
\end{proof}
\begin{prop}
  \label{gg-rule}
  Assume \eqref{Pr}.  Then
  \begin{gather*}
    (\ot\u\v)\rp(\ot\x\y)=\ot{(\u\rp\x)}{(\v\rp\y)},
    \\ (\ot\id\x)\rp(\ot\y\id)=\ot\x\y=(\ot\x\id)\rp(\ot\id\y).
  \end{gather*}
\end{prop}
\begin{proof}
  The first equation holds because
  \begin{align*}
    (\ot\u\v)\rp(\ot\x\y)&=(\a\rp\u\rp\cona\bp\b\rp\v\rp\conb)
    \rp(\a\rp\x\rp\cona\bp\b\rp\y\rp\conb)&&\text{def $\otimes$}
    \\ &=\a\rp\u\rp(\x\rp\cona)\bp\b\rp\v\rp(\y\rp\conb)
    &&\text{\eqref{Pr}}
    \\ &=\a\rp(\u\rp\x)\rp\cona\bp\b\rp(\v\rp\y)\rp\conb&&\text{assoc}
    \\ &=\ot{(\u\rp\x)}{(\v\rp\y)}&&\text{def $\otimes$}
  \end{align*}
  The other equations follow from the first by \eqref{id} and
  Prop.\ \ref{p5}.
\end{proof}
All parts of \eqref{Qu} are needed to get permutational elements as
outputs from $\otimes$, as shown by the next proposition.
\begin{prop}
  \label{g-closed}
  Assume \eqref{Q}, \eqref{D}, and \eqref{U}. Then $\Pm\AA$ is closed
  under $\ot{}{}$.
\end{prop}
\begin{proof}
  Let $\x,\y\in\Pm\AA$.  From \eqref{D} and \eqref{U} we get
  $\a\rp\con\a\bp\b\rp\con\b=\id$ by Prop.\ \ref{prop2a}, hence
  $\ot\id\id=\id$ by Prop.\ \ref{g-id}.  From \eqref{Q} we get the
  pairing identity \eqref{Pr} by Prop.\ \ref{QPr}, so
  Prop.\ \ref{gg-rule} can be applied.
  \begin{align*}
    \con{\ot\x\y}\rp(\ot\x\y) &=(\ot{\con\x}{\con\y})\rp(\ot\x\y)
    &&\text{def $\otimes$, con}
    \\ &=\ot{(\con\x\rp\x)}{(\con\y\rp\y)}&&\text{Prop.\ \ref{gg-rule}}
    \\ &=\ot\id\id&&\text{$\x,\y\in\Pm\AA$}
    \\ &=\id&&\text{Prop.\ \ref{g-id}}
  \end{align*}
  and a similar proof shows $(\ot\x\y)\rp\con{\ot\x\y}=\id$.
  Therefore, $\ot\x\y\in\Pm\AA$.
\end{proof}
\begin{prop}
  \label{fg-rule}
  If \eqref{Pr} then
  \begin{align*}
    (\fk\x\y)\rp(\ot\u\v)&=\fk{(\x\rp\u)}{(\y\rp\v)},
    \\ (\fk\x\y)\rp(\ot\u\id)&=\fk{(\x\rp\u)}\y,
    \\ (\fk\x\y)\rp(\ot\id\v)&=\fk\x{(\y\rp\v)}.
  \end{align*}
\end{prop}
\begin{proof}
  The first equation holds because
  \begin{align*}
    (\fk\x\y)\rp(\ot\u\v)
    &=(\x\rp\cona\bp\y\rp\conb)\rp(\a\rp\u\rp\cona\bp\b\rp\v\rp\conb)
    &&\text{defs $\nabla,\otimes$}
    \\&=\x\rp(\u\rp\cona)\bp\y\rp(\v\rp\conb)&&\text{\eqref{Pr}}
    \\&=\x\rp\u\rp\cona\bp\y\rp\v\rp\conb&&\text{assoc}
    \\&=\fk{(x\rp\u)}{(\y\rp\v)}&&\text{def $\fk{}{}$}
  \end{align*}
  The second and third equations come from taking $\v=\id$ or $\u=\id$
  in the first and invoking axiom \eqref{id}.
\end{proof}

\begin{prop}
  \label{fh-rule}
  If \eqref{Pr} then $(\fk\u\v)\rp(\fkc\x\y)=\u\rp\x\bp\v\rp\y$
\end{prop}
\begin{proof}
  \begin{align*}
    (\fk\u\v)\rp(\fkc\x\y)&=
    (\u\rp\cona\bp\v\rp\conb)\rp(\a\rp\x\bp\b\rp\y) &&\text{defs
      $\nabla,\fkc{}{}$} \\ &=\u\rp\x\bp\v\rp\y&&\eqref{Pr}
  \end{align*}
\end{proof}
The following proposition is needed for the proof of Prop.\ \ref{same}.
\begin{prop}
  \label{Ux0K}
  Assume \eqref{Q}. Then $(\cona\bp\conb)\rp(\ot\x\id)\rp\a=\x$ for
  all $\x$.
\end{prop}
\begin{proof}
  To prove $(\cona\bp\conb)\rp(\ot\x\id)\rp\a\leq\x$, note first that
  \eqref{Pr} follows from \eqref{Q} by Prop.\ \ref{QPr} so
  \begin{align*}
    (\cona\bp\conb)\rp(\ot\x\id)
    &=(\id\rp\cona\bp\id\rp\conb)\rp\Big(\a\rp\x\rp\cona\bp\b\rp\conb\Big)
    &&\text{id, def $\otimes$}
    \\ &=\id\rp(\x\rp\cona)\bp\id\rp\conb&&\eqref{Pr}
    \\ &=\x\rp\cona\bp\conb&&\text{id} \intertext{Multiply by $\a$ on
      the right and get}
    (\cona\bp\conb)\rp(\ot\x\id)\rp\a&=(\x\rp\cona\bp\conb)\rp\a
    \\ &\leq\x\rp(\cona\rp\a)&&\text{mon, assoc}
    \\ &\leq\x&&\text{\eqref{Q}, mon, id}
  \end{align*}
  For the opposite inclusion, first note that $\a\in\Fn\AA$ by
  \eqref{Q}, so $\a\bp\b\in\Fn\AA$ by Prop.\ \ref{func}(ii), and
  $1=\cona\rp\b$ by \eqref{Q}, so $1 = \con1 = \con{\cona\rp\b} =
  \conb\rp\a$ by Prop.\ \ref{p5}, \eqref{6}, and \eqref{5}.  Then
  $\x\leq(\cona\bp\conb)\rp(\ot\x\id)\rp\a$ because
  \begin{align*}
    \x&=\id\rp\x&&\text{id}
    \\ &=(\id\bp\cona\rp\b)\rp(\x\bp\conb\rp\a)&&1=\cona\rp\b=\conb\rp\a
    \\&=(\id\bp(\cona\bp\conb)\rp(\a\bp\b))\rp(\x\bp\conb\rp\a)
    &&\text{Prop.\ \ref{exch}}
    \\&\leq(\id\bp(\cona\bp\conb)\rp(\a\bp\b))\rp((\conb\bp\x\rp\cona)\rp\a)
    &&\text{rot}
    \\&\leq(\cona\bp\conb)\rp(\a\bp\b)\rp(\conb\bp\x\rp\cona)\rp\a
    &&\text{mon, assoc}
    \\&=(\cona\bp\conb)\rp\Big((\a\bp\b)\rp\conb\bp(\a\bp\b)
    \rp(\x\rp\cona)\Big)\rp\a &&\text{func dist, $\a\bp\b\in\Fn\AA$}
    \\&\leq(\cona\bp\conb)\rp(\b\rp\conb\bp\a\rp\x\rp\cona)\rp\a
    &&\text{mon, assoc}
    \\&=(\cona\bp\conb)\rp(\b\rp\id\rp\conb\bp\a\rp\x\rp\cona)\rp\a&&\text{id}
    \\&=(\cona\bp\conb)\rp(\ot\x\id)\rp\a&&\text{def $\otimes$}
  \end{align*}
\end{proof}
The next three propositions are used in the proof of
Prop.\ \ref{presM}.
\begin{prop}
  \label{sub0}
  Assume \eqref{Q}. Then
  $(\cona\bp\conb)\rp\a=\id=(\cona\bp\conb)\rp\b$.
\end{prop}
\begin{proof}
  By \eqref{Q}, $\cona\rp\a\leq\id$, $\conb\rp\b\leq\id$, and
  $1=\cona\rp\b$.  The first equation holds because
  \begin{align*}
    \id&=\id\bp\cona\rp\b&&1=\cona\rp\b
    \\ &=\id\bp(\cona\bp\conb)\rp(\a\bp\b)&&\text{Prop.\ \ref{exch}}
    \\ &\leq\cona\rp\a&&\text{mon}
    \\ &\leq\id
  \end{align*}
  and the second may be derived similarly, using $\conb\rp\b\leq\id$
  for the last step.
\end{proof}
\begin{prop}
  \label{pok}
  Assume \eqref{Q} and \eqref{D}.
  \begin{enumerate}
  \item
    If $1=\y\rp1$ then $(\ot\x\y)\rp\a=\a\rp\x$.
  \item
    If $1=\x\rp1$ then $(\ot\x\y)\rp\b=\b\rp\y$.
  \item
    If $1=\y\rp1$ then $(\ot\id\y)\rp\a=\a$.
  \item
    If $1=\x\rp1$ then $(\ot\x\id)\rp\b=\b$.
  \item
    $(\ot\x\id)\rp\a=\a\rp\x$ and $(\ot\id\y)\rp\b=\b\rp\y$.
  \end{enumerate}
\end{prop}
\begin{proof}
  From \eqref{Q} we get \eqref{Pr} by Prop.\ \ref{QPr} and
  $\id=\cona\rp\a=\conb\rp\b$ by Prop.\ \ref{prop1a}.  We get
  $1=\b\rp1=\a\rp1$ from \eqref{D}.  For part (i), assume
  $1=\y\rp1$. Then
  \begin{align*}
    (\ot\x\y)\rp\a&=(\a\rp\x\rp\cona\bp\b\rp\y\rp\conb)\rp(\a\rp\id\bp\b\rp1)
    &&\text{def $\otimes$, id, $1=\b\rp1$}
    \\ &=\a\rp\x\rp\id\bp\b\rp\y\rp1 &&\text{\eqref{Pr}}
    \\ &=\a\rp\x\bp\b\rp(\y\rp1) &&\text{id, assoc}
    \\ &=\a\rp\x &&1=\y\rp1=\b\rp1
  \end{align*}
  Part (ii) has a similar proof that uses $1=\x\rp1=\a\rp1$ instead of
  $1=\y\rp1=\b\rp1$.  Parts (iii) and (iv) follow by \eqref{id} from
  parts (i) and (ii), respectively. Part (v) follows from parts (i)
  and (ii) because $1=\id\rp1$.
\end{proof}
\begin{prop}
  \label{swap}
  Assume \eqref{Q} and \eqref{D} and let
  $\P=\a\rp\conb\bp\b\rp\cona$. Then $\P\rp\a=\b$ and $\P\rp\b=\a$.
\end{prop}
\begin{proof}
  By the symmetry of the situation, it suffices to prove just the
  first equation. We get \eqref{Pr} from \eqref{Q} and
  $1=\a\rp1=\b\rp1$ from \eqref{D}. Then
  \begin{align*}
    \P\rp\a&=(\a\rp\conb\bp\b\rp\cona)\rp\a &&\text{def $\P$}
    \\ &=(\a\rp\conb\bp\b\rp\cona)\rp(\a\rp\id\bp\b\rp1) &&\text{id,
      $\b\rp1=1$} \\ &=\b\rp\id\bp\a\rp1 &&\text{\eqref{Pr}} \\ &=\b
    &&\text{id, $\a\rp1=1$}
  \end{align*}
\end{proof}
\section{Properties of generators of $\F,\T,\VV,\M$}
\label{s20}
\begin{prop}
  \label{perms}
  Assume $\a,\b$ satisfy \eqref{Q}, \eqref{D}, and \eqref{U}. Let
  \begin{align*}
    \K&=\a, \qquad \L=\b, & \U&=\cona\bp\conb,
    \\ \P&=\a\rp\conb\bp\b\rp\cona,
    & \P_0&=\ot\P\id=\a\rp\P\rp\cona\bp\b\rp\conb,
    \\ \A=\R&=\a\rp\cona^2\bp\b\rp\a\rp\conb\rp\cona\bp\b^2\rp\conb,
    & \R_0&=\ot\R\id=\a\rp\R\rp\cona\bp\b\rp\conb,
    \\ \C&=\a\rp\conb^2\bp\b\rp\a\rp\cona\bp\b^2\rp\cona\rp\conb,
    & \B&=\ot\id\A=\a\rp\cona\bp\b\rp\A\rp\conb,
    \\ \pi_0&=\a\rp\cona\rp\conb\bp\b\rp\a\rp\cona\bp\b^2\rp\conb^2.
  \end{align*}
  Then $\K$, $\L$, $\U$, and $\con\U$ are functional and
  $\P$, $\P_0$, $\R$, $\R_0$, $\A$, $\B$, $\C$, and $\pi_0$ are
  permutational.
\end{prop}
\begin{proof}
  \eqref{Q} says that $\a$ and $\b,$ \ie, $\K$ and $\L$, are
  functional.  By \eqref{Q} and \eqref{U}, $\U$ and $\con\U$ are
  functional because
  \begin{align*}
    \con\U\rp\U&=\con{\cona\bp\conb}\rp(\cona\bp\conb) &&\text{def
      $\U$} \\ &=(\a\bp\b)\rp(\cona\bp\conb)&&\text{con}
    \\ &\leq\a\rp\cona\bp\b\rp\conb&&\text{mon} \\ &\leq\id&&\eqref{U}
    \\ \U\rp\con\U&=(\cona\bp\conb)\rp\con{\cona\bp\conb} &&\text{def
      $\U$} \\ &=(\cona\bp\conb)\rp(\a\bp\b)&&\text{con}
    \\ &\leq\cona\rp\a\bp\conb\rp\b&&\text{mon}
    \\ &\leq\id&&\a,\b\in\Fn\AA
  \end{align*}
  The two equations \eqref{Pr} and \eqref{T} are applied many times
  below. Recall that \eqref{T} follows from \eqref{D} and \eqref{U}
  while the pairing identity \eqref{Pr} comes from \eqref{Q}.  Since
  $\a,\b\in\Fn\AA$, Prop.\ \ref{f-dist} is available and will be
  applied many times with the customary notation ``func dist''.  By
  \eqref{2}, \eqref{6}, \eqref{5}, and \eqref{4},
  \begin{align*}
    \con\P&=\P,
    \\ \con\R=\con\A&=
    \a^2\rp\cona \bp
    \a\rp\b\rp\cona\rp\conb \bp
    \b\rp\conb^2,
    \\ \con\C&=
    \b^2\rp\cona \bp
    \a\rp\cona\rp\conb \bp
    \b\rp\a\rp\conb^2,
    \\ \con{\pi_0}&=
    \b\rp\a\rp\cona \bp
    \a\rp\cona\rp\conb \bp
    \b^2\rp\conb^2.
  \end{align*}
  Since $\con\P=\P$, we need only observe that
  \begin{align*}
    \P\rp\P &=(\a\rp\conb\bp\b\rp\cona)\rp(\a\rp\conb\bp\b\rp\cona)
    \\ &=\b\rp\con\b\bp\a\rp\con\a&&\eqref{Pr} \\ &=\id&&\eqref{T}
  \end{align*}
  to see that $\P$ is permutational.  It follows by
  Prop.\ \ref{g-closed} that $\P_0=\ot\P\id$ is also permutational.
  The permutationality of $\A$ and $\R$ is established directly by
  showing that $\A\rp\con\A=\con\A\rp\A=\id,$ which implies
  $\B=\ot\A\id$ is also permutational by Prop.\ \ref{g-closed}.
  \begin{align*}
    \A\rp\con\A
    &=(\a\rp\cona^2\bp\b\rp\a\rp\conb\rp\cona\bp\b^2\rp\conb)
    \rp(\a^2\rp\cona \bp \a\rp\b\rp\cona\rp\conb \bp
    \b\rp\conb^2) &&\text{defs $\A$, $\con\A$}
    \\ &=((\a\rp\cona\bp\b\rp\a\rp\conb)\rp\cona\bp\b^2\rp\conb)\rp
    (\a\rp(\a\rp\cona\bp\b\rp\cona\rp\conb)\bp\b\rp\conb^2)
    &&\text{func dist}
    \\ &=(\a\rp\cona\bp\b\rp\a\rp\conb)\rp(\a\rp\cona\bp\b\rp\cona\rp\conb)
    \bp\b^2\rp\conb^2&&\text{\eqref{Pr}}
    \\ &=\a\rp\cona\bp\b\rp\a\rp\cona\rp\conb\bp
    \b^2\rp\conb^2&&\text{\eqref{Pr}}
    \\ &=\a\rp\cona\bp\b\rp(\a\rp\cona\bp\b\rp\conb)\rp\conb
    &&\text{func dist}
    \\ &=\a\rp\cona\bp\b\rp\id\rp\conb&&\text{\eqref{T}}
    \\ &=\a\rp\cona\bp\b\rp\conb&&\text{id}
    \\ &=\id&&\text{\eqref{T}}
  \end{align*}
  \begin{align*}
    \con\A\rp\A
    &=(\a^2\rp\cona\bp\a\rp\b\rp\cona\rp\conb\bp\b\rp\conb^2)
    \rp(\a\rp\cona^2\bp\b\rp\a\rp\conb\rp\cona\bp\b^2\rp\conb)
    &&\text{defs $\A$, $\con\A$}
    \\ &=(\a^2\rp\cona\bp(\a\rp\b\rp\cona\bp\b\rp\conb)\rp\conb)
    \rp(\a\rp\cona^2\bp\b\rp(\a\rp\conb\rp\cona\bp\b\rp\conb))
    &&\text{func dist} \\ &=\a^2\rp\cona^2 \bp
    (\a\rp\b\rp\cona\bp\b\rp\conb)\rp(\a\rp\conb\rp\cona\bp\b\rp\conb)
    &&\text{\eqref{Pr}} \\ &=\a^2\rp\cona^2 \bp
    \a\rp\b\rp\conb\rp\cona \bp \b\rp\conb &&\text{\eqref{Pr}}
    \\ &=\a\rp(\a\rp\cona \bp \b\rp\conb)\rp\cona \bp \b\rp\conb
    &&\text{func dist} \\ &=\a\rp\cona \bp \b\rp\conb
    &&\text{\eqref{T}, id} \\ &=\id &&\text{\eqref{T}}
  \end{align*}
  Now we show $\C$ and $\pi_0$ are permutational.
  \begin{align*}
    \con\C\rp\C
    &=(\b^2\rp\cona \bp
    \a\rp\cona\rp\conb \bp
    \b\rp\a\rp\conb^2)
    \rp(\a\rp\conb^2 \bp
    \b\rp\a\rp\cona \bp
    \b^2\rp\cona\rp\conb)
    &&\text{defs $\C$, $\con\C$}
    \\ &=(\b^2\rp\cona \bp
    (\a\rp\cona \bp
    \b\rp\a\rp\conb)\rp\conb)
    \rp(\a\rp\conb^2 \bp
    \b\rp(\a\rp\cona \bp
    \b\rp\cona\rp\conb))
    &&\text{func dist}
    \\ &=\b^2\rp\conb^2 \bp
    (\a\rp\cona \bp
    \b\rp\a\rp\conb)\rp
    (\a\rp\cona \bp
    \b\rp\cona\rp\conb)
    &&\text{\eqref{Pr}}
    \\ &=\b^2\rp\conb^2 \bp
    \a\rp\cona \bp
    \b\rp\a\rp\cona\rp\conb
    &&\text{\eqref{Pr}}
    \\ &=\a\rp\cona \bp
    \b\rp(\b\rp\conb \bp
    \a\rp\cona)\rp\conb
    &&\text{func dist}
    \\ &=\b\rp\conb \bp \a\rp\cona
    &&\text{\eqref{T}, id}
    \\ &=\id
    &&\text{\eqref{T}}
  \end{align*}
  \begin{align*}
    \C\rp\con\C
    &=(\a\rp\conb^2 \bp
    \b\rp\a\rp\cona \bp
    \b^2\rp\cona\rp\conb)
    \rp(\b^2\rp\cona \bp
    \a\rp\cona\rp\conb \bp
    \b\rp\a\rp\conb^2)
    &&\text{defs $\C$, $\con\C$}
    \\ &=(\b\rp\a\rp\cona \bp
    (\a\rp\conb \bp
    \b^2\rp\cona)\rp\conb)
    \rp(\a\rp\cona\rp\conb \bp
    \b\rp(\b\rp\cona \bp
    \a\rp\conb^2))
    &&\text{func dist}
    \\ &=\b\rp\a\rp\cona\rp\conb \bp
    (\a\rp\conb \bp \b^2\rp\cona)
    \rp(\b\rp\cona \bp \a\rp\conb^2)
    &&\text{\eqref{Pr}}
    \\ &=\b\rp\a\rp\cona\rp\conb \bp
    \b^2\rp\conb^2 \bp
    \a\rp\cona
    &&\text{\eqref{Pr}}
    \\ &=\b\rp(\a\rp\cona \bp
    \b\rp\conb)\rp\conb \bp
    \a\rp\cona
    &&\text{func dist}
    \\ &=\b\rp\conb \bp \a\rp\cona
    &&\text{\eqref{T}, id}
    \\ &=\id
    &&\text{\eqref{T}}
  \end{align*}
  \begin{align*}
    \con\pi_0\rp{\pi_0}
    &=(\b\rp\a\rp\cona \bp
    \a\rp\cona\rp\conb \bp
    \b^2\rp\conb^2)
    \rp(\a\rp\cona\rp\conb \bp
    \b\rp\a\rp\cona \bp
    \b^2\rp\conb^2)
    &&\text{defs ${\pi_0}$, $\con{\pi_0}$}
    \\ &=(\b\rp\a\rp\cona \bp
    (\a\rp\cona \bp
    \b^2\rp\conb)\rp\conb)
    \rp(\a\rp\cona\rp\conb \bp
    \b\rp(\a\rp\cona \bp
    \b\rp\conb^2))
    &&\text{func dist}
    \\ &=\b\rp\a\rp\cona\rp\conb \bp
    (\a\rp\cona \bp
    \b^2\rp\conb)\rp
    (\a\rp\cona \bp
    \b\rp\conb^2)
    &&\text{\eqref{Pr}}
    \\ &=\b\rp\a\rp\cona\rp\conb \bp
    \a\rp\cona \bp
    \b^2\rp\conb^2
    &&\text{\eqref{Pr}}
    \\ &=\a\rp\cona \bp
    \b\rp(\a\rp\cona \bp
    \b\rp\conb)\rp\conb
    &&\text{func dist}
    \\ &=\a\rp\cona \bp \b\rp\conb
    &&\text{\eqref{T}, id}
    \\ &=\id
    &&\text{\eqref{T}}
  \end{align*}
  \begin{align*}
    \pi_0\rp\con{\pi_0}
    &=(\a\rp\cona\rp\conb \bp
    \b\rp\a\rp\cona \bp
    \b^2\rp\conb^2)
    \rp(\b\rp\a\rp\cona \bp
    \a\rp\cona\rp\conb \bp
    \b^2\rp\conb^2)
    &&\text{defs $\pi_0$, $\con{\pi_0}$}
    \\ &=(\b\rp\a\rp\cona \bp
    (\a\rp\cona \bp
    \b^2\rp\conb)\rp\conb)
    \rp(\a\rp\cona\rp\conb \bp
    \b\rp(\a\rp\cona \bp
    \b\rp\conb^2))
    &&\text{func dist}
    \\ &=\b\rp\a\rp\cona\rp\conb \bp
    (\a\rp\cona \bp
    \b^2\rp\conb)\rp
    (\a\rp\cona \bp
    \b\rp\conb^2)
    &&\text{\eqref{Pr}}
    \\ &=\b\rp\a\rp\cona\rp\conb \bp
    \a\rp\cona \bp
    \b^2\rp\conb^2
    &&\text{\eqref{Pr}}
    \\ &=\a\rp\cona \bp
    \b\rp(\a\rp\cona \bp
    \b\rp\conb)\rp\conb
    &&\text{func dist}
    \\ &=\a\rp\cona \bp \b\rp\conb
    &&\text{\eqref{T}, id}
    \\ &=\id
    &&\text{\eqref{T}}
  \end{align*}
\end{proof}

\section{Thompson's group $\F$}
\label{s22}
\begin{prop}
  \label{Rg}
  Assume $\a$ and $\b$ satisfy \eqref{Q}, \eqref{D}, and
  \eqref{U}. Let
  \begin{align*}
    \A=\a\rp\cona^2\bp\b\rp\a\rp\conb\rp\cona\bp\b^2\rp\conb.
  \end{align*}
  Then
  \begin{align*}
    \A\rp(\ot\id\x)\rp\con\A&=\ot\id{(\ot\id\x)},
    \\ \con\A\rp(\ot\x\id)\rp\A&=\ot{(\ot\x\id)}\id.
  \end{align*}
\end{prop}
\begin{proof}
  \begin{align*}
    &\A\rp(\ot\id\x)\rp\con\A \\ &=
    (\a\rp\cona^2\bp\b\rp\a\rp\conb\rp\cona\bp\b^2\rp\conb) \rp
    (\a\rp\id\rp\cona\bp\b\rp\x\rp\conb)\rp\con\A &&\text{defs $\A$,
      $\otimes$} \\ &= ((\a\rp\cona \bp
    \b\rp\a\rp\conb)\rp\cona\bp\b^2\rp\conb) \rp
    (\a\rp\cona\bp\b\rp\x\rp\conb)\rp\con\A &&\text{func dist, id}
    \\ &=( (\a\rp\cona \bp \b\rp\a\rp\conb)\rp\cona\bp
    \b^2\rp\x\rp\conb) \rp(\a^2\rp\cona \bp \a\rp\b\rp\cona\rp\conb \bp
    \b\rp\conb^2) &&\text{\eqref{Pr}, def $\A$, con} \\ &=((\a\rp\cona
    \bp \b\rp\a\rp\conb)\rp\cona\bp \b^2\rp\x\rp\conb)
    \rp(\a\rp(\a\rp\cona \bp \b\rp\cona\rp\conb) \bp \b\rp\conb^2)
    &&\text{func dist} \\ &= (\a\rp\cona \bp
    \b\rp\a\rp\conb)\rp(\a\rp\cona \bp \b\rp\cona\rp\conb) \bp
    \b^2\rp\x\rp\conb^2 &&\eqref{Pr} \\ &=\a\rp\cona \bp
    \b\rp\a\rp\cona\rp\conb \bp \b^2\rp\x\rp\conb^2 &&\eqref{Pr}
    \\ &=\a\rp\cona\bp\b\rp(\a\rp\cona\bp\b\rp\x\rp\conb)\rp\conb
    &&\text{func dist}
    \\ &=\a\rp\id\cona\bp\b\rp(\a\rp\id\rp\cona\bp\b\rp\x\rp\conb)\rp\conb
    &&\text{id} \\ &=\a\rp\id\cona\bp\b\rp(\ot\id\x)\rp\conb &&\text{def
      $\otimes$} \\ &=\ot\id{(\ot\id\x)} &&\text{def $\otimes$}
  \end{align*}
  The second equation has a similar proof that can be obtained by
  simply interchanging $\a$ and $\b$. The conditions imposed by
  \eqref{Q} on $\a$ and $\b$ are symmetric in $\a$ and $\b$, and
  interchanging $\a$ and $\b$ in the definitions of $\A$ and $\ot\x\y$
  produces $\con\A$ and $\ot\y\x$.
\end{proof}
\begin{prop}
  \label{presF}
  Assume $\a,\b$ satisfy \eqref{Q}, \eqref{D}, and \eqref{U}. Let
  \begin{align*}
    \A&=
    \a\rp\cona^2 \bp
    \b\rp\a\rp\conb\rp\cona \bp
    \b^2\rp\conb,
    \\ \B&=
    \a\rp\cona\bp\b\rp\a\rp\cona^2\rp\conb \bp
    \b^2\rp\a\rp\conb\rp\cona\rp\conb \bp
    \b^3\rp\conb^2.
  \end{align*}
  The elements $\A$ and $\B$ satisfy the two relations \eqref{ta1} and
  \eqref{ta2} in the presentation of Thompson's group $\F$, \ie,
  \begin{align*}
    [\;\con\B\rp\A,\;\A\rp\B\rp\con\A\;]
    &=\id=[\;\con\B\rp\A,\;\A^2\rp\B\rp\con\A^2\;].
  \end{align*}
\end{prop}
\begin{proof}
  $\A$ and $\B$ are permutational by Prop.\ \ref{perms}, hence
  $\A\rp\con\A=\con\A\rp\A=\id=\B\rp\con\B=\con\B\rp\B$.  This implies
  that the two relations are equivalent to equations expressing
  commutativity,
  \begin{align*}
    (\A\rp\B\rp\con\A)\rp(\con\B\rp\A) &=
    (\con\B\rp\A)\rp(\A\rp\B\rp\con\A),
    \\ (\A^2\rp\B\rp\con\A^2)\rp(\con\B\rp\A) &=
    (\con\B\rp\A)\rp(\A^2\rp\B\rp\con\A^2).
  \end{align*}
  In both equations, move the final $\A$ on the left side to the right
  side and the initial $\con\B$ on the right side to the left side,
  obtaining
  \begin{align*}
    \B\rp\A\rp\B\rp\con\A\rp\con\B &= \A^2\rp\B\rp\con\A^2, &
    \B\rp\A^2\rp\B\rp\con\A^2\rp\con\B &=\A^3\rp\B\rp\con\A^3 .
  \end{align*}
  We prove these last two equations. First note that, since
  $\a,\b\in\Fn\AA$, we have
  \begin{align}
    \label{e1}
    \B&=\ot\id\A
  \end{align}
  because
  \begin{align*}
    \B&= \a\rp\cona \bp \b\rp\a\rp\cona^2\rp\conb \bp
    \b^2\rp\a\rp\conb\rp\cona\rp\conb \bp \b^3\rp\conb^2
    \\ &=\a\rp\cona \bp \b\rp(\a\rp\cona^2 \bp \b\rp\a\rp\conb\rp\cona
    \bp \b^2\rp\conb)\rp\conb&&\text{func dist}
    \\ &=\a\rp\id\rp\cona\bp\b\rp\A\rp\conb&&\text{id, def $\A$}
    \\ &=\ot\id\A&&\text{def $\otimes$}
  \end{align*}
  By \eqref{e1} and Prop.\ \ref{Rg} with $\x=\A$,
  \begin{align}
    \label{e2}
    \A\rp\B\rp\con\A&=\ot\id\B.
  \end{align}
  Applying Prop.\ \ref{Rg} again, this time with $\x=\B$, produces
  \begin{align}
    \label{e3}
    \A^2\rp\B\rp\con\A^2&=\ot\id{(\ot\id\B)}.
  \end{align}
  By a third application of Prop.\ \ref{Rg} with $\x=\ot\id\B$,
  \begin{align}
    \label{e4}
    \A^3\rp\B\rp\con\A^3 &=\ot\id{(\ot\id{(\ot\id\B)})}.
  \end{align}
  The converse of $\B$ is
  \begin{align}
    \label{e5}
    \con\B&=\ot\id{\con\A}
  \end{align}
  because
  \begin{align*}
    \con\B&=\con{\ot\id\A} &&\eqref{e1} \\\notag
    &=\con{\a\rp\id\rp\cona\bp\b\rp\A\rp\conb} &&\text{def $\otimes$}
    \\\notag
    &=\a\rp\id\rp\cona\bp\b\rp\con\A\rp\conb&&\text{con, assoc}
    \\\notag &=\ot\id{\con\A}&&\text{def $\otimes$}
  \end{align*}
  We can then obtain the desired equations.
  \begin{align*}
    \B\rp(\A\rp\B\rp\con\A)\rp\con\B
    &=(\ot\id\A)\rp(\ot\id\B)\rp(\ot\id{\con\A}) &&\text{\eqref{e1},
      \eqref{e2} \eqref{e5}} \\\notag&=\ot\id{(\A\rp\B\rp\con\A)}
    &&\text{Prop.\ \ref{gg-rule}}
    \\\notag&=\ot\id{(\ot\id\B)}&&\text{\eqref{e2}}
    \\\notag&=\A^2\rp\B\rp\con\A^2&&\text{\eqref{e3}}
    \\ \B\rp(\A^2\rp\B\rp\con\A^2)\rp\con\B
    &=(\ot\id\A)\rp(\ot\id{(\ot\id\B)})\rp(\ot\id{\con\A})
    &&\text{\eqref{e1}, \eqref{e3}, \eqref{e5}}
    \\\notag&=\ot\id{(\A\rp(\ot\id\B)\rp\con\A)}
    &&\text{Prop.\ \ref{gg-rule}}
    \\\notag&=\ot\id{(\ot\id{(\ot\id\B)})} &&\text{Prop.\ \ref{Rg}}
    \\\notag&=\A^3\rp\B\rp\con\A^3&&\text{\eqref{e4}}
  \end{align*}
\end{proof}

\section{Thompson's group $\T$}
\label{s23}
\begin{prop}
  \label{presT}
  Assume $\a,\b$ satisfy \eqref{Q}, \eqref{D}, and \eqref{U}. Let
  \begin{align*}
    \A&=
    \a\rp\cona^2 \bp
    \b\rp\a\rp\conb\rp\cona \bp
    \b^2\rp\conb,
    \\ \B&=
    \a\rp\cona\bp\b\rp\a\rp\cona^2\rp\conb \bp
    \b^2\rp\a\rp\conb\rp\cona\rp\conb \bp
    \b^3\rp\conb^2,
    \\ \C&=
    \a\rp\conb^2 \bp
    \b\rp\a\rp\cona \bp
    \b^2\rp\cona\rp\conb.
  \end{align*}
  The elements $\A,\B,\C$ satisfy the six relations
  \eqref{ta1}--\eqref{ta6} in the presentation of Thompson's group
  $\T$.
\end{prop}
\begin{proof}
  The relations \eqref{ta1} and \eqref{ta2} were proved in
  Prop.\ \ref{presF}. The remaining relations to be verified are
  \eqref{ta3}--\eqref{ta6}:
  \begin{align*}
    \C&=\C_2\rp\B \\ \X_2\rp\C_2&=\C_3\rp\B \\ \A\rp\C&=\C_2^2
    \\ \C^3&=\id
  \end{align*}
  where
  \begin{align*}
    \X_2&=\A\rp\B\rp\con\A&\X_3&=\A^2\rp\B\rp\con\A^2
    \\ \C_2&=\B\rp\C\rp\con\A&\C_3&=\B^2\rp\C\rp\con\A^2
  \end{align*}
  For \eqref{ta6}, we show the second power of $\C$ is its converse:
  \begin{align*}
    \C^2&=(\a\rp\conb^2\bp\b\rp\a\rp\cona\bp\b^2\rp\cona\rp\conb)
    \rp(\a\rp\conb^2\bp\b\rp\a\rp\cona\bp\b^2\rp\cona\rp\conb)
    &&\text{def $\C$}
    \\ &=(\b\rp\a\rp\cona\bp(\b^2\rp\cona\bp\a\rp\conb)\rp\conb)
    \rp(\a\rp\conb^2\bp\b\rp(\a\rp\cona\bp\b\rp\cona\rp\conb))
    &&\text{func dist}
    \\ &=\b\rp\a\rp\conb^2\bp(\b^2\rp\cona\bp\a\rp\conb)
    \rp(\a\rp\cona\bp\b\rp\cona\rp\conb)&&\text{\eqref{Pr}}
    \\ &=\b\rp\a\rp\conb^2\bp\b^2\rp\cona\bp\a\rp\cona\rp\conb
    &&\text{\eqref{Pr}}
    \\ &=\con{\b\rp\a\rp\cona\bp\b^2\rp\cona\rp\conb\bp\a\rp\conb^2}
    &&\text{con} \\ &=\con\C &&\text{def $\C$}
  \end{align*}
  Since $\C$ is permutational, $\C^3=\C^2\rp\C=\con\C\rp\C=\id$.  For
  \eqref{ta3} we calculate in turn $\B\rp\C$, $\C_2$, and
  $\C_2\rp\B$. In the last of these computations we use
  $\a\rp\cona\bp\b\rp\conb=\id$, which holds by \eqref{D} and
  \eqref{U}.  We have
  \begin{align}
    \label{BC}
    \B\rp\C &= (\b\rp\a\rp\cona \bp \b^2\rp\a\rp\conb)\rp\cona \bp
    (\a\rp\conb \bp \b^3\rp\cona)\rp\conb,
  \end{align}
  because
  \begin{align*}
    \B\rp\C &= (\a\rp\cona \bp \b\rp\a\rp\cona^2\rp\conb \bp
    \b^2\rp\a\rp\conb\rp\cona\rp\conb \bp \b^3\rp\conb^2)
    \\ &\qquad\rp(\a\rp\conb^2 \bp \b\rp\a\rp\cona \bp
    \b^2\rp\cona\rp\conb)&&\text{defs $\B,\C$}
    \\ &=(\a\rp\cona\bp(\b\rp\a\rp\cona^2 \bp
    \b^2\rp\a\rp\conb\rp\cona \bp \b^3\rp\conb)\rp\conb)
    \\ &\qquad\rp(\a\rp\conb^2
    \bp\b\rp(\a\rp\cona\bp\b\rp\cona\rp\conb)) &&\text{func dist}
    \\ &= \a\rp\conb^2 \bp (\b\rp\a\rp\cona^2 \bp
    \b^2\rp\a\rp\conb\rp\cona \bp
    \b^3\rp\conb)\rp(\a\rp\cona\bp\b\rp\cona\rp\conb) &&\eqref{Pr}
    \\ &= \a\rp\conb^2 \bp ((\b\rp\a\rp\cona \bp
    \b^2\rp\a\rp\conb)\rp\cona\bp\b^3\rp\conb)\rp
    (\a\rp\cona\bp\b\rp\cona\rp\conb) &&\text{func dist}
    \\ &=\a\rp\conb^2\bp(\b\rp\a\rp\cona\bp\b^2\rp\a\rp\conb)
    \rp\cona\bp\b^3\rp\cona\rp\conb&&\eqref{Pr}
    \\ &=(\b\rp\a\rp\cona\bp\b^2\rp\a\rp\conb)\rp\cona\bp
    (\a\rp\conb\bp\b^3\rp\cona)\rp\conb &&\text{func dist}
  \end{align*}
  and
  \begin{align}
    \label{BCA'}
    \C_2&=\b\rp\a\rp\cona \bp \b^2\rp\a\rp\cona\rp\conb \bp
    \a\rp\conb^3 \bp \b^3\rp\cona\rp\conb^2,
  \end{align}
  because
  \begin{align*}
    \C_2&=\B\rp\C\rp\con\A \\ &= ((\b\rp\a\rp\cona \bp
    \b^2\rp\a\rp\conb)\rp\cona \bp (\a\rp\conb \bp
    \b^3\rp\cona)\rp\conb) \\&\qquad \rp(\a^2\rp\cona \bp
    \a\rp\b\rp\cona\rp\conb \bp \b\rp\conb^2) &&\text{\eqref{BC}, def
      $\con\A$} \\ &= ((\b\rp\a\rp\cona \bp
    \b^2\rp\a\rp\conb)\rp\cona \bp (\a\rp\conb \bp
    \b^3\rp\cona)\rp\conb) \\&\qquad \rp(\a\rp(\a\rp\cona \bp
    \b\rp\cona\rp\conb) \bp \b\rp\conb^2) &&\text{func dist} \\ &=
    (\b\rp\a\rp\cona \bp \b^2\rp\a\rp\conb)\rp(\a\rp\cona \bp
    \b\rp\cona\rp\conb) \bp (\a\rp\conb \bp \b^3\rp\cona)\rp\conb^2
    &&\eqref{Pr} \\ &= \b\rp\a\rp\cona \bp
    \b^2\rp\a\rp\cona\rp\conb \bp (\a\rp\conb \bp
    \b^3\rp\cona)\rp\conb^2 &&\eqref{Pr} \\ &= \b\rp\a\rp\cona \bp
    \b^2\rp\a\rp\cona\rp\conb \bp \a\rp\conb^3 \bp
    \b^3\rp\cona\rp\conb^2 &&\text{func dist}
  \end{align*}
  and
  \begin{align}
    \label{BCA'B}
    \C_2\rp\B&=\C,
  \end{align}
  because
  \begin{align*}
    \C_2\rp\B &= \Big(\b\rp\a\rp\cona \bp \b^2\rp\a\rp\cona\rp\conb
    \bp \a\rp\conb^3 \bp \b^3\rp\cona\rp\conb^2 \Big) \\&\qquad
    \rp\Big( \a\rp\cona\bp\b\rp\a\rp\cona^2\rp\conb \bp
    \b^2\rp\a\rp\conb\rp\cona\rp\conb \bp \b^3\rp\conb^2 \Big)
    &&\text{\eqref{BCA'}, def $\B$} \\ &= \Big(\b\rp\a\rp\cona \bp
    (\b^2\rp\a\rp\cona \bp \a\rp\conb^2 \bp
    \b^3\rp\cona\rp\conb)\rp\conb\Big) \\&\qquad \rp\Big(\a\rp\cona
    \bp \b\rp(\a\rp\cona^2\rp\conb \bp \b\rp\a\rp\conb\rp\cona\rp\conb
    \bp \b^2\rp\conb^2)\Big) &&\text{func dist} \\ &= \b\rp\a\rp\cona
    \bp (\b^2\rp\a\rp\cona \bp \a\rp\conb^2 \bp \b^3\rp\cona\rp\conb)
    \\&\qquad \rp(\a\rp\cona^2\rp\conb \bp
    \b\rp\a\rp\conb\rp\cona\rp\conb \bp \b^2\rp\conb^2) &&\eqref{Pr}
    \\ &= \b\rp\a\rp\cona \bp (\b^2\rp\a\rp\cona \bp (\a\rp\conb \bp
    \b^3\rp\cona)\rp\conb) \\&\qquad \rp(\a\rp\cona^2\rp\conb \bp
    \b\rp(\a\rp\conb\rp\cona\rp\conb \bp \b\rp\conb^2)) &&\text{func
      dist} \\ &= \b\rp\a\rp\cona \bp \b^2\rp\a\rp\cona^2\rp\conb \bp
    (\a\rp\conb \bp \b^3\rp\cona) \\&\qquad
    \rp(\a\rp\conb\rp\cona\rp\conb \bp \b\rp\conb^2) &&\eqref{Pr} \\ &=
    \b\rp\a\rp\cona \bp \b^2\rp\a\rp\cona^2\rp\conb \bp
    \b^3\rp\conb\rp\cona\rp\conb \bp \a\rp\conb^2 &&\eqref{Pr} \\ &=
    \b\rp\a\rp\cona \bp \b^2\rp(\a\rp\cona \bp
    \b\rp\conb)\rp\cona\rp\conb \bp \a\rp\conb^2 &&\text{func dist}
    \\ &= \b\rp\a\rp\cona \bp \b^2\rp\cona\rp\conb \bp \a\rp\conb^2
    &&\a\rp\cona\bp\b\rp\conb=\id \\ &= \C &&\text{def $\C$}
  \end{align*}
  This last calculation confirms the relation \eqref{ta3}.  We can use
  \eqref{ta3} to simplify \eqref{ta5} from $\A\rp\C=\C_2^2$ to
  $\A\rp\C=\C^2\rp\con\A$ since $\C_2^2 = \C_2\rp\B\rp\C\rp\con\A =
  \C\rp\C\rp\con\A$.  Since $\C$ is permutational, \eqref{ta6} is
  equivalent to $\C^2=\con\C$, so we can simplify \eqref{ta5} further
  to $\A\rp\C=\con\C\rp\con\A=\con{\A\rp\C}$.  We confirm
  $\A\rp\C=\con{\A\rp\C}$ by calculating the form of $\A\rp\C$ and
  noticing that it is fixed by converse. Indeed, the last step holds
  because $\con{\a\rp\cona\rp\conb^2}=\b^2\rp\a\rp\cona$ and
  $\con{\b\rp\a\rp\conb^3}=\b^3\rp\cona\rp\conb$.
  \begin{align*}
    \A\rp\C &= (\a\rp\cona^2 \bp \b\rp\a\rp\conb\rp\cona \bp
    \b^2\rp\conb)\rp(\a\rp\conb^2 \bp \b\rp\a\rp\cona \bp
    \b^2\rp\cona\rp\conb) &&\text{defs $\A,\C$} \\ &= ((\a\rp\cona \bp
    \b\rp\a\rp\conb)\rp\cona \bp \b^2\rp\conb)\rp(\a\rp\conb^2 \bp
    \b\rp(\a\rp\cona \bp \b\rp\cona\rp\conb)) &&\text{func dist} \\ &=
    (\a\rp\cona \bp \b\rp\a\rp\conb)\rp\conb^2 \bp \b^2\rp(\a\rp\cona
    \bp \b\rp\cona\rp\conb) &&\eqref{Pr} \\ &= \a\rp\cona\rp\conb^2 \bp
    \b\rp\a\rp\conb^3 \bp \b^2\rp\a\rp\cona \bp \b^3\rp\cona\rp\conb
    &&\text{func dist} \\ &= \con{\A\rp\C}&&\text{con}
  \end{align*}
  We turn to \eqref{ta4}, which expands to
  $\A\rp\B\rp\con\A\rp\B\rp\C\rp\con\A=\B^2\rp\C\rp\con\A^2\rp\B$
  according to the definitions of $\X_2$, $\C_2$, and $\C_3$.
  Analyzed in parenthetical notation, the actions of the two sides of
  \eqref{ta4} are
  \begin{align*}
    0(1(2(34))) &\rrr\A (01)(2(34)) \rrr\B (01)((23)4) \rrr{\con\A}
    0(1((23)4)) \rrr\B \\ &\rrr\B 0((1(23))4) \rrr\C (1(23))(40)
    \rrr{\con\A} 1((23)(40)) \\ \\ 0(1(2(34))) &\rrr\B 0((12)(34))
    \rrr\B 0(((12)3)4) \rrr\C ((12)3)(40) \rrr{\con\A}
    \\ &\rrr{\con\A} (12)(3(40)) \rrr{\con\A} 1(2(3(40))) \rrr\B
    1((23)(40))
  \end{align*}
  Note that the actions are the same. For an algebraic proof, we first
  calculate $\X_2=\A\rp\B\rp\con\A$ in a more expanded form than
  \eqref{e2} and combine it with \eqref{BCA'} to get $\X_2\rp\C_2$,
  the left side of \eqref{ta4}. We have
  \begin{align}
    \label{ABA'}
    \X_2&=\a\rp\cona \bp (\b\rp\a\rp\cona \bp
    \b^2\rp\a\rp\cona^2\rp\conb \bp \b^3\rp\a\rp\conb\rp\cona\rp\conb
    \bp \b^4\rp\conb^2)\rp\conb
  \end{align}
  because
  \begin{align*}
    &\X_2=\A\rp\B\rp\con\A=\ot\id\B&&\eqref{e2}
    \\ &=\a\rp\id\rp\cona\bp\b\rp\B\rp\conb&&\text{def $\otimes$}
    \\ &=\a\rp\cona\bp\b\rp(\a\rp\cona \bp \b\rp\a\rp\cona^2\rp\conb
    \bp \b^2\rp\a\rp\conb\rp\cona\rp\conb \bp
    \b^3\rp\conb^2)\rp\conb&&\text{id, def $\B$} \\ &=\a\rp\cona \bp
    (\b\rp\a\rp\cona \bp \b^2\rp\a\rp\cona^2\rp\conb \bp
    \b^3\rp\a\rp\conb\rp\cona\rp\conb \bp
    \b^4\rp\conb^2)\rp\conb&&\text{func dist}
  \end{align*}
  and
  \begin{align}
    \label{ABA'BCA'}
    \X_2\rp\C_2&=\a\rp\conb^3 \bp \b\rp\a\rp\cona \bp
    \b^2\rp\a\rp\cona^2\rp\conb \bp \b^3\rp\a\rp\conb\rp\cona\rp\conb
    \bp \b^4\rp\cona\rp\conb^2
  \end{align}
  because
  \begin{align*}
    \X_2\rp\C_2 &=(\a\rp\cona \bp (\b\rp\a\rp\cona \bp
    \b^2\rp\a\rp\cona^2\rp\conb \bp \b^3\rp\a\rp\conb\rp\cona\rp\conb
    \bp \b^4\rp\conb^2)\rp\conb) \\ &\qquad \rp (\a\rp\conb^3 \bp
    \b\rp(\a\rp\cona \bp \b\rp\a\rp\cona\rp\conb \bp
    \b^2\rp\cona\rp\conb^2)) &&\text{\eqref{ABA'}, \eqref{BCA'}}
    \\ &=\a\rp\conb^3 \bp (\b\rp\a\rp\cona \bp (\b^2\rp\a\rp\cona^2
    \bp \b^3\rp\a\rp\conb\rp\cona \bp \b^4\rp\conb)\rp\conb)
    \\ &\qquad \rp(\a\rp\cona \bp \b\rp(\a\rp\cona\rp\conb \bp
    \b\rp\cona\rp\conb^2)) &&\eqref{Pr} \\ &=\a\rp\conb^3 \bp
    \b\rp\a\rp\cona \bp (\b^2\rp\a\rp\cona^2 \bp
    \b^3\rp\a\rp\conb\rp\cona \bp \b^4\rp\conb) \\ &\qquad
    \rp(\a\rp\cona\rp\conb \bp \b\rp\cona\rp\conb^2) &&\eqref{Pr}
    \\ &=\a\rp\conb^3 \bp \b\rp\a\rp\cona \bp ((\b^2\rp\a\rp\cona \bp
    \b^3\rp\a\rp\conb)\rp\cona \bp \b^4\rp\conb) \\ &\qquad
    \rp(\a\rp\cona\rp\conb \bp \b\rp\cona\rp\conb^2) &&\text{func
      dist} \\ &=\a\rp\conb^3 \bp \b\rp\a\rp\cona \bp
    (\b^2\rp\a\rp\cona \bp \b^3\rp\a\rp\conb)\rp\cona\rp\conb \bp
    \b^4\rp\cona\rp\conb^2 &&\eqref{Pr} \\ &=\a\rp\conb^3 \bp
    \b\rp\a\rp\cona \bp \b^2\rp\a\rp\cona^2\rp\conb \bp
    \b^3\rp\a\rp\conb\rp\cona\rp\conb \bp \b^4\rp\cona\rp\conb^2
    &&\text{func dist}
  \end{align*}
  For the other half of \eqref{ta4}, namely
  $\C_3\rp\B=\B^2\rp\C\rp\con\A^2\rp\B$, we calculate $\con\A^2$,
  $\B^2$, and $\B^2\rp\C$. Combining the results gives
  $\C_3=\B^2\rp\C\rp\con\A^2$, and finally we calculate $\C_3\rp\B$,
  getting the same result as \eqref{ABA'BCA'}. We prove, successively,
  the equations
  \begin{align}
    \label{AA}
    \A^2 &=(\a\rp\cona^3 \bp \b\rp\a\rp\conb\rp\cona^2 \bp
    \b^2\rp\a\rp\conb\rp\cona \bp \b^3\rp\conb) \\
    \label{AAc}
    \con\A^2 &= \a^3\rp\cona \bp \a^2\rp\b\rp\cona\rp\conb \bp
    \a\rp\b\rp\cona\rp\conb^2 \bp \b\rp\conb^3 \\
    \label{BB}
    \B^2&=\a\rp\cona \bp \b\rp\a\rp\cona^3\rp\conb \bp
    \b^2\rp\a\rp\conb\rp\cona^2\rp\conb \bp
    \b^3\rp\a\rp\conb\rp\cona\rp\conb \bp \b^4\rp\conb^2  \\
    \label{BBC} \B^2\rp\C &=\a\rp\conb^2 \bp \b\rp\a\rp\cona^3 \bp
    \b^2\rp\a\rp\conb\rp\cona^2 \bp \b^3\rp\a\rp\conb\rp\cona \bp
    \b^4\rp\cona\rp\conb \\
    \label{BBCA'A'}
    \C_3&=\b\rp\a\rp\cona \bp \b^2\rp\a\rp\cona\rp\conb \bp
    \b^3\rp\a\rp\cona\rp\conb^2 \bp \a\rp\conb^4 \bp
    \b^4\rp\cona\rp\conb^3 \\
    \label{BBCA'A'B}
    \C_3\rp\B &=\b\rp\a\rp\cona \bp \b^2\rp\a\rp\cona^2\rp\conb \bp
    \b^3\rp\a\rp\conb\rp\cona\rp\conb \bp \a\rp\conb^3 \bp
    \b^4\rp\cona\rp\conb^2
  \end{align}
  as follows.
  \begin{align*}
    \A^2 &= (\a\rp\cona^2\bp\b\rp\a\rp\conb\rp\cona\bp
    \b^2\rp\conb)\rp(\a\rp\cona^2\bp\b\rp\a\rp\conb\rp\cona\bp
    \b^2\rp\conb)&&\text{def $\A$}
    \\ &=((\a\rp\cona\bp\b\rp\a\rp\conb)\rp\cona\bp
    \b^2\rp\conb)\rp(\a\rp\cona^2\bp\b\rp(\a\rp\conb\rp\cona\bp
    \b\rp\conb))&&\text{func dist}
    \\ &=(\a\rp\cona\bp\b\rp\a\rp\conb)\rp\cona^2\bp
    \b^2\rp(\a\rp\conb\rp\cona\bp\b\rp\conb) &&\text{\eqref{Pr}}
    \\ &=(\a\rp\cona^3\bp\b\rp\a\rp\conb\rp\cona^2\bp
    \b^2\rp\a\rp\conb\rp\cona\bp\b^3\rp\conb) &&\text{func dist}
    \\ \con\A^2 &=\a^3\rp\cona\bp\a^2\rp\b\rp\cona\rp\conb\bp
    \a\rp\b\rp\cona\rp\conb^2\bp\b\rp\conb^3 &&\text{\eqref{AA}, con}
    \\ \B^2&=(\ot\id\A)\rp(\ot\id\A)&&\eqref{e1}
    \\ &=\ot\id(\A^2)&&\text{Prop.\ \ref{gg-rule}}
    \\ &=\a\rp\cona\bp\b\rp\A^2\rp\conb &&\text{def $\otimes$}
    \\ &=\a\rp\cona\bp\b\rp(\a\rp\cona^3
    \bp\b\rp\a\rp\conb\rp\cona^2\bp\b^2\rp\a\rp\conb\rp\cona
    \bp\b^3\rp\conb)\rp\conb &&\eqref{AA}
    \\ &=\a\rp\cona\bp\b\rp\a\rp\cona^3\rp\conb
    \bp\b^2\rp\a\rp\conb\rp\cona^2\rp\conb
    \bp\b^3\rp\a\rp\conb\rp\cona\rp\conb\bp\b^4\rp\conb^2 &&\text{func
      dist} \\ \B^2\rp\C &= (\a\rp\cona\bp\b\rp\a\rp\cona^3\rp\conb\bp
    \b^2\rp\a\rp\conb\rp\cona^2\rp\conb\bp
    \b^3\rp\a\rp\conb\rp\cona\rp\conb\bp\b^4\rp\conb^2)
    \\ &\qquad\rp(\a\rp\conb^2\bp\b\rp\a\rp\cona\bp
    \b^2\rp\cona\rp\conb) &&\text{\eqref{BB}, def $\C$}
    \\ &=(\a\rp\cona\bp (\b\rp\a\rp\cona^3\bp
    \b^2\rp\a\rp\conb\rp\cona^2\bp
    \b^3\rp\a\rp\conb\rp\cona\bp\b^4\rp\conb)\rp\conb)
    \\ &\qquad\rp(\a\rp\conb^2\bp\b\rp(\a\rp\cona\bp
    \b\rp\cona\rp\conb)) &&\text{func dist}
    \\ &=\a\rp\conb^2\bp(\b\rp\a\rp\cona^3\bp
    \b^2\rp\a\rp\conb\rp\cona^2\bp
    \b^3\rp\a\rp\conb\rp\cona\bp\b^4\rp\conb)
    \\ &\qquad\rp(\a\rp\cona\bp\b\rp\cona\rp\conb) &&\eqref{Pr}
    \\ &=\a\rp\conb^2\bp((\b\rp\a\rp\cona^2\bp
    \b^2\rp\a\rp\conb\rp\cona
    \bp\b^3\rp\a\rp\conb)\rp\cona\bp\b^4\rp\conb)
    \\ &\qquad\rp(\a\rp\cona\bp\b\rp\cona\rp\conb) &&\text{func dist}
    \\ &=\a\rp\conb^2\bp(\b\rp\a\rp\cona^2\bp\b^2\rp\a\rp\conb\rp\cona
    \bp\b^3\rp\a\rp\conb)\rp\cona\bp\b^4\rp\cona\rp\conb&&\eqref{Pr}
    \\ &=\a\rp\conb^2\bp\b\rp\a\rp\cona^3\bp
    \b^2\rp\a\rp\conb\rp\cona^2\bp\b^3\rp\a\rp\conb\rp\cona\bp
    \b^4\rp\cona\rp\conb &&\text{func dist}
    \\ \C_3&=\B^2\rp\C\rp\con\A^2&&\text{def $\C_3$}
    \\ &=(\a\rp\conb^2\bp\b\rp\a\rp\cona^3\bp
    \b^2\rp\a\rp\conb\rp\cona^2\bp\b^3\rp\a\rp\conb\rp\cona\bp
    \b^4\rp\cona\rp\conb)
    \\ &\qquad\rp(\a^3\rp\cona\bp\a^2\rp\b\rp\cona\rp\conb\bp
    \a\rp\b\rp\cona\rp\conb^2\bp\b\rp\conb^3) &&\text{\eqref{BBC},
      \eqref{AAc}}
    \\ &=((\b\rp\a\rp\cona^2\bp\b^2\rp\a\rp\conb\rp\cona\bp
    \b^3\rp\a\rp\conb)\rp\cona\bp (\a\rp\conb\bp \b^4\rp\cona)\rp\conb)
    \\ &\qquad\rp(\a\rp(\a^2\rp\cona\bp\a\rp\b\rp\cona\rp\conb\bp
    \b\rp\cona\rp\conb^2)\bp\b\rp\conb^3) &&\text{func dist}
    \\ &=(\b\rp\a\rp\cona^2\bp
    \b^2\rp\a\rp\conb\rp\cona\bp\b^3\rp\a\rp\conb)
    \\ &\qquad\rp(\a^2\rp\cona\bp\a\rp\b\rp\cona\rp\conb
    \bp\b\rp\cona\rp\conb^2)\bp (\a\rp\conb\bp \b^4\rp\cona)\rp\conb^3
    &&\eqref{Pr}
    \\ &=((\b\rp\a\rp\cona\bp\b^2\rp\a\rp\conb)\rp\cona\bp\b^3\rp\a\rp\conb)
    \\ &\qquad\rp(\a\rp(\a\rp\cona\bp\b\rp\cona\rp\conb)\bp
    \b\rp\cona\rp\conb^2)\bp\a\rp\conb^4\bp\b^4\rp\cona\rp\conb^3
    &&\text{func dist} \\ &= (\b\rp\a\rp\cona\bp
    \b^2\rp\a\rp\conb)\rp(\a\rp\cona\bp\b\rp\cona\rp\conb)\bp
    \b^3\rp\a\rp\cona\rp\conb^2 \\ &\qquad\bp\a\rp\conb^4\bp
    \b^4\rp\cona\rp\conb^3 &&\eqref{Pr} \\ &=\b\rp\a\rp\cona\bp
    \b^2\rp\a\rp\cona\rp\conb\bp\b^3\rp\a\rp\cona\rp\conb^2\bp
    \a\rp\conb^4\bp\b^4\rp\cona\rp\conb^3&&\eqref{Pr} \\ \C_3\rp\B &=
    (\b\rp\a\rp\cona\bp\b^2\rp\a\rp\cona\rp\conb\bp
    \b^3\rp\a\rp\cona\rp\conb^2\bp\a\rp\conb^4\bp
    \b^4\rp\cona\rp\conb^3) \\ &\qquad\rp(\a\rp\cona\bp
    \b\rp\a\rp\cona^2\rp\conb\bp\b^2\rp\a\rp\conb\rp\cona\rp\conb
    \bp\b^3\rp\conb^2) &&\text{\eqref{BBCA'A'}, def $\B$} \\ &=
    (\b\rp\a\rp\cona\bp (\b^2\rp\a\rp\cona\bp
    \b^3\rp\a\rp\cona\rp\conb\bp\a\rp\conb^3\bp
    \b^4\rp\cona\rp\conb^2)\rp\conb) \\ &\qquad\rp(\a\rp\cona\bp
    \b\rp(\a\rp\cona^2\rp\conb\bp\b\rp\a\rp\conb\rp\cona\rp\conb\bp
    \b^2\rp\conb^2)) &&\text{func dist} \\ &=\b\rp\a\rp\cona\bp
    (\b^2\rp\a\rp\cona\bp\b^3\rp\a\rp\cona\rp\conb\bp\a\rp\conb^3
    \bp\b^4\rp\cona\rp\conb^2) \\ &\qquad\rp(\a\rp\cona^2\rp\conb
    \bp\b\rp\a\rp\conb\rp\cona\rp\conb\bp \b^2\rp\conb^2)&&\eqref{Pr}
    \\ &=\b\rp\a\rp\cona\bp (\b^2\rp\a\rp\cona\bp
    (\b^3\rp\a\rp\cona\bp\a\rp\conb^2\bp \b^4\rp\cona\rp\conb)\rp\conb)
    \\ &\qquad\rp
    (\a\rp\cona^2\rp\conb\bp\b\rp(\a\rp\conb\rp\cona\rp\conb\bp
    \b\rp\conb^2))&&\text{func dist} \\ &=\b\rp\a\rp\cona\bp
    \b^2\rp\a\rp\cona^2\rp\conb\bp (\b^3\rp\a\rp\cona\bp
    \a\rp\conb^2\bp\b^4\rp\cona\rp\conb) \\ &\qquad
    \rp(\a\rp\conb\rp\cona\rp\conb\bp{}\b\rp\conb^2)&&\eqref{Pr}
    \\ &=\b\rp\a\rp\cona\bp\b^2\rp\a\rp\cona^2\rp\conb \\ &\qquad \bp
    (\b^3\rp\a\rp\cona\bp (\a\rp\conb\bp
    \b^4\rp\cona)\rp\conb)\rp(\a\rp\conb\rp\cona\rp\conb\bp
    \b\rp\conb^2) &&\text{func dist} \\ &=\b\rp\a\rp\cona\bp
    \b^2\rp\a\rp\cona^2\rp\conb\bp\b^3\rp\a\rp\conb\rp\cona\rp\conb \bp
    (\a\rp\conb\bp\b^4\rp\cona)\rp\conb^2 &&\eqref{Pr} \\ &=
    \b\rp\a\rp\cona\bp\b^2\rp\a\rp\cona^2\rp\conb\bp
    \b^3\rp\a\rp\conb\rp\cona\rp\conb\bp\a\rp\conb^3\bp
    \b^4\rp\cona\rp\conb^2 &&\text{func dist}
  \end{align*}
  Then \eqref{ta4} holds because $\C_3\rp\B=\X_2\rp\C_2$ by
  \eqref{BBCA'A'B} and \eqref{ABA'BCA'}.
\end{proof}

\section{Thompson's monoid $\M$}
\label{s24}
The generators of $\M$ are called $\R$, $\P$, $\P_0$, $\R_0$, $\U$,
$\K$, and $\L$ in Brin's notes. Thompson used $\{\P,\R,\K,\U\}$ as a
generating set in his 2004 talk and $\{\P_0,\R_0,\K,\U\}$ in 2008.  We
show next that Thompson's two sets generate the same monoid.
\begin{prop}
  \label{same}
  Assume $\a,\b$ satisfy \eqref{Q}, \eqref{D}, and \eqref{U}. The same
  monoid is generated in $\gc{Fn}(\AA)=\<\Fn\AA,\rp,\id\>$ by both
  $\{\P,\R,\K,\U\}$ and $\{\P_0,\R_0,\K,\U\}$.
\end{prop}
\begin{proof}
  To see that $\{\P,\R,\K,\U\}$ is generated by $\{\P_0,\R_0,\K,\U\}$,
  it is enough to note that $\P=\U\rp\P_0\rp\K$ and
  $\R=\U\rp\R_0\rp\K$ by Prop.\ \ref{Ux0K}.  For the other direction,
  we present only a proof in parenthetical notation, showing for two
  equations that the action of both sides on trees is the same.  Each
  proof begins with one application of $\U$. Cyclic order is preserved
  by $\P$ and $\R$, even with $\K$ admitted, so $\U$ is needed in
  order to change cyclic order, as must be done to obtain
  $\P_0=(01)2\to(10)2$. It seems likely that $\U$ is needed for $\R_0$
  as well. In any case, it is used just once in what follows.

  Start with an input and make another copy, set in boldface type to
  make it easier to track.  Use $\R$ to eliminate the rightmost
  parenthesis until a single digit remains. Use $\P$ to flip that
  digit to the other side or delete it using $\K$, and then apply $\R$
  again to eliminate rightmost parentheses until another digit is
  isolated and then flipped or deleted.

  Start with three digits for $\P_0$, four digits for $\R_0$.  Copy to
  get six or eight, choose three or four as targets, delete the other
  three or four via the procedure described above, and see whether the
  letters are reassociated as desired. For $\R_0$, the target choice
  $\0\123$ happens to work, as shown below, but $\0123$ does not.
  First we have
  \begin{align*}
    \P_0=\U\rp\R\rp\P\rp\R^2\rp\K\rp\P\rp\R^2\rp\K\rp\R\rp\P\rp\R\rp\K\rp\R
  \end{align*}
  because
  \begin{align*}
    (01)2&\rrr\U((\0\1)\2)((01)2) \rrr\R (((\0\1)\2)(01))2 \rrr\P
    2(((\0\1)\2)(01)) \\ &\rrr\R (2((\0\1)\2))(01) \rrr\R
    ((2((\0\1)\2))0)1 \rrr\K (2((\0\1)\2))0 \\ &\rrr\P 0(2((\0\1)\2))
    \rrr\R (02)((\0\1)\2) \rrr\R ((02)(\0\1))\2 \\ &\rrr\K (02)(\0\1)
    \rrr\R ((02)\0)\1 \rrr\P \1((02)\0) \\ &\rrr\R (\1(02))\0 \rrr\K
    \1(02) \rrr\R (\10)2
  \end{align*}
  Recall $\R_0=(0(12))3\mapsto((01)2)3$.
  \begin{align*}
    \R_0=\U\rp\R\rp\P\rp\R^3\rp\P\rp\R\rp\K\rp\R\rp\K\rp\R^2\rp\K\rp\R^2
    \rp\K\rp\P\rp\R\rp\P\rp\R^2
  \end{align*}
  because
  \begin{align*}
    (0(12))3&\rrr\U\((\0(\1\2))\3\)\((0(12))3\)
    \rrr\R\(((\0(\1\2))\3)(0(12))\)3
    \\ &\rrr\P3\(((\0(\1\2))\3)(0(12))\)
    \rrr\R\(3((\0(\1\2))\3)\)\(0(12)\)
    \\ &\rrr\R\((3((\0(\1\2))\3))0\)\(12\)
    \rrr\R\(((3((\0(\1\2))\3))0)1\)2
    \\ &\rrr\P2\(((3((\0(\1\2))\3))0)1\)\rrr\R\(2((3((\0(\1\2))\3))0)\)1
    \\ &\rrr\K2\((3((\0(\1\2))\3))0\) \rrr\R\(2(3((\0(\1\2))\3))\)0
    \\ &\rrr\K2\(3((\0(\1\2))\3)\) \rrr\R\(23\)\((\0(\1\2))\3\)
    \\ &\rrr\R\((23)(\0(\1\2))\)\3\rrr\K\(23\)\(\0(\1\2)\)
    \rrr\R\((23)\0\)\(\1\2\)
    \\ &\rrr\R((23)\0)\1)\2\rrr\K((23)\0)\1\rrr\P\1((23)\0)\rrr\R(\1(23))\0
    \\ &\rrr\P \0(\1(23)) \rrr\R (\0\1)(23) \rrr\R ((\0\1)2)3
  \end{align*}
\end{proof}

\begin{prop}
  \label{presM}
  Assume $\a,\b$ satisfy \eqref{Q}, \eqref{D}, and \eqref{U}. Let
  \begin{align*}
    \K&=\a, \qquad \L=\b, & \U&=\cona\bp\conb,
    \\ \P&=\a\rp\conb\bp\b\rp\cona,
    & \P_0&=\ot\P\id=\a\rp\P\rp\cona\bp\b\rp\conb,
    \\ \R&=\a\rp\cona^2\bp\b\rp\a\rp\conb\rp\cona\bp\b^2\rp\conb,
    & \R_0&=\ot\R\id=\a\rp\R\rp\cona\bp\b\rp\conb.
  \end{align*}
  Then $\P_0$, $\R_0$, $\K$, and $\U$ satisfy the relations for $\M$
  in \SS\ref{s13a}.
\end{prop}
\begin{proof} The invertibility relations hold because
  $\P\rp\P=\id$ was proved in Prop.\ \ref{perms}, while
  $(\R\rp\P)^3=\id$ and $(\P\rp\R)^3=\id$ have proofs in parenthetical
  notation:
  \begin{align*}
    0(12) \rrr\R (01)2 \rrr\P 2(01) \rrr\R (20)1 \rrr\P 1(20) \rrr\R
    (12)0 \rrr\P 0(12), \\ (01)2 \rrr\P 2(01) \rrr\R (20)1 \rrr\P 1(20)
    \rrr\R (12)0 \rrr\P 0(12) \rrr\R (01)2.
  \end{align*}
  A direct equational proof is not too long in this case. We derive
  $\P\rp\R\rp\P\rp\R=\con{\P\rp\R}$, which is equivalent to both
  relations.
  \begin{align}
    \P\rp\R &=(\a\rp\conb\bp\b\rp\cona)\rp(\a\rp\cona^2 \bp
    \b\rp\a\rp\conb\rp\cona \bp \b^2\rp\conb)&&\text{defs $\P$, $\R$}
    \\\notag &=(\a\rp\conb\bp\b\rp\cona)\rp(\a\rp\cona^2 \bp
    \b\rp(\a\rp\conb\rp\cona \bp \b\rp\conb))&&\text{func dist} \\\notag
    &= \b\rp\cona^2\bp\a\rp(\a\rp\conb\rp\cona\bp\b\rp\conb)
    &&\eqref{Pr} \\\notag &=
    \b\rp\cona^2\bp\a^2\rp\conb\rp\cona\bp\a\rp\b\rp\conb
    &&\text{func dist}
  \end{align}
  \begin{align}
    \P\rp\R\rp\P
    &=(\b\rp\cona^2\bp\a^2\rp\conb\rp\cona\bp\a\rp\b\rp\conb)
    \rp(\a\rp\conb\bp\b\rp\cona) &&\text{(17), def $\P$} \\\notag
    &=((\b\rp\cona \bp \a^2\rp\conb)\rp\cona \bp \a\rp\b\rp\conb)
    \rp(\a\rp\conb \bp \b\rp\cona) &&\text{func dist} \\\notag
    &=(\b\rp\cona \bp \a^2\rp\conb)\rp\conb \bp \a\rp\b\rp\cona
    &&\eqref{Pr}
  \end{align}
  \begin{align*}
    &(\P\rp\R\rp\P)\rp\R \\ &=((\b\rp\cona \bp \a^2\rp\conb)\rp\conb \bp
    \a\rp\b\rp\cona)\rp(\a\rp\cona^2 \bp \b\rp\a\rp\conb\rp\cona \bp
    \b^2\rp\conb) &&\text{(18), def $\R$} \\ &=((\b\rp\cona \bp
    \a^2\rp\conb)\rp\conb \bp \a\rp\b\rp\cona)\rp(\a\rp\cona^2 \bp
    \b\rp(\a\rp\conb\rp\cona \bp \b\rp\conb)) &&\text{func dist}
    \\ &=\a\rp\b\rp\cona^2 \bp (\b\rp\cona \bp \a^2\rp\conb)
    \rp(\a\rp\conb\rp\cona \bp \b\rp\conb) &&\eqref{Pr}
    \\ &=\a\rp\b\rp\cona^2 \bp \b\rp\conb\rp\cona \bp \a^2\rp\conb
    &&\eqref{Pr} \\ &=\con{\P\rp\R}&&\text{(17), con}
  \end{align*}
  The commutativity relations hold for arbitrary elements $\x$ and $\y$
  of $\AA$. By the definitions of $\otimes$, $\x_0$, and $\x_1$ we have
  $\x_0=\ot\x\id$ and $\x_1=\ot\id\x$.  Clearly
  $\x_0\rp\y_1=\y_1\rp\x_0$ because, by parallel composition
  Prop.\ \ref{gg-rule},
  \begin{align*}
    \x_0\rp\y_1=(\ot\x\id)\rp(\ot\id\y)=\ot\x\y =
    (\ot\id\y)\rp(\ot\x\id)=\y_1\rp\x_0.
  \end{align*}
  The splitting relations, $\x\rp\U=\U\rp\x_0\rp\x_1$, are proved for a
  functional $\x\in\Fn\AA$ as follows.
  \begin{align*}
    \U\rp\x_0\rp\x_1&=(\cona\bp\conb)\rp(\ot\x\id)\rp(\ot\id\x)&&\text{defs}
    \\ &=(\cona\bp\conb)\rp(\ot\x\x)&&\text{Prop.\ \ref{gg-rule}}
    \\ &=(\cona\bp\conb)\rp(\a\rp\x\rp\cona\bp\b\rp\x\rp\conb)&&\text{def
      $\otimes$}
    \\ &=(\id\rp\cona\bp\id\rp\conb)\rp(\a\rp\x\rp\cona\bp\b\rp\x\rp\conb)
    &&\text{id} \\ &=\id\rp\x\rp\cona\bp\id\rp\x\rp\conb&&\eqref{Pr}
    \\ &=\x\rp\cona\bp\x\rp\conb&&\text{id}
    \\ &=\x\rp(\cona\bp\conb)&&\text{func dist, $\x\in\Fn\AA$}
    \\ &=\x\rp\U&&\text{def $\U$}
  \end{align*}
  The reconstruction relations $\x=\U\rp\x_0\rp\x_1\rp\K_0\rp\L_1$ can
  be proved for $\x\in\Fn\AA$ as follows.
  \begin{align*}
    &\U\rp\x_0\rp\x_1\rp\K_0\rp\L_1 \\ &=\x\rp\U\rp\K_0\rp\L_1
    &&\text{splitting relations}
    \\ &=\x\rp\U\rp(\ot\K\L)&&\text{commutativity relations}
    \\ &=\x\rp(\cona\bp\conb) \rp(\a\rp\a\rp\con\a\bp\b\rp\b\rp\conb)
    &&\text{defs $\U$, $\K$, $\L$, $\otimes$}
    \\ &=\x\rp(\id\rp\cona\bp\id\rp\conb)
    \rp(\a\rp(\a\rp\con\a)\bp\b\rp(\b\rp\conb)) &&\text{id, assoc}
    \\ &=\x\rp(\id\rp(\a\rp\con\a)\bp\id\rp(\b\rp\conb))
    &&\text{\eqref{Pr}}
    \\ &=\x\rp(\a\rp\con\a\bp\b\rp\conb)
    &&\text{id}
    \\ &=\x&&\a\rp\con\a\bp\b\rp\conb=\id
  \end{align*}
  The rewriting relations are
  \begin{align*}
    \U\rp\K&=\id,&\P_0\rp\K\rp\K&=\K\rp\L,&\R_0\rp\K\rp\K\rp\K&=\K\rp\K,
    \\ \U\rp\L&=\id,&\P_0\rp\K\rp\L&=\K\rp\K,
    &\R_0\rp\K\rp\K\rp\L&=\K\rp\L\rp\K,
    \\ &&\P_0\rp\L&=\L,&\R_0\rp\K\rp\L&=\K\rp\L\rp\L,
    \\ &&&&\R_0\rp\L&=\L.
  \end{align*}
  We get $\U\rp\K=(\cona\bp\conb)\rp\a=\id$ and
  $\U\rp\L=(\cona\bp\conb)\rp\b=\id$ by Prop.\ \ref{sub0}.  Note that
  $\P\rp1=1$ because $\P$ is permutational, more exactly,
  $1=\id\rp1=\P\rp\con\P\rp1=\P\rp(\con\P\rp1)\leq\P\rp1\leq1$.  It
  follows by Prop.\ \ref{pok}(iv) that $\P_0\rp\L=(\ot\P\id)\rp\b
  =\b=\L$, one of the rewriting relations.  From Prop.\ \ref{pok}(v)
  we get $\P_0\rp\K=(\ot\P\id)\rp\a=\a\rp\P=\K\rp\P$. Together with
  $\P\rp\K=\L$ and $\P\rp\L=\K$ from Prop.\ \ref{swap}, this gives two
  more relations, namely,
  \begin{align*}
    \P_0\rp\K\rp\K&=\K\rp\P\rp\K=\K\rp\L,
    \\ \P_0\rp\K\rp\L&=\K\rp\P\rp\L=\K\rp\K.
  \end{align*}
  From Prop.\ \ref{pok}(iv) and $\R\rp1=1$ we get $\R_0\rp\L=\L$,
  another one of the rewriting relations.  After deriving
  \begin{align*}
    &\R\rp\K\rp\K=\K, &&\R\rp\K\rp\L=\L\rp\K, &&\R\rp\L=\L\rp\L,
  \end{align*}
  the three remaining rewriting relations will follow by multiplying
  on the left by $\K$ and noting that $\R_0\rp\K=\K\rp\R$ by
  Prop.\ \ref{pok}(v). First we show that
  \begin{align}
    \label{rk}
    \R\rp\K&=\a\rp\cona\bp\b\rp\a\rp\conb
  \end{align}
  because
  \begin{align*}
    \R\rp\K&=(\a\rp\cona^2\bp\b\rp\a\rp\conb\rp\cona\bp\b^2\rp\conb)\rp\a
    &&\text{defs $\R$, $\K$} \\\notag
    &=((\a\rp\cona\bp\b\rp\a\rp\conb)\rp\cona\bp
    \b^2\rp\conb)\rp(\a\rp\id\bp\b\rp1)&&\text{func dist, id,
      $\b\rp1=1$} \\\notag &=((\a\rp\cona\bp\b\rp\a\rp\conb)\rp\id \bp
    \b^2\rp1) &&\eqref{Pr} \\\notag &=((\a\rp\cona\bp\b\rp\a\rp\conb)
    \bp \b\rp(\b\rp1)) &&\text{id, assoc} \\\notag
    &=\a\rp\cona\bp\b\rp\a\rp\conb&&\text{$\b\rp1=1$}
  \end{align*}
  Then $\R\rp\K\rp\K=\K$ because
  \begin{align*}
    \R\rp\K\rp\K&=(\a\rp\cona\bp\b\rp\a\rp\conb)\rp\a
    &&\text{\eqref{rk}, def $\K$}
    \\ &=(\a\rp\cona\bp\b\rp\a\rp\conb)\rp(\a\rp\id\bp\b\rp1)
    &&\text{id, $\b\rp1=1$} \\ &=\a\rp\id\bp\b\rp\a\rp1&&\eqref{Pr}
    \\ &=\a\bp\b\rp(\a\rp1)&&\text{id, assoc} \\ &=\a &&\a\rp1=1=\b\rp1
  \end{align*}
  and $\R\rp\K\rp\L=\L\rp\K$ because
  \begin{align*}
    \R\rp\K\rp\L&=(\a\rp\cona\bp\b\rp\a\rp\conb)\rp\b
    &&\text{\eqref{rk}, def $\L$}
    \\ &=(\a\rp\cona\bp\b\rp\a\rp\conb)\rp(\a\rp1\bp\b\rp\id)
    &&\text{id, $\a\rp1=1$}
    \\ &=\a\rp1\bp\b\rp\a\rp\id&&\eqref{Pr}
    \\ &=\b\rp\a&&\text{id, $\a\rp1=1$}
  \end{align*}
  For the last rewriting relation we must first note that
  \begin{align}
    \label{note1}
    1&=(\a\rp\cona\bp\b\rp\a\rp\conb)\rp1
  \end{align}
  because
  \begin{align*}
    1&=\a\rp1\bp\b\rp\a\rp1&&\text{\eqref{D}, assoc}
    \\ &=(\a\rp1\bp\b\rp\a)\rp1&&\text{Prop.\ \ref{i1}}
    \\ &=(\a\rp\cona\rp\b\bp\b\rp\a)\rp1&&\text{\eqref{Q}, assoc}
    \\ &=(\a\rp\cona\bp\b\rp\a\rp\conb)\rp1&&\text{Prop.\ \ref{icyc}}
  \end{align*}
  At last, $\R\rp\L=\L\rp\L$ because
  \begin{align*}
    \R\rp\L&=(\a\rp\cona^2\bp\b\rp\a\rp\conb\rp\cona\bp\b^2\rp\conb)\rp\b
    &&\text{defs $\R$, $\L$}
    \\ &=((\a\rp\cona\bp\b\rp\a\rp\conb)\rp\cona\bp\b^2\rp\conb)
    \rp(\b\rp\id\bp\a\rp1) &&\text{func dist, id, $\a\rp1=1$}
    \\ &=(\a\rp\cona\bp\b\rp\a\rp\conb)\rp1 \bp \b^2\rp\id &&\eqref{Pr}
    \\ &=\b^2&&\text{\eqref{note1}, id}
  \end{align*}
\end{proof}

\section*{{\bf Part III.}}
Definition \ref{tab-orig} (tabularity for relation algebras) and
Theorem \ref{TRA} in \SS\ref{s8} (tabular relation algebras are
representable) are generalized to J-algebras by Definition
\ref{tab-def} and Theorem \ref{tabrra} in Part III. The proof of
Theorem \ref{tabrra} covers both J-algebras and relation
algebras. Throughout this part is useful to remember that, by
Prop.\ \ref{jaisra}, every elementary property of J-algebras also
applies to relation algebras.
\section{Tabularity and partial representations}
\label{s25}
In Definitions \ref{tab-def} and \ref{partrep}, Lemmas \ref{LEMMA2},
\ref{LEMMA3}, \ref{LEMMA5}, and \ref{LEMMA6}, and in Theorem
\ref{tabrra}, $\A$ is the universe of the algebra $\AA$ (rather than
an element, as in \SS\ref{s12a}).
\begin{defn}
  \label{tab-def}
  An algebra $\AA\in\JA\cup\RA$ is {\bf tabular} if, for all
  $\v,\w\in\A$, $\v<\w$ implies there are $\p,\q\in\Fn\AA$ such that
  $0\neq\con\p\rp\q\leq\w$ and $\v\bp\con\p\rp\q=0$.
\end{defn}
By Prop.\ \ref{tab-prop} in \SS\ref{s8}, Definition \ref{tab-orig} and
Definition \ref{tab-def} are equivalent for relation algebras.  The
proof of Prop.\ \ref{tab-prop} uses the fact that $\v<\w$ implies
$\w\bp\min\v\neq0$ in a relation algebra. Since Definition
\ref{tab-def} is intended for J-algebras as well as relation algebras,
the condition $\v<\w$ means $\v\bp\w=\v\neq\w$ according to
Def.\,\ref{d13}. The equivalence of the two notions of tabularity
follows from the observation that $\v\bp\w=\v\neq\w$ implies
$\w\bp\min\v\neq0$ in a relation algebra.
\begin{defn}
  \label{partrep}
  For every $\AA\in\JA\cup\RA$ and $\n\in\omega$ let $\S_\n\AA$ be the
  set of sequences $\f=\<\f_0,\cdots,\f_{\n-1}\>$ such that, for all
  $\i,\j<\n$, $0\neq\f_\i\in\Fn\AA$ and $\f_\i\rp1=\f_\j\rp1$.  For
  every $\f\in\S_\n\AA$ define a \bf partial representation
  $\eu\colon\A\to\pow{\n\times\n}$ by
  $\eu(\x)=\{\<\i,\j\>:\f_\i\rp\x\geq\f_\j,\,\i,\j<\n\}.$
\end{defn}
For any $\f\in\S_\n\AA$, it is useful to note that
$\<\i,\j\>\in\eu(\x)$ iff $\f_\i\rp\x\bp\f_\j=\f_\j$. For example, if
$\<\i,\j\>\in \widehat\f(\x)$ and $\f_\i\rp\x\bp\f_\j=0$ then
$\f_\j=0$, contradicting $\f_\j\neq0$, so
\begin{align*}
  \f_\i\rp\x\bp\f_\j=0\implies\<\i,\j\>\notin\widehat\f(\x).
\end{align*}
The next lemma shows that $\eu$ has several of the properties of a
representation of $\AA$ over $\Re\n$.
\begin{lem}
  \label{LEMMA2}
  Assume $\AA\in\JA\cup\RA$, $\f\in\S_\n\AA$, and $\x,\y\in\A$. Then
  \begin{enumerate}
  \item $\x\leq\y\implies\eu(\x)\subseteq \eu(\y)$,
  \item $\eu(\x)\cap\eu(\y)\subseteq\eu(\x\bp\y)$,
  \item $\eu(\x)|\eu(\x)\subseteq \eu(\x\rp\y)$,
  \item $\eu(\con\x)=\conv{\eu(\x)}$,
  \item $\eu(0)=\emptyset$.
  \end{enumerate}
\end{lem}
\begin{proof}
  \begin{proof}[\rm(i)]
    Assume $\x\leq\y$ and $\<\i,\j\>\in\eu(\x)$. Then
    $\f_\i\rp\y\geq\f_\i\rp\x$ by Prop.\ \ref{p6} and
    $\f_\i\rp\x\geq\f_\j$ by the definition of $\eu$, so
    $\f_\i\rp\y\geq\f_\j$ by Prop.\ \ref{p1}, \ie,
    $\<\i,\j\>\in\eu(\y)$.
  \end{proof}
  \begin{proof}[\rm(ii)]
    Assume $\<\i,\j\>\in\eu(\x)\cap\eu(\y)$. Then
    $\f_\i\rp\x\geq\f_\j$ and $\f_\i\rp\y\geq\f_\j$ by the definition
    of $\eu$, so $\f_\i\rp\x\bp\f_\i\rp\y\geq\f_\j$ by
    Prop.\ \ref{p3}, but $\f_\i\rp\x\bp\f_\i\rp\y=\f_\i\rp(\x\bp\y)$
    by Prop.\ \ref{f-dist}, so $\<\i,\j\>\in\eu(\x\bp\y)$.
  \end{proof}
  \begin{proof}[\rm(iii)]
    Assume $\<\i,\j\>\in\eu(\x)|\eu(\y)$. Then $\<\i,\k\>\in\eu(\x)$
    and $\<\k,\j\>\in\eu(\y)$ for some $\k<\n$, so
    $\f_\i\rp\x\geq\f_\k$ and $\f_\k\rp\y\geq\f_\j$ by the definition
    of $\eu$. We obtain $\<\i,\j\>\in\eu(\x\rp\y)$ because
    $\f_\i\rp(\x\rp\y)=(\f_\i\rp\x)\rp\y\geq\f_\k\rp\y\geq\f_\j$, by
    \eqref{2} and Prop.\ \ref{p6}.
  \end{proof}
  \begin{proof}[\rm(iv)]
    Assume $\<\i,\j\>\in\eu(\con\x)$, which implies
    $\f_\j=\f_\j\bp\f_\i\rp\con\x$. Then we have
    \begin{align*}
      \f_\i&=\f_\i\bp\f_\i\rp1&&\text{Prop.\ \ref{p9}}
      \\ &=\f_\i\bp\f_\j\rp1&&\f\in\S_\n\AA
      \\ &=\f_\i\bp(\f_\j\bp\f_\i\rp\con\x)\rp1&&\f_\j
      =\f_\j\bp\f_\i\rp\con\x
      \\ &=\f_\i\bp(\f_\i\bp\f_\j\rp\x)\rp1&&\text{Prop.\ \ref{cyc1}}
      \\ &\leq\f_\j\rp\x&&\text{Prop.\ \ref{f1}, $\f_\i$ is functional}
    \end{align*}
    so $\<\j,\i\>\in\eu(\x)$, from which we get
    $\<\i,\j\>\in\conv{\eu(\x)}$.  This proves
    $\eu(\con\x)\subseteq\conv{\eu(\x)}$. Substitute $\con\x$ in this
    formula and use $\con{\con\x}=\x$ to conclude that
    $\eu(\x)\subseteq(\eu(\con\x))^{-1}$, so
    $\eu(\con\x)=\conv{\eu(\x)}$.
  \end{proof}
  \begin{proof}[\rm(v)]
    If $\<\i,\j\>\in\eu(0)$ then $\f_\j=\f_\i\rp0\bp\f_\j=0\bp\f_\j=0$
    by \eqref{norm} and \eqref{zero}, contradicting $\f\in\S_\n\AA$.
    Thus, $\eu(0)=\emptyset$.
  \end{proof}
\end{proof}
\section{Two extension lemmas}
\label{s26}
The first of the two extension lemmas in this section is used for the
relation algbraic case in the proof of Theorem \ref{tabrra}. It
involves join and therefore applies only to relation algebras.
Tabularity is not needed.
\begin{lem}
  \label{LEMMA3}
  Assume $\AA\in\RA$, $\i,\j,\m\in\omega$, $\f\in\S_\m\AA$,
  $\x,\y\in\A$, and $\<\i,\j\>\in\eu(\x+\y)$.  Then there is some
  $\g\in\S_\m\AA$ such that
  \begin{enumerate}
  \item $\<\i,\j\>\in\ev(\x)\cup\ev(\y)$,
  \item $\eu(\z)\subseteq\ev(\z)$ for all $\z\in\A$,
  \item if $\k,\ell<\m$, $\z\in\A$, and $\f_\k\rp\z\bp\f_\ell=0$, then
    $\g_\k\rp\z\bp\g_\ell=0$.
  \end{enumerate}
\end{lem}
\begin{proof}
  We get $0\neq\f_\j$ from $\f\in\S_\m\AA$ and
  $\f_\j=\f_\i\rp(\x+\y)\bp\f_\j$ from $\<\i,\j\>\in\eu(\x+\y)$.  In a
  relation algebra, $\bp$ and $\rp$ distribute over $+$, so
  $0\neq\f_\i\rp\x\bp\f_\j+\f_\i\rp\y\bp\f_\j$, which implies that
  $\f_\i\rp\x\bp\f_\j$ and $\f_\i\rp\y\bp\f_\j$ are not both zero.  We
  choose $\r$ to be one of these two elements that is not zero. An
  explicit choice yielding $\r\neq0$ can be made by setting
  \begin{align*}
    \r=\begin{cases}
    \f_\i\rp\x\bp\f_\j &\text{if }\f_\i\rp\x\bp\f_\j\neq0
    \\
    \f_\i\rp\y\bp\f_\j &\text{if }\f_\i\rp\x\bp\f_\j=0.
    \end{cases}
  \end{align*}
  Note that $\r\leq\f_\j$ by Prop.\ \ref{p2}, hence
  $\r\rp1\leq\f_\j\rp1$ by Prop.\ \ref{p6}.  Define $\g\in\A^\m$ by
  letting $\g_\k=\r\rp1\bp\f_\k$ if $\k<\m$.  Then
  $\g_\k\leq\f_\k\in\Fn\AA$ so $\g_\k$ is functional by
  Prop.\ \ref{func}(ii), and
  \begin{align*}
    \r\rp1&=\r\rp1\bp\f_\j\rp1 &&\r\rp1\leq\f_\j\rp1
    \\ &=\r\rp1\bp\f_\k\rp1 &&\f\in\S_\m\AA \\ &=(\r\rp1\bp\f_\k)\rp1
    &&\text{Prop.\ \ref{i1}} \\ &=\g_\k\rp1 &&\text{def $\g_\k$}
  \end{align*}
  To conclude that $\g\in\S_\m\AA$ we need only show the elements of
  $\g$ are not zero.  Suppose, to the contrary, that $0=\g_\k$ for
  some $\k<\m$.  Then $\r=0$ because $\r\leq\r\rp1=\g_\k\rp1=0\rp1=0$
  by Props.\ \ref{p5} and \ref{p9}, contradicting the definition of
  $\r$. Thus, $\g\in\S_\m\AA$.

  For (i), suppose that $\r=\f_\i\rp\x\bp\f_\j$.  Then
  $\g_\j=\r\rp1\bp\f_\j=(\f_\i\rp\x\bp\f_\j)\rp1\bp\f_\j
  \leq\f_\i\rp\x$ by Prop.\ \ref{f1} since $\f_\j\in\Fn\AA$. We also
  have $\g_\j\leq\r\rp1$, so by Prop.\ \ref{i1},
  $\g_\j\leq\r\rp1\bp\f_\i\rp\x=(\r\rp1\bp\f_\i)\rp\x=\g_\i\rp\x$ and
  hence $\<\i,\j\>\in\ev(\x)$.  On the other hand, if
  $\r=\f_\i\rp\y\bp\f_\j$ then $\<\i,\j\>\in\ev(\y)$ by the same
  computation with $\y$ in place of $\x$. Thus, (i) holds.  For (ii),
  assume $\z\in\A$ and $\<\k,\ell\>\in\eu(\z)$, \ie,
  $\f_\ell\leq\f_\k\rp\z$.  Then $\<\k,\ell\>\in\ev(\z)$ because
  \begin{align*}
    \g_\ell=\r\rp1\bp\f_\ell\leq\r\rp1\bp\f_\k\rp\z=
    (\r\rp1\bp\f_\k)\rp\z=\g_\k\rp\z
  \end{align*}
  by Prop.\ \ref{i1}.  For (iii), assume $\k,\ell<\m$, $\z\in\A$, and
  $\f_\k\rp\z\bp\f_\ell=0$.  Then
  $\g_\k\rp\z\bp\g_\ell=(\r\rp1\bp\f_\k)\rp\z\bp\r\rp1\bp\f_\ell
  =\r\rp1\bp\f_\k\rp\z\bp\r\rp1\bp\f_\ell=\r\rp1\bp0=0$ by
  Prop.\ \ref{i1}.
\end{proof}
\begin{lem}
  \label{LEMMA5}
  Assume $\AA\in\JA\cup\RA$, $\AA$ is tabular, $i,j<m\in\omega$,
  $\f\in\S_\m\AA$, $\x,\y\in\A$, and $\<i,j\>\in\eu(\x\rp\y)$. Then
  there is some $\g\in\S_{\m+1}\AA$ such that
  \begin{enumerate}
  \item
    $\<i,j\>\in\ev(\x)|\ev(\y)$,
  \item
    $\eu(\z)\subseteq\ev(\z)$ for every $\z\in\A$,
  \item
    if $\k,\ell<\m$, $\z\in\A$, and $\f_\k\rp\z\bp\f_\ell=0$, then
    $\g_\k\rp\z\bp\g_\ell=0$.
  \end{enumerate}
\end{lem}
\begin{proof}
  If $0=\f_\i\rp\x\bp\f_\j\rp\con\y$ then $\f_\j=0$, contradicting
  $\f\in\S_\m\AA$, because
  \begin{align*}
    \f_\j&=\f_\i\rp(\x\rp\y)\bp\f_\j&&\text{by $\<i,j\>\in\eu(\x\rp\y)$}
    \\ &=(\f_\i\rp\x)\rp\y\bp\f_\j&&\text{assoc}
    \\ &\leq(\f_\i\rp\x\bp\f_\j\rp\con\y)\rp\y&&\text{rot}
    \\ &=0\rp\y=0&&\text{hyp, Prop.\ \ref{p5}}
  \end{align*}
  Therefore, $0<\f_\i\rp\x\bp\f_\j\rp\con\y$. Since $\AA$ is tabular,
  there are functional elements $\p,\q\in\Fn\AA$ such that
  $0\neq\con\p\rp\q<\f_\i\rp\x\bp\f_\j\rp\con\y$.  This implies
  $0\neq\q\bp\p\rp\f_\i\rp\x\bp\p\rp\f_\j\rp\con\y$ because
  \begin{align*}
    0\neq\con\p\rp\q&=\con\p\rp\q\bp\f_\i\rp\x\bp\f_\j\rp\con\y
    \\ &\leq\con\p\rp(\q\bp\p\rp(\f_\i\rp\x\bp\f_\j\rp\con\y))
    &&\text{rot, con}
    \\ &=\con\p\rp(\q\bp\p\rp(\f_\i\rp\x)\bp\p\rp(\f_\j\rp\con\y))
    &&\text{func dist, $\p\in\Fn\AA$}
    \\ &=\con\p\rp(\q\bp\p\rp\f_\i\rp\x\bp\p\rp\f_\j\rp\con\y)
    &&\text{assoc}
  \end{align*}
  Define $\r$ by $\r=\q\bp\p\rp\f_\i\rp\x\bp\p\rp\f_\j\rp\con\y$ and
  $\g\in\A^{m+1}$ by $\g_\m=\r\rp1\bp\q$ and
  $\g_\k=\r\rp1\bp\p\rp\f_\k$ for all $k<m$. Note that $\r\neq0$ and
  $\g_0,\dots,\g_\m\in\Fn\AA$ by Prop.\ \ref{func} since
  $\p,\q,\f_0,\dots,\f_{m-1}\in\Fn\AA$.  We show next that
  $\g_0,\dots,\g_\m$ all have the same domain $\r\rp1$. For $\g_\m$ we
  have
  \begin{align*}
    \g_\m\rp1&=(\r\rp1\bp\q)\rp1&&\text{def $\g_\m$}
    \\ &=\r\rp1\bp\q\rp1&&\text{Prop.\ \ref{i1}}
    \\ &=\r\rp1&&\text{by $\r\leq\q$, mon}
  \end{align*}
  For $\i<\m$, note $\r\rp1\leq\p\rp\f_\i\rp1$ because
  \begin{align*}
    \r\rp1 &\leq\p\rp\f_\i\rp\x\rp1&&\text{mon}
    \\ &=\p\rp\f_\i\rp(\x\rp1)&&\text{assoc}
    \\ &\leq\p\rp\f_\i\rp1&&\text{mon}
  \end{align*}
  so if $\k<\m$ then
  \begin{align*}
    \g_\k\rp1&=(\r\rp1\bp\p\rp\f_\k)\rp1&&\text{def $\g_\k$}
    \\ &=\r\rp1\bp\p\rp\f_\k\rp1&&\text{Prop.\ \ref{i1}}
    \\ &=\r\rp1\bp\p\rp(\f_\k\rp1)&&\text{assoc}
    \\ &=\r\rp1\bp\p\rp(\f_\i\rp1)&&\text{$\f\in\S_\n\AA$}
    \\ &=\r\rp1\bp\p\rp\f_\i\rp1&&\text{assoc}
    \\ &=\r\rp1&&\r\rp1\leq\p\rp\f_\i\rp1
  \end{align*}
  Thus, we have shown $\g_0\rp1=\cdots=\g_\m\rp1=\r\rp1$. Furthermore,
  $\g_\k\neq0$ for every $\k\leq\m$ since otherwise
  $\r\leq\r\rp1=\g_\k\rp1=0\rp1=0$, implying $\r=0$, a contradiction.
  This proves that $\g\in\S_{\m+1}$.  For a proof of
  $\<\i,\m\>\in\ev(\x)$, we note that
  \begin{align*}
    \g_\m&=\r\rp1\bp\q&&\text{def $\g_\m$}
    \\ &=\r\rp1\bp\q\bp\r\rp1&&\text{\eqref{bassoc}--\eqref{idem}}
    \\ &=\r\rp1\bp\q\bp(\q\bp\p\rp\f_\i\rp\x\bp\p\rp\f_\j\rp\con\y)\rp1
    &&\text{def $\r$}
    \\ &\leq\r\rp1\bp\p\rp\f_\i\rp\x\bp\p\rp\f_\j\rp\con\y
    &&\text{$\q\in\Fn\AA$, Prop.\ \ref{f1}}
    \\ &\leq\r\rp1\bp\p\rp\f_\i\rp\x&&\text{Prop.\ \ref{p2}}
    \\ &=(\r\rp1\bp\p\rp\f_\i)\rp\x&&\text{Prop.\ \ref{i1}}
    \\ &=\g_\i\rp\x&&\text{def $\g_\i$}
  \end{align*}
  and $\<m,j\>\in\ev(\y)$ because
  \begin{align*}
    \g_\j&=\r\rp1\bp\p\rp\f_\j&&\text{def $\g_\j$}
    \\ &=\r\rp1\bp\p\rp\f_\j\bp\r\rp1
    &&\text{\eqref{bassoc}--\eqref{idem}}
    \\ &=\r\rp1\bp\p\rp\f_\j\bp(\q\bp\p\rp\f_\i\rp\x\bp
    \p\rp\f_\j\rp\con\y)\rp1 &&\text{def $\r$}
    \\ &\leq\r\rp1\bp\p\rp\f_\j\bp(\q\bp\p\rp\f_\j\rp\con\y)\rp1&&\text{mon}
    \\ &\leq\r\rp1\bp\p\rp\f_\j\bp(\p\rp\f_\j\bp\q\rp\y)\rp1
    &&\text{Prop.\ \ref{cyc1}}
    \\ &\leq\r\rp1\bp\q\rp\y&&\text{Prop.\ \ref{f1},
      $\p\rp\f_\j\in\Fn\AA$}
    \\ &=(\r\rp1\bp\q)\rp\y&&\text{Prop.\ \ref{i1}}
    \\ &=\g_\m\rp\y&&\text{def $\g_\m$}
  \end{align*}
  From $\<i,m\>\in\ev(\x)$ and $\<m,j\>\in\ev(\y)$ we get
  $\<i,j\>\in\ev(\x)|\ev(\y)$. Thus, part (i) holds.

  For part (ii), suppose $\z\in\A$ and $\<k,\ell\>\in\eu(\z)$, \ie,
  $\f_\k\rp\z\geq\f_\ell$ and $k,\ell<\m$.  From $\f_\k\rp\z\geq
  \f_\ell$ we get $\p\rp\f_\k\rp\z\geq\p\rp\f_\ell$ by Prop.\ \ref{p6}
  and \eqref{2}, so $\g_\k\rp\z\geq\g_\ell$ because, by
  Props.\ \ref{i1} and \ref{p3}, $\g_\k\rp\z =
  (\r\rp1\bp\p\rp\f_\k)\rp\z =\r\rp1\bp\p\rp\f_\k\rp\z\geq
  \r\rp1\bp\p\rp\f_\ell=\g_\ell.$
 
  For part (iii), if $\z\in\A$, $k,\ell<\m$, and
  $\f_\k\rp\z\bp\f_\ell=0$, then $\g_\k\rp\z\bp\g_\ell=0$ because
  \begin{align*}
    \g_\k\rp\z\bp\g_\ell
    &=(\r\rp1\bp\p\rp\f_\k)\rp\z\bp\r\rp1\bp\p\rp\f_\ell&&\text{defs}
    \\ &\leq\p\rp\f_\k\rp\z\bp\p\rp\f_\ell&&\text{mon}
    \\ &=\p\rp(\f_\k\rp\z)\bp\p\rp\f_\ell&&\text{assoc}
    \\ &=\p\rp(\f_\k\rp\z\bp\f_\ell)&&\text{func dist, $\p\in\Fn\AA$}
    \\ &=\p\rp0&&\text{hyp} \\ &=0&&\eqref{norm}
  \end{align*}
\end{proof}

\section{Key lemma and main result}
\label{s27}
In the conclusion of the key lemma, note that the identity condition
involving $\rho(\id)$ need not hold.
\begin{lem}
  \label{LEMMA6}
  Assume $\AA\in\JA\cup\RA$, $\AA$ is tabular, and $\v,\w\in\A$.  If
  $\v<\w$ then there is a set $\U$ and a function
  $\rho\colon\A\rightarrow\pow{\U^2}$ such that for all $\x,\y\in\A$,
  \begin{align*}
    \rho(\v)&\neq\rho(\w), \\ \rho(0)&=\emptyset,
    \\ \rho(\x\bp\y)&=\rho(\x)\cap\rho(\y),
    \\ \rho(\x\rp\y)&=\rho(\x)|\rho(\y),
    \\ \rho(\con\x)&=\conv{\rho(\x)},
    \\ \text{if $\AA\in\RA$ then }\rho(\x+\y)&=\rho(\x)\cup\rho(\y).
  \end{align*}
\end{lem}
\begin{proof}
  For every finite $X\subseteq\A$ we construct a function $\hx\colon
  \A\to\pow{\omega\times\omega}$ such that for all $\x,\y\in\A$,
  \begin{enumerate}
  \item
    \label{k1} $\<0,1\>\in\hx(\w)$,
  \item
    \label{k2} $\<0,1\>\notin\hx(\v)$,
  \item
    \label{k3} $\hx(0)=\emptyset$,
  \item
    \label{k4} $\hx(\x\bp\y)=\hx(\x)\cap\hx(\y)$,
  \item
    \label{k6} $\hx(\con\x)=\conv{\hx\x}$,
  \item
    \label{k5} if $\x,\y\in X$ then $\hx(\x\rp\y)=\hx(\x)|\hx(\y)$,
  \item
    \label{k7} if $\AA\in\RA$ and $\x,\y\in X$ then
    $\hx(\x+\y)=\hx(\x)\cup\hx(\y)$.
  \end{enumerate}
  Let $\tau\colon\omega\to\omega\times\omega\times{X}\times{X}$ be an
  $\omega$-sequence in which every element of
  $\omega\times\omega\times{X}\times{X}$ occurs infinitely many times.
  This is possible because $X$ is finite so
  $\omega\times\omega\times{X}\times{X}$ is countable.  Set
  $\EE=\bigcup_{m\in\omega}\S_\m\AA$. From $\tau$ we will construct an
  $\omega$-sequence $\f\colon\omega\to\EE$ such that for all
  $\n\in\omega$,
  \begin{align}
    \label{(a)}
    &\<0,1\>\in\widehat{\f_\n}(\w), \\
    \label{(b)}
    &(\f_\n)_0\rp\v\bp(\f_\n)_1=0, \\
    \label{(c)}
    &\eu_{\n-1}(\z)\subseteq\eu_\n(\z)\text{ if $0<n$ and $\z\in\A$}.
  \end{align}
  Define $\hx\colon\A\to\pow(\omega\times\omega)$ by setting
  \begin{align*}
    \hx(\x)=\bigcup_{\n\in\omega}\widehat{\f_\n}(\x)=\bigcup_{\n\in\omega}
    \{\<\i,\j\>:(\f_\n)_\i\rp\x\bp(\f_\n)_\j=(\f_\n)_\j,\,i,j<n\}
  \end{align*}
  for every $\x\in\A$.  Next we explain why $\hx$ has the properties
  \eqref{k1}--\eqref{k6}.  Property \eqref{k1} follows from \eqref{(a)} and
  the definition of $\hx$: since every $\widehat{\f_\n}(\w)$ contains
  $\<0,1\>$ their union also contains it.  Similarly, \eqref{k2}
  follows from \eqref{(b)} and the definition of $\hx$, for if
  $\<0,1\>\in\hx(\v)$ then for some $n\in\omega$,
  $\<0,1\>\in\widehat{\f_\n}(\v)$, \ie,
  $(\f_\n)_0\rp\v\bp(\f_\n)_1=(\f_\n)_1$, but then $(\f_\n)_1=0$ by
  \eqref{(b)}, contradicting $\f_\n\in\EE$, which requires every element of
  $\f_\n$ to be non-zero.  By Lemma \ref{LEMMA2}(v), each partial
  representation $\widehat{\f_\n}$ sends $0$ to the empty set, so
  their union does and \eqref{k3} holds.  By Lemma \ref{LEMMA2}(i),
  each partial representation $\widehat{\f_\n}$ is monotonic, so $\hx$
  is also monotonic because it is the union of monotonic
  functions. This gives us one direction of \eqref{k4}, namely,
  $\hx(\x\bp\y)\subseteq\hx(\x)\cap\hx(\y)$. For the other direction,
  suppose $\<\u,\v\>\in\hx(\x)$ and $\<\u,\v\>\in\hx(\y)$.  Then there
  are $\k,\ell\in\omega$ such that $\<\u,\v\>\in\widehat{\f_\k}(\x)$
  and $\<\u,\v\>\in\widehat{\f_\ell}(\y)$.  Let $\m=\max(\k,\ell)$.
  From \eqref{(c)} it follows that $\widehat{\f_\k}(\x)\subseteq
  \widehat{\f_\m}(\x)$ and $\widehat{\f_\ell}(\y)\subseteq
  \widehat{\f_\m}(\y)$, so by Lemma \ref{LEMMA2}(ii),
  \begin{align*}
    \<\u,\v\>\in\widehat{\f_\k}(\x)\cap\widehat{\f_\ell}(\y)
    \subseteq\widehat{\f_\m}(\x)\cap\widehat{\f_\m}(\y)
    \subseteq\widehat{\f_\m}(\x\bp\y)\subseteq\hx(\x\bp\y).
  \end{align*}
  Property \eqref{k6} holds because, by Lemma \ref{LEMMA2}(iv) and the
  distributivity of converse over arbitrary unions,
  \begin{align*}
    \hx(\con\x)=\bigcup_{\n\in\omega}\widehat{\f_\n}(\con\x)
    =\bigcup_{\n\in\omega}\conv{\widehat{\f_\n}(\x)}
    =\conv{\bigcup_{\n\in\omega}\widehat{\f_\n}(\x)}=\conv{\hx(\x)}.
  \end{align*}
  One direction of \eqref{k5} holds because
  \begin{align*}
    \hx(\x)|\hx(\y)
    &=\bigcup_{\n\in\omega}\widehat{\f_\n}(\x)|\bigcup_{\ell\in\omega}
    \widehat{\f_\ell}(\y)
    \\ &=\bigcup_{\k,\ell\in\omega}\widehat{\f_\k}(\x)|\widehat{\f_\ell}(\y)
    &&\text{composition distributes}
    \\ &\subseteq\bigcup_{m\in\omega}\widehat{\f_\m}(\x)|\widehat{\f_\m}(\y)
    &&\text{\eqref{(c)}}
    \\ &\subseteq\bigcup_{m\in\omega}\widehat{\f_\m}(\x\rp\y)=\hx(\x\rp\y)
    &&\text{Lemma \ref{LEMMA2}(iii)}
  \end{align*}
  The other direction of \eqref{k5} depends on the construction of
  $\f$ and will be treated later.  Similarly, one direction of
  \eqref{k7} can be deduced now while the other depends on the
  construction of $\f$ and will be treated later.  Assume $\AA\in\RA$.
  Then $\x\leq\x+\y$ and $\y\leq\x+\y$ so by Lemma \ref{LEMMA2}(i) we
  get, for all $\n\in\omega$, $\widehat{\f_\n}(\x)\subseteq
  \widehat{\f_\n}(\x+\y)$ and $\widehat{\f_\n}(\y)\subseteq
  \widehat{\f_\n}(\x+\y)$, hence
  $\widehat{\f_\n}(\x)\cup\widehat{\f_\n}(\y)\subseteq
  \widehat{\f_\n}(\x+\y)$, so
  \begin{align*}
    \hx(\x)\cup\hx(\y) &=\bigcup_{\n\in\omega}\widehat{\f_\n}(\x)\cup
    \bigcup_{\n\in\omega}\widehat{\f_\n}(\y)
    \\ &=\bigcup_{\n\in\omega}\(\widehat{\f_\n}(\x)\cup\widehat{\f_\n}(\y)\)
    \subseteq\bigcup_{\n\in\omega}\widehat{\f_\n}(\x+\y) =\hx(\x+\y).
  \end{align*}\par
  Now we begin the construction of the sequence $\f$ of partial
  representations, starting with $\f_0$. Since $\v<\w$ and $\AA$ is
  tabular, there are functional elements $\p,\q\in\Fn\AA$ such that
  \begin{align}
    \label{vw}
    \v\bp\con\q\rp\p=0\neq\con\q\rp\p\leq\w.
  \end{align}
  Let $\x=\p\rp\con\w\bp\q$, $\y=\p\bp\q\rp\w$, and set
  $\f_0=\<\x,\y\>$.  By \eqref{6}--\eqref{rot}, from
  $0\neq\con\q\rp\p\leq\w$ we get
  \begin{align*}
    0\neq\con\q\rp\p\bp\w
    \leq(\w\rp\con\p\bp\con\q)\rp(\p\bp\q\rp\w)=\con{\x}\rp\y,
  \end{align*}
  so $\x\neq0\neq\y$ by \eqref{norm} and Prop.\ \ref{p5}. Also,
  $\x,\y\in\Fn\AA$ by Prop.\ \ref{func}(ii) and $\x\rp1=\y\rp1$ by
  Prop.\ \ref{cyc1}, so we have shown $\f_0=\<\x,\y\>\in\S_2\AA$.  For
  \eqref{(a)} we note that, by Prop.\ \ref{p8}, $(\f_0)_1=\y=\p\bp\q\rp\w
  \leq(\p\rp\con\w\bp\q)\rp\w =\x\rp\w=(\f_0)_0\rp\w$, so
  $\<0,1\>\in\widehat{\f_0}(\w)$.  Condition \eqref{(b)} holds for $\n=0$
  because
  \begin{align*}
    (\f_0)_0\rp\v\bp(\f_0)_1&\leq(\p\rp\con\w\bp\q)\rp\v\bp\p
    &&\text{mon}
    \\ &\leq(\p\rp\con\w\bp\q)\rp(\v\bp\con{\p\rp\con\w\bp\q}\rp\p)
    &&\text{rot}
    \\ &=(\p\rp\con\w\bp\q)\rp(\v\bp(\w\rp\con\p\bp\con\q)\rp\p)
    &&\text{con}
    \\ &\leq(\p\rp\con\w\bp\q)\rp(\v\bp\con\q\rp\p)
    &&\text{mon}
    \\ &=(\p\rp\con\w\bp\q)\rp0=0
    &&\text{\eqref{vw}, \eqref{norm}}
  \end{align*}
  Condition \eqref{(c)} does not apply when $\n=0$. This completes the
  initial case in the construction of the sequence
  $\f=\<\f_0,\f_1,\f_2,\cdots\>$.

  Assume $\mu\in\omega$, $\tau_\mu=\<\i,\j,\x,\y\>$, and
  $\f_0,\f_1,\f_2,\cdots,\f_{2\mu}\in\EE$ have been selected so that
  \eqref{(a)}--\eqref{(c)} hold for all $\n\leq2\mu$.  Then
  $\f_{2\mu+1}$ and $\f_{2\mu+2}$ are chosen as follows.
  \begin{enumerate}
  \item[Case 1a.]  If $\AA\in\JA$ then
    $\f_{2\mu+1}=\f_{2\mu}$.
  \item[Case 1b.]  If $\AA\in\RA$ and
    $\<\i,\j\>\notin\widehat{\f_{2\mu}}(\x+\y)$ then
    $\f_{2\mu+1}=\f_{2\mu}$.
  \item[Case 1c.]  Suppose $\AA\in\RA$ and
    $\<\i,\j\>\in\widehat{\f_{2\mu}}(\x+\y)$.  By Lemma \ref{LEMMA3}
    there is some $\g\in\EE$ such that
    \begin{enumerate}
    \item
      $\<\i,\j\>\in\ev(\x)\cup\ev(\y)$,
    \item
      $\widehat{\f_{2\mu}}(\z)\subseteq\ev(\z)$ for all $\z\in\A$,
    \item
      if $\k,\ell<m$, $\z\in\A$, and
      $(\f_{2\mu})_\k\rp\z\bp(\f_{2\mu})_\ell=0$, then
      $\g_\k\rp\z\bp\g_\ell=0$.
    \end{enumerate}
    In Case 1c, set $\f_{2\mu+1}=\g$.
  \end{enumerate}
  Note that
  $\widehat{\f_{2\mu}}(\z)\subseteq\widehat{\f_{2\mu+1}}(\z)$ for all
  $\z\in\A$, either because $\f_{2\mu}=\f_{2\mu+1}$ or by condition
  1c(b).  Therefore, \eqref{(c)} holds when $\n=2\mu+1$.  We get
  $\<0,1\>\in\widehat{\f_{2\mu}}(\w)$ and
  $(\f_{2\mu})_0\rp\v\bp(\f_{2\mu})_1=0$ from \eqref{(a)} and
  \eqref{(b)} when $\n=2\mu$.  Therefore,
  $\<0,1\>\in\widehat{\f_{2\mu+1}}(\w)$ by \eqref{(c)} with
  $\n=2\mu+1$ and $\z=\w$, and
  $(\f_{2\mu+1})_0\rp\v\bp(f_{2\mu+1})_1=0$ either because
  $\f_{2\mu}=\f_{2\mu+1}$ or by condition 1c(c) with $\k=0$, $\ell=1$,
  and $\z=\v$.  Thus, \eqref{(a)} and \eqref{(b)} hold when
  $\n=2\mu+1$.  Next, choose $f_{2\mu+2}$, assuming
  \eqref{(a)}--\eqref{(c)} hold for all $\n\leq2\mu+1$.
  \begin{enumerate}
  \item[Case 2a.]  If
    $\<\i,\j\>\notin\widehat{\f_{2\mu+1}}(\x\rp\y)$ then
    $f_{2\mu+2}=f_{2\mu+1}$.
  \item[Case 2b.]  Assume
    $\<\i,\j\>\in\widehat{\f_{2\mu+1}}(\x\rp\y)$, $\m\in\omega$, and
    $\widehat{\f_{2\mu+1}}\in\S_\m\AA$.  By Lemma \ref{LEMMA5} there
    is some $\g\in\S_{\m+1}\AA$ such that
    \begin{enumerate}
    \item
      $\<\i,\j\>\in\widehat\g(\x)|\widehat\g(\y)$,
    \item
      $\widehat{\f_{2\mu+1}}(\z)\subseteq\widehat\g(\z)$ for all
      $\z\in\A$,
    \item
      if $\k,\ell<m$, $\z\in\A$, and
      $(\f_{2\mu+1})_\k\rp\z\bp(\f_{2\mu+1})_\ell=0$, then
      $\g_\k\rp\z\bp\g_\ell=0$.
    \end{enumerate}
    In Case 2b, set $\f_{2\mu+2}=\g$.
  \end{enumerate}
  Then \eqref{(a)}--\eqref{(c)} hold when $\n=2\mu+2$ by the same
  argument given above for $\n=2\mu+1$ but with conditions 2b(b) and
  2b(c) in place of 1c(b) and 1c(c).  This completes construction of
  $\f_{2\mu+1}$ and $\f_{2\mu+2}$ from $\f_{2\mu}$.

  We can now complete the proofs of \eqref{k5} and \eqref{k7}. We did
  one direction of \eqref{k5} above. For the other direction, assume
  $\x,\y\in X$ and $\<i,j\>\in\hx(\x\rp\y)$, so
  $\<i,j\>\in\widehat{\f_\m}(\x\rp\y)$ for some $m\in\omega$.  By
  \eqref{(c)}, we have $\<\i,\j\>\in\widehat{\f_\n}(\x\rp\y)$ for all
  larger $\n\geq\m$.  The quadruple $\<\i,\j,\x,\y\>$ occurs in $\tau$
  infinitely many times, so $\tau_\mu=\<\i,\j,\x,\y\>$ and
  $\<\i,\j\>\in\widehat{\f_{2\mu}}(\x\rp\y)$ for some $\mu\geq\m$. By
  condition (a) in Case~2b we have
  $\<i,j\>\in\widehat{\f_{2\mu+2}}(\x)|\widehat{\f_{2\mu+2}}(\y)$, so
  there is a `witness' $k\in\omega$ such that
  $\<\i,\k\>\in\widehat{\f_{2\mu+2}}(\x)$ and
  $\<\k,\j\>\in\widehat{\f_{2\mu+2}}(\y)$. By the definition of $\hx$,
  $\<\i,\k\>\in\hx(\x)$ and $\<\k,\j\>\in\hx(\y)$, yielding
  $\<\i,\j\>\in\hx(\x)|\hx(\y)$.  Thus, a countable set of witnesses
  is provided by the construction. It may happen that all these
  witnesses are destined to be the same element, as might be the case
  when a functional element is involved. That equality will emerge
  only at a later stage in the proof, when we factor out by the
  equivalence relation that represents the identity element $\id$.
  All the witnesses that must be the same will end up in the same
  equivalence class.

  For the other direction of \eqref{k7}, assume $\AA\in\RA$, $\x,\y\in
  X$ and $\<i,j\>\in\hx(\x+\y)$, so $\<i,j\>\in\widehat{\f_\m}(\x+\y)$
  for some $\m\in\omega$. Again we have
  $\<\i,\j\>\in\widehat{\f_{2\mu}}(\x+\y)$ and
  $\tau_\mu=\<\i,\j,\x,\y\>$ for some $\mu\geq\m$.  By condition (a)
  in Case~1c,
  $\<i,j\>\in\widehat\f_{2\mu+1}(\x)\cup\widehat\f_{2\mu+1}(\y)$,
  hence $\<i,j\>\in\widehat\f_{2\mu+1}(\x)$ or
  $\<i,j\>\in\widehat\f_{2\mu+1}(\y)$.  By the definition of $\hx$,
  $\<\i,\j\>\in\hx(\x)$ or $\<\i,\j\>\in\hx(\y)$, yielding
  $\<\i,\j\>\in\hx(\x)\cup\hx(\y)$.

  This completes the construction of $\hx$ from the finite subset
  $\X\subseteq\A$ and the proof that \eqref{k1}--\eqref{k7} hold.

  Let $\F$ be the set of non-empty finite subsets of $A$. We will use
  $\F$ as the index set for an ultraproduct. Let $\U=\omega^\F$ be the
  set of functions that map $\F$ into $\omega$. The elements of the
  ultraproduct are equivalence classes of functions in $\U$.  For
  every finite subset $\X\in\F$ let $\II(\X)$ be the principal filter
  of finite subsets that contain $\X$.
  \begin{align*}
    \II(\X)=\{\Y:\X\subseteq\Y\in\F\}.
  \end{align*}
  The set of such filters, $\{\II(\X):\X\in\F\}$, has the finite
  intersection property (closure under finite intersections), for if
  $n\in\omega$ and $\{\X_\i:\i<\n\}\subseteq\F$ then
  $\bigcap\{\X_\i:\i<\n\}=\II(\bigcup_{\i<\n}\X_\i)$ and
  $\bigcup_{\i<\n}\X_\i\in\F$ since the union of finitely many finite
  sets is finite. A set of subsets of $\F$ with the finite
  intersection property is contained is a proper ultrafilter $\D$ on
  $\F$, so we have $\{\II(\X):\X\in\F\}\subseteq\D$.  For
  $\alpha,\beta\in\omega^\F$ and $\x\in\A$ let
  \begin{align*}
    \J(\alpha,\beta,\x)=\{\X:\X\in\F\land\<\alpha_\X,\beta_\X\>
    \in\hx(\x)\}.
  \end{align*}
  Conditions \eqref{k1}--\eqref{k7}, which hold for every $\X\in\F$,
  can be reformulated in terms of $\J$. The fact that \eqref{k1} holds
  for every $\X\in\F$ is expressed by (i$'$), \eqref{k2} is expressed
  by (ii$'$), \etc. Therefore, we have
  \begin{enumerate}
  \item[(i$'$)] $\J(\F\times\{0\},\F\times\{1\},\w)=\F$,
  \item[(ii$'$)] $\J(\F\times\{0\},\F\times\{1\},\v)=\emptyset$,
  \item[(iii$'$)] $\J(\alpha,\beta,0)=\emptyset$,
  \item[(iv$'$)]
    $\J(\alpha,\beta,\x\bp\y)=\J(\alpha,\beta,\x)\cap\J(\alpha,\beta,\y)$,
  \item[(v$'$)] $\J(\alpha,\beta,\con\x)=J(\beta,\alpha,\x)$,
  \item[(vi$'$)] for some $\gamma\in\U$,
    \begin{align*}
      \II(\{\x,\y\})\cap \J(\alpha,\beta,\x\rp\y)=\II(\{\x,\y\})\cap
      \J(\alpha,\gamma,\x)\cap\J(\gamma,\beta,\y),
    \end{align*}
  \item[(vii$'$)] if $\AA\in\RA$ then
    \begin{align*}
      \II(\{\x,\y\})\cap \J(\alpha,\beta,\x+\y)=\II(\{\x,\y\})\cap
      \(\J(\alpha,\beta,\x)\cup\J(\alpha,\beta,\y)\).
    \end{align*}
  \end{enumerate}
  Define a function $\rho\colon\A\to\pow{\U\times\U}$ by
  \begin{align*}
    \rho(\x)=\{\<\alpha,\beta\>:\alpha,\beta\in\U,\,
    \J(\alpha,\beta,\x)\in\D\}
  \end{align*}
  for every $\x\in\A$.  We will show
  \begin{enumerate}
  \item[(i$''$)] $\<\F\times\{0\},\F\times\{1\}\>\in \rho(\w)$,
  \item[(ii$''$)] $\<\F\times\{0\},\F\times\{1\}\>\notin \rho(\v)$,
  \item[(iii$''$)] $\rho(0)=0$,
  \item[(iv$''$)] $\rho(\x\bp\y)=\rho(\x)\cap \rho(\x)$,
  \item[(v$''$)] $\rho(\con\x)=\conv{\rho(\x)}$,
  \item[(vi$''$)] $\rho(\x\rp\y)=\rho(\x)|\rho(\y)$,
  \item[(vii$''$)] if $\AA\in\RA$ then $\rho(\x+\y)=\rho(\x)\cup\rho(\y)$.
  \end{enumerate}
  It is easy to check that each of these follows from its counterpart
  together with the fact that $\D$ is a proper ultrafilter and
  $\{\II(\X):\X\in\F\}\subseteq\D$.  For example, we get (iv$''$) from
  (iv$'$) because the filter $D$ is closed under intersection, (i$''$)
  from (i$'$) because $\F\in D$, and (ii$''$) from (ii$'$) because
  $\emptyset\notin D$. For (vi$''$), the second step below holds
  because $\D$ is closed under intersection and $\II(\{\x,\y\})\in\D$,
  and the third step comes from (vi$'$),
  \begin{align*}
    \rho(\x\rp\y)&=\{\<\alpha,\beta\>:\alpha,\beta\in\omega^\F,\,
    \J(\alpha,\beta,\x\rp\y)\in\D\}
    \\ &=\{\<\alpha,\beta\>:\alpha,\beta\in\omega^\F,\,
    \II(\{\x,\y\})\cap\J(\alpha,\beta,\x\rp\y)\in\D\}
    \\ &=\{\<\alpha,\beta\>:\alpha,\beta,\gamma\in\omega^\F,\,
    \II(\{\x,\y\})\cap\J(\alpha,\gamma,\x)\cap\J(\gamma,\beta,\y)\in\D\}
    \\ &=\dots=\rho(\x)|\rho(\y).
  \end{align*}
  To complete the proof of Lemma \ref{LEMMA6}, we need only note that
  $\rho(\v)\neq\rho(\w)$ by \eqref{k1} and \eqref{k2}, while the
  remaining desired properties coincide with (iii$''$)--(vii$''$).
\end{proof}
\begin{thm}
  \label{tabrra}
  If $\AA\in\JA\cup\RA$ is tabular then $\AA$ is representable.
\end{thm}
\begin{proof}
  By Lemma \ref{LEMMA6}, each pair $\v<\w$ gives rise to a set
  $\U^\w_\v$ and a function
  \begin{align*}
    \rho^\w_\v\colon\A\rightarrow\pow{\U^\w_\v\times\U^\w_\v}
  \end{align*}
  such that for all $\x,\y\in\A$
  \begin{align*}
    \rho^\w_\v(\v)&\neq\rho^\w_\v(\w), \\ \rho^\w_\v(0)&=\emptyset,
    \\ \rho^\w_\v(\x\bp\y)&=\rho^\w_\v(\x)\cap\rho^\w_\v(\y),
    \\ \rho^\w_\v(\x\rp\y)&=\rho^\w_\v(\x)|\rho^\w_\v(\y),
    \\ \rho^\w_\v(\con\x)&=\conv{\rho^\w_\v(\x)}, \\ \text{if
      $\AA\in\RA$ then }
    \rho^\w_\v(\x+\y)&=\rho^\w_\v(\x)\cup\rho^\w_\v(\y).
  \end{align*}
  Although Lemma \ref{LEMMA6} tells us that $\U^\w_\v=\omega^\F$ for
  every pair $\v<\w$, we will assume instead that the sets $\U^\w_\v$
  are pairwise disjoint. This can arranged by various set-theoretical
  devices or by an elaboration of the proof of Lemma \ref{LEMMA6}
  (\eg, replace $\omega^\F$ with $\omega^\F\times\{\<\v,\w\>\}$). Let
  \begin{align*}
    \UU&=\bigcup_{\v<\w}\U^\w_\v.
  \end{align*}
  This is a disjoint union because of our arrangement that if
  $\v,\w,\v',\w'\in\A$, $\v<\w$, $\v'<\w'$, and
  $\<\v,\w\>\neq\<\v',\w'\>$ then $\U^\w_\v\cap\U^{\w'}_{\v'} =
  \emptyset$.  For $\x\in\A$, let $\varphi(\x)$ be the union (another
  disjoint union) of $\rho^\w_\v(\x)$ over all pairs $\v<\w\in\A$,
  \begin{align*}
    \varphi(\x)=\bigcup_{\v<\w}\rho^\w_\v(\x).
  \end{align*}
  This gives us a function $\varphi\colon\A\to\pow{\UU^2}$.  The
  properties of $\rho^\w_\v$ can now be transferred to $\varphi$. To
  begin, $\varphi$ sends $0$ to the empty set, since
  \begin{align*}
    \varphi(0)&=\bigcup_{\v<\w}\rho^\w_\v(0)
    =\bigcup_{\v<\w}\emptyset=\emptyset.
  \end{align*}
  Converse is distributive over arbitrary unions, so we get
  \begin{align*}
    \varphi(\con\x) &=\bigcup_{\v<\w}\rho^\w_\v(\con\x)
    =\bigcup_{\v<\w}\conv{\rho^\w_\v(\x)}
    =\conv{\bigcup_{\v<\w}\rho^\w_\v(\x)}=\conv{\varphi(\x)}.
  \end{align*}
  Intersection and composition are also distributive over unions.  By
  the disjointness assumption, if $\v,\w,\v',\w'\in\A$, $\v<\w$,
  $\v'<\w'$, and $\<\v,\w\>\neq\<\v',\w'\>$ then $\emptyset =
  \rho^\w_\v(1)\cap\rho^{\w'}_{\v'}(1) =
  \rho^\w_\v(1)\rp\rho^{\w'}_{\v'}(1)$.  We have
  \begin{align*}
    \varphi(\x)|\varphi(\y) &=\(\bigcup_{\v<\w}\rho^\w_\v(\x)\)\Bigg|
    \(\bigcup_{\v'<\w'}\rho^{\w'}_{\v'}(\y)\)
    &&\text{def of $\varphi$}
    \\ &=\bigcup_{\v<\w,\;\v'<\w'}\rho^\w_\v(\x)|\rho^{\w'}_{\v'}(\y)
    &&\text{distributivity}
    \\ &=\bigcup_{\v<\w}\rho^\w_\v(\x)|\rho^\w_\v(\y)
    &&\text{disjointness assumption}
    \\ &=\bigcup_{\v<\w}\rho^\w_\v(\x\rp\y)
    &&\text{$\rho^\w_\v$ sends $\rp$ to $|$}
    \\ &=\varphi(\x\rp\y) &&\text{def of $\varphi$}
  \end{align*}
  and, simlarly, $\varphi(\x\bp\y)=\varphi(\x)\cap\varphi(\y)$.  A
  consequence of this that we will need for the injectivity of
  $\varphi$ is given by
  \begin{align*}
    \varphi(\x)\cap\rho^{\w}_{\v}(1)
    &=\(\bigcup_{\v'<\w'}\rho^{\w'}_{\v'}(\x)\)\cap\rho^{\w}_{\v}(1)
    &&\text{def of $\varphi$}
    \\ &=\bigcup_{\v'<\w'}\(\rho^{\w'}_{\v'}(\x)\cap\rho^{\w}_{\v}(1)\)
    &&\text{distributivity}
    \\ &=\rho^{\w}_{\v}(\x)\cap\rho^{\w}_{\v}(1) &&\text{disjointness
      assumption} \\ &=\rho^{\w}_{\v}(\x\bp1) &&\text{$\rho^\w_\v$
      sends $\bp$ to $\cap$} \\ &=\rho^{\w}_{\v}(\x)&&\eqref{one}
  \end{align*}
  To show $\varphi$ is injective, assume
  $\varphi(\x)=\varphi(\y)$. Note that for all pairs $\v<\w$,
  \begin{align}
    \label{note}
    \rho^\w_\v(\x)=\varphi(\x)\cap\rho^\w_\v(1)
    =\varphi(\y)\cap\rho^\w_\v(1)=\rho^\w_\v(\y)
  \end{align}
  To get a contradication, assume $\x\neq\y$ and let $\v=\x\bp\y$.
  Then either $\v<\x$ or $\v<\y$, for otherwise $\x=\v=\y$, but
  $\x\neq\y$. If $\x\bp\y=\v<\w=\x$ then we have a contradiction with
  \eqref{note}, namely,
  \begin{align*}
    \rho^\w_\v(\x)\cap\rho^\w_\v(\y)=
    \rho^\w_\v(\x\bp\y)=\rho^\w_\v(\v)\neq\rho^\w_\v(\w)
    =\rho^\w_\v(\x)=\rho^\w_\v(\y),
  \end{align*}
  and if $\x\bp\y=\v<\w=\y$ then we arrive at the same
  contradiction. Thus, $\varphi$ is injective.

  Finally, $\varphi$ can be altered to get a representation $\varphi'$
  that sends $\id$ to the identity relation on the
  $\varphi(\id)$-equivalence classes of elements.  Let
  $\E=\varphi(\id)$.  Then $\E$ is an equivalence relation on $\UU$.
  Let $\UU/\E=\{\u/\E:\u\in\UU\}$ where $\u/\E=\{\r:\<\r,\u\>\in\E\}$.
  Define a function $\h\colon\A\to\pow{\UU/\E\times\UU/\E}$ from $\A$
  to relations on $\UU/\E$ by
  \begin{equation*}
    \h(\x)=\{\<\r/\E,\s/\E\>:\<\r,\s\>\in\varphi(\x)\}
    \;\text{for}\;\x\in\A.
  \end{equation*}
  It is easy to check that $\h$ preserves $0$, $\bp$, $\rp$,
  $\con\blank$, and $+$ (if $\AA\in\RA$). One uses in an essential way
  the observation that $\E$ is a left and right identity for the
  representation of every element of $\AA$, that is,
  $\E|\varphi(\x)|E=\varphi(\x)$ for all $\x\in\A$.  Note that the
  equations $\varphi(\x)=\varphi(\y)$ and $\h(\x)=\h(\y)$ are
  logically equivalent, relative to the fact that $\E$ is an
  equivalence relation.  Consequently, $\h$ is also injective. The
  preservation of complementation in case $\AA\in\RA$ now follows from
  the fact that the complement $\min\x$ of an element $\x$ in a
  relation algebra is definable in terms of union and intersection by
  $\y=\min\x\iff\y+\x=1\land\y\bp\x=0$.
\end{proof}

\section{Conclusion}
\label{s28}
The primary open problem concerning Thompson's groups is to determine
whether $\F$ is amenable ($\T$ and $\VV$ are not). The group theory
community is evenly divided on whether $\F$ will turn out to be
amenable.  The possibility that $\F$ might not be amenable was a
motivating factor in Thompson's creation of his groups.  Thompson's
groups have been described and represented by tree diagrams (as in
Figures \ref{fog} and \ref{fog1}), parenthetical notation for actions
on trees (made precise in \SS\ref{s12a} and used in
Props.\ \ref{presT} and \ref{same}), other kinds of diagrams (forests,
strands, links, \etc), piecewise linear homeomorphisms of the unit
interval, real line, and unit circle, and many other topological and
diagrammatic methods that have yielded valuable insights and proofs of
many properties of the Thompson groups.  All these methods of group
representation suffer from the limitation that the only available
operation is composition of functions and permutations. Some papers
even need proofs to show that composition is possible.  The viewpoint
of this paper is that Thompson's groups and monoid are part of a much
larger structure (any finitely presented J-algebra with generators
satisfying \eqref{Q}, \eqref{D}, and \eqref{U}) in which functional
and permutational elements are not merely (and obviously) composable
via relative product but can also be intersected.

This viewpoint opens a new realm of diagrammatic representation,
exemplified by Figures \ref{fog2} and \ref{fog3}.  Many diagrams of
this kind were used in the production of this paper to visualize the
meanings of the rather complicated terms that appear in the proofs of
the relations in the presentations of the Thompson groups and
monoid. Diagrams like Figure \ref{fog2} were used to make clear the
meanings of the otherwise mysterious intersections $(\ot\A\id)\bp\A$
and $\A\bp\con\A$.  There is a natural correlation between terms in a
J-algebra and directed series-parallel graphs. An example appears in
Figure \ref{fog3}.  This link to graph theory and the new viewpoint on
Thompson groups (and other similarly defined groups) may conceivably
lead to the resolution of the amenability problem or other worthwhile
results.

The axioms for J-algebras are quite natural since they are the
simplest equations that hold in all relation algebras and involve only
the relative and Boolean products and the constants $0$, $1$, and
$\id$.  Some combinations of these axioms already do occur in many
papers and the entire list may have already been considered.  The
awkward axiom \eqref{mon} is sometimes avoided by postulating
(non-equationally) that $\leq$ is a partial ordering and relative
product is order-preserving. In any case, the variety of J-algebras
deserves independent study from a universal algebraic point of view.

The Boolean algebraic part of a relation algebra enables Tarski's
theorem that Q-algebras are representable.  The extension of
tabularity to J-algebras and the subsequent proof of representability
in Part III compensate for this lack of a Boolean part. It does not
seem likely that the existence of conjugated quasiprojections in a
J-algebra is sufficient for representability. Counterexamples or a
proof of representability should be found to settle this problem.  The
importance of the pairing identity portrayed in Part I is also
illustrated by its numerous applications in the derivations in Part
II. These derivations could be avoided if J-algebras are representable
whenever they have conjugated quasiprojections satisfying the Domain
and Unicity Conditions.

\def\url#1{{\tt #1}}
\providecommand{\bysame}{\leavevmode\hbox to3em{\hrulefill}\thinspace}
\providecommand{\MR}{\relax\ifhmode\unskip\space\fi MR }
\providecommand{\MRhref}[2]{%
  \href{http://www.ams.org/mathscinet-getitem?mr=#1}{#2}}
\providecommand{\href}[2]{#2}

\end{document}